\documentclass[11pt,reqno]{amsart}

\usepackage{amsrefs}
\usepackage{amsmath}
\usepackage{amsthm}
\usepackage{graphicx}
\usepackage{subcaption}
\usepackage{tikz}
\usetikzlibrary{trees,snakes,external,backgrounds,patterns}
\usepackage{amsfonts}
\usepackage{amssymb}
\usepackage{color}
\usepackage{bm}
\usepackage{hyperref} 
\usepackage{enumerate}
\usepackage[margin=1.1in]{geometry}  
\numberwithin{equation}{section}

\newtheorem{theorem}{Theorem}[section]
\newtheorem{lemma}[theorem]{Lemma}
\newtheorem{proposition}[theorem]{Proposition}
\newtheorem{corollary}[theorem]{Corollary}

\theoremstyle{definition}
\newtheorem{example}[theorem]{Example}\newtheorem{definition}[theorem]{Definition}
\newtheorem{remark}[theorem]{Remark}

\newtheorem{innercustomthm}{Assumption}
\newenvironment{assumption}[1]
  {\renewcommand\theinnercustomthm{#1}\innercustomthm}
  {\endinnercustomthm}

  \def\treex{{\mathcal T}}
  \def\Nvarx{\Nvar}
  \def\omegax{\Omega}
  \def\Fmcx{\Fmc}
  \def\Fmbx{\Fmb}
  \def\Embx{\Emb}
  \def\Pmbx{\PP} 
  \def\QQx{\QQ}
  \def\QQy{\QQ'}
  \def\Emby{\widetilde{\E}'}
  \def\Ly{{\mathcal L}'}
  \def\tLy{\widetilde{{\mathcal L}}'}
  \def\Ex{\E}
  \def\gammax{\gamma}
  \def\tgammax{\tilde{\gamma}}
  \def\omegay{\Omega'}
  \def\Fmcy{\Fmc'}
  \def\Fmby{\Fmb'}
  \def\Pmby{\PP'}
  \def\By{B'}
  \def\newy{X'}
  \def\newY{X'}
  \def\gammay{\gamma'}
  \def\tgammay{\tilde{\gamma}'} 
  \def\treey{{\mathcal T}'}
  \def\Nvary{\Nvar'}

  \def\perm{{\mathcal S}}
  \def\fmeas{\mu}
  \def\Pone{\mu}
  \def\Ptwo{\mu'}
   \def\tPonet{\widetilde{\mu}}
   \def\tPtwot{\widetilde{\mu}'} 
   \def\nuone{\widetilde{\nu}}
   \def\nutwo{\widetilde{\nu}'}
\def\Xo{X}
\def\Xt{X'}
\def\tX{X}

\def\E{{\mathbb E}}
\def\EE{{\mathcal E}}
\def\R{{\mathbb R}}

\def\N{{\mathbb N}}
\def\PP{{\mathbb P}}

\def\FF{{\mathbb F}}

\def\P{{\mathcal P}}

\def\B{{\mathcal B}}
\def\V{{\mathcal V}}

\def\QQ{\widetilde{\mathbb{P}}}
\def\K{{\mathcal K}}
\def\X{{\mathcal X}}

\def\V{{\mathbb V}}
\def\Y{{\mathcal Y}}

\def\L{{\mathcal L}}

\def\G{{\mathcal G}}

\def\T{{\mathbb T}}

\def\D{{\mathcal D}}
\def\W{{\mathcal W}}
\def\Z{{\mathbb Z}}
\def\A{{\mathcal A}}

\def\F{{\mathcal F}}
\def\C{{\mathcal C}}

\def\SQ{S^{\sqcup}}
\def\newb{\beta}
\def\newx{y}

\newcommand{\indep}{\perp \!\!\! \perp}

\newcommand{\Erdos}{Erd\H{o}s-R\'enyi }

\newcommand{\lan}{\langle}
\newcommand{\ran}{\rangle}

\newcommand{\Emb}{{\mathbb{E}}}
\newcommand{\Fmb}{{\mathbb{F}}}

\newcommand{\Pmb}{{\mathbb{P}}}

\newcommand{\Rmb}{{\mathbb{R}}}

\newcommand{\Tmb}{{\mathbb{T}}}

\newcommand{\Vmb}{{\mathbb{V}}}

\newcommand{\Fmc}{{\mathcal{F}}}

\newcommand{\Lmc}{{\mathcal{L}}}

\newcommand{\gmu}{\bar{\mu}}
\newcommand{\tree}{{\mathcal T}}

\newcommand{\parm}{\theta}

\newcommand{\bCa}{K}
\newcommand{\bCb}{\bar{K}}

\def\mom{\pi}
\def\detfn{\tau}
\def\mmap{\Lambda}

\def\Nvar{{\widehat C}_1}
\def\nufinal{Q}
\def\h{h}
\def\onepoint{{\circ}}

\title{Marginal dynamics of  interacting diffusions on unimodular Galton-Watson trees} 
	\date{today}
	 	\author[Lacker]{Daniel Lacker}
                \address{Columbia University, New York, New York} 
	 	\author[Ramanan]{Kavita Ramanan}
                 \thanks{K. Ramanan was supported in part by the National Science Foundation via  Grant DMS-1713032, 
             the Army Research Office via grant  W911NFF2010133 and a Simon Guggenheim Fellowship.}  
	 	\address{Division of Applied Mathematics, Brown University, 182 George Street, Providence, RI 02912} 
	 	\author[Wu]{Ruoyu Wu}
                 \address{Department of Mathematics, Iowa State University, 411 Morrill Road, Ames, IA 50011} 
	 	 \email{daniel.lacker@columbia.edu, kavita\_ramanan@brown.edu, ruoyu@iastate.edu}
                
\date{\today}
\begin{document}

\begin{abstract}  
  Consider a system of homogeneous interacting diffusive particles labeled by the nodes of
  a unimodular  Galton-Watson tree,
  where the state of each node evolves infinitesimally like a 
  $d$-dimensional diffusion whose drift coefficient  depends on (the histories
  of) its own state and the states of neighboring nodes,  
  and whose diffusion coefficient depends only on (the history of)  
  its own state.
     Under suitable   regularity assumptions on the coefficients, 
   an autonomous characterization is obtained for  the  marginal distribution of the dynamics of
   the neighborhood of a typical node in terms of a  certain local equation, which is
  a new kind of stochastic differential equation that  is nonlinear in the sense of McKean. This equation 
    describes 
   a   finite-dimensional  
  non-Markovian stochastic process 
   whose infinitesimal evolution at any time depends not only on the structure and current state 
   of the neighborhood, but  also  on   the conditional law 
   of the current  state given the past of the states of neighborhing nodes until that time.
   Such  marginal distributions  are 
  of interest because they  arise as  weak limits of both  marginal distributions and
  empirical measures of interacting diffusions on many sequences of   sparse random graphs, including the configuration
  model and 
  Erd\"{o}s-R\'{e}nyi graphs
  whose average degrees  converge  to  a finite non-zero limit.
   The results obtained complement classical results in the  mean-field regime, which 
    characterize 
  the limiting dynamics of homogeneous interacting diffusions on  complete graphs, as the number of nodes goes to infinity, in terms of   a corresponding nonlinear Markov process.
  However, in the sparse graph setting, the topology of the graph strongly influences the dynamics, and 
   the analysis  requires a completely different approach.
The proofs of existence and uniqueness of the local equation rely on delicate new conditional independence and symmetry properties of particle trajectories on unimodular Galton-Watson trees, as well as judicious use of  changes of measure. 
  
\end{abstract}

\maketitle
    
\tableofcontents

\section{Introduction}

\subsection{Background and Motivation}
\label{subs-backmot}
\footnote{This paper, along with \cite{LacRamWu19b,LacRamWu20a}, supersedes the earlier arXiv version \cite{LacRamWu19a}, after reorganizing and expanding upon several aspects of the material. Notably, this paper removes the assumption of bounded drift in the derivation of the local equation, whereas \cite{LacRamWu20a} sharpens and strengthens the results on local weak convergence of particle systems, and \cite{LacRamWu19b} elaborates further on related yet rather separate conditional independence properties. These three papers treat very different, complementary aspects of the same class of particle systems and may be read independently.}
Given a  (possibly random) simple, (almost surely) locally finite  rooted graph $G=(V,E)$,  consider 
interacting diffusions of the form 
\begin{equation}
  \label{eqn-genericmod}
  dX^{G}_v (t) =  b(X^{G}_v (t),  \mu^{G}_v (t)) dt  +  \sigma (X^{G}_v(t)) dW_v (t),   \quad v \in V, \quad  t \geq 0,  \end{equation}
with  initial condition
$x(0) \in (\R^d)^V$.  
 Here,   $(W_v)_{v \in V}$ are 
independent $d$-dimensional standard Brownian motions,  $b$ and $\sigma$
  are suitably regular drift and diffusion coefficients, and  $\mu^{G}_v(t)$ is the local (random) empirical measure of the states of the neighbors of $v$ at time $t \geq 0$: 
\[  \mu_v^{G}(t) = \frac{1}{|N_v(G)|} \sum_{u \in N_v(G)} \delta_{X^G_u(t)},
\] 
  with 
  $N_v(G) = \{u \in V: (u,v) \in E\}$  denoting the neighborhood of the vertex $v$ in the graph $G$.
(By convention, set $\mu^G_v(t)=\delta_0$ when $N_v(G)$ is empty, that is, when the vertex $v$ is isolated.)
  Large systems of interacting diffusions of the form \eqref{eqn-genericmod} arise as models in a range of applications
in  neuroscience,
physics,
and economics (see \cite{LacRamWu20a} for references). 
     Important quantities of interest include the  dynamics of  the state of a ``typical" vertex 
   and the (global) empirical measure process  defined by 
  \begin{equation}
    \label{gmuG1}
    \gmu^G (t) = \frac{1}{|V|} \sum_{v \in V} \delta_{X^G_v(t)},  \quad t \ge 0. 
  \end{equation}
  However,   these systems are typically too large and complex 
   to be analytically or numerically tractable. 
 Therefore,  it is natural to seek approximations that are provably accurate
in a suitable  asymptotic regime.

Classical works of 
  McKean, Vlasov and others
  (see \cite{Mck67,sznitman1991topics,Kol10,kurtz1999particle} and references therein)
  focused on such particle systems when  $G = K_n$,  the complete graph on $n$ vertices. 
They showed that,  under suitable conditions, the limit of $X_{\o}^{K_n}$ in  \eqref{eqn-genericmod}, where
$\o$ is a randomly chosen root vertex, 
      is described by the following  nonlinear Markov process:
      \begin{equation}
        \label{eq-nonlinear}
        d X (t)  =  b( X(t), \gmu(t)) dt + \sigma (X(t)) dW(t), \quad \gmu(t) = 
        {\mathcal L} (X(t)), \qquad t \geq 0, 
      \end{equation}
      where  $\gmu(t)$ is the (deterministic) weak limit, as $n \rightarrow \infty$,
      of the global empirical measure  $\gmu^{K_n}(t)$, and ${\mathcal L} (Z)$ denotes the law of a random variable $Z$.  
        The measure-valued function  $\gmu(\cdot)$ can also be characterized as the unique
      solution to a nonlinear partial differential equation (namely, the forward
      Kolmogorov equation associated with this process), whence the name {\em nonlinear} 
      Markov process.    
      The key property that  leads to such a  characterization is  the observation that
      particles interact only \emph{weakly},
      with the influence of any single 
      particle on any other particle being of order $1/n$. 
      This leads to asymptotic independence of any finite collection of particles 
      and  convergence of the random (global)  empirical measure process of the finite particle systems
      to a deterministic limit (see \cite{sznitman1991topics,Mck67} for further discussion
      of this phenomenon, known as propagation of chaos).
 An alternative two-step perspective to mean-field limits,
      taken in \cite{KotKur10}, is to first use exchangeability
      to  show that 
      $X^{K_n}$ converges to a limit, which is the unique
      solution of a countably infinite coupled system of diffusions,
      and then show that the marginal of any vertex in this infinite coupled system of diffusions
      can be  autonomously described as the nonlinear Markov process  in \eqref{eq-nonlinear}.

      Given the above intuition, it is natural to expect that asymptotic independence  and the same mean-field characterization \eqref{eq-nonlinear}  for  the limiting dynamics of a typical node may continue to hold for suitably  ``dense'' graph sequences $\{G_n\}_{n \in \N}$, where each graph is not necessarily complete, but the (minimum or average) degree of
      the graphs grows to infinity.  
      Indeed,  
      several recent works \cite{DelGiaLuc16,BhaBudWu18,CopDieGia18,OliRei18,Luc18quenched,Med18} have shown that 
      either asymptotic independence or a mean-field characterization like \eqref{eq-nonlinear} continues 
        to hold  under different sets of   assumptions on the precise nature of denseness of the graph sequence.  
As in the complete graph case, these works exploit the fact that  
       the local interaction strength at a vertex
      is inversely proportional to the degree at that vertex, and thus  vanishes in dense regimes, 
      although the proofs are more involved than in the complete graph case due to a lack of full exchangeability.

      In contrast, 
       very few works  have studied  the limiting behavior of $\bar{\mu}^{G_n}$ or $X_{\o}^{G_n}$ 
      in the complementary sparse graph regime, that is,  when the average degrees of 
      possibly random graphs in the sequence $\{G_n\}_{n \in \N}$ remain uniformly bounded as the size of the graph goes to infinity. 
     In this regime, 
     neighboring particles 
     are strongly interacting and do not become asymptotically independent as the graph size goes to infinity, 
     and so  the limiting dynamics of any finite set of  particles is no longer described in terms of 
     the mean-field limit. 
      In Theorem 3.3 of a companion paper \cite{LacRamWu20a} (which extends results of a previous version \cite{LacRamWu19a}), 
      we consider  a more general  class of  (possibly non-Markovian) dynamics
      than \eqref{eqn-genericmod}, 
    and  show under broad assumptions that if  the  sequence
      $\{(G_n, X^{G_n}(0))\}_{n \in \N}$ of (possibly random)
      rooted graphs and their initial conditions  converges  in distribution
      (in the sense of local convergence of marked graphs)
      to a limit (random) graph $G$, 
      then $\{(G_n, X^{G_n})\}_{n \in \N}$ also converges in distribution
      (again in the sense of  local convergence of marked graphs) to $(G,X^G)$.   
     This, in particular, implies that  the marginal dynamics at the root
     $X^{G_n}_{{\o}}$ converges in law to the corresponding marginal dynamics  $X_{{\o}}^G$
     on the limit  graph $G$, where $\o$ denotes an appropriate root vertex in $G$.

     In many cases of interest, the limit graph $G$ is a 
     so-called unimodular Galton-Watson (UGW) tree
      (see Definition \ref{def-UGW}).   
This is the case,   for example, when 
     $G_n$ is the  \Erdos graph on $n$ vertices with parameter $p_n \in (0,1)$,
     with  $n p_n \rightarrow \parm \in (0,\infty)$, or the graphs 
      $G_n$ are sampled from a  configuration model with a converging  empirical 
     degree sequence with finite non-zero first moment,
     and the root is chosen uniformly
     at random (see, e.g., \cite{dembo-montanari,van2009randomII} or Section 2.2.3 of \cite{LacRamWu20a}). 
     Under the  assumption of local convergence in probability of the  graph sequence $\{G_n\}_{n \in\N}$
       (which is stronger than the local convergence in distribution imposed in the previous paragraph,
       but is nevertheless still satisfied by the examples mentioned above) and 
     suitable assumptions on the initial conditions
     $(X_v^{G_n}(0))_{v \in V_n}$ that are satisfied, for example, when they are independent and identically distributed (i.i.d.) with a distribution independent of $n$ (or, more generally, distributed according to a
     Gibbs measure with a pairwise interaction potential independent of $n$),   
   it is shown in Theorem 3.7 of \cite{LacRamWu20a} that
     the global empirical measure process 
     $\bar{\mu}^{G_n}$ also converges in distribution to the law of  $X_{\o}^{G}$, with
     the same i.i.d.\  initial conditions (respectively, Gibbs measure with the same interaction potential).

     The  only other work that considers  asymptotic limits in the sparse graph regime is
     \cite{OliReiSto19}, which considers a slightly  different   model of Markovian
     interacting diffusions with identity diffusion
     coefficient, weighted 
     pairwise interactions,   an i.i.d.\ random environment and i.i.d.\ initial
     conditions,      when the largest vertex degree in  $G_n$ is 
     additionally assumed to be of order $|G_n|^{o(1)}$.   They  prove a 
     local convergence result for the  interacting processes,  but only state, without proof, an  
     empirical measure convergence result.    However, as  shown  in \cite{LacRamWu20a} (see Theorems 3.9 and 6.4 therein),  
     the empirical measure convergence to a deterministic limit  can fail under local
     convergence in distribution, rather than in probability, of the graph sequences, 
     thus demonstrating that
      the proof of empirical measure convergence is more subtle in
      the sparse graph regime 
      than in the complete or dense graph regimes.  
     
     \subsection{Our Contributions}

     \subsubsection{Discussion of our results}

     Both works \cite{LacRamWu20a} and \cite{OliReiSto19} can be viewed as implementing, for sparse
     graph sequences,  the
     first step of the two-step approach of \cite{KotKur10} mentioned above for complete graphs, namely
     showing that the limit of $\{X^{G_n}\}_{n \in \N}$ exists and can be characterized as the unique
     solution to a countably infinite
     coupled system of SDEs. 
     However, both these works leave open the important question of providing an
     autonomous characterization of the  marginal dynamics of this infinite system of SDEs.
      The main contribution of this article is a resolution of 
      this issue in the case when $G = {\mathcal T}$ is a UGW  tree.   
     Specifically, we consider the interacting particle  
     system on ${\mathcal T}$ defined in \eqref{eqn-genericmod} (or rather
     a possibly non-Markovian generalization of the dynamics described in Section \ref{sec-main}), 
     and  show that the marginal dynamics of the root particle and its neighbors 
     can be characterized by an autonomous system of equations that we call the \emph{local equation}.
      The choice of a UGW tree is in some 
      sense canonical in view of 
      the results in \cite{LacRamWu20a,LacRamWu19a,OliReiSto19} mentioned above that
      show that the
      law of  the root particle dynamics on a UGW tree ${\mathcal T}$
      arises as the limit of both marginal dynamics and the
      global empirical measure processes of diffusive particle systems
      on many sequences of sparse random graphs of growing size.  
     Thus,       the  local equation can be viewed as the analogue, in the sparse graph setting,
      of the equation \eqref{eq-nonlinear} that characterizes the
      nonlinear Markov process describing 
      the limiting evolution of a typical particle for suitably dense graph sequences, 
       although in this work we
         do not work with the most general  initial conditions.
      
        To the best of our knowledge, prior to this work, there did  not exist even a
conjecture regarding  the form of  the limiting  marginal dynamics  of a typical   particle or global empirical measure process in the sparse
graph regime.   In this regime the graph structure clearly plays a role, and        
new ideas are required.
In the particular case when ${\mathcal T}$ is  
      the (deterministic) $\kappa$-regular tree $\mathbb{T}_\kappa$ 
      the local equation describes a new kind of stochastic differential equation that characterizes 
        an $(\R^d)^{1+\kappa}$-valued process  
      whose infinitesimal evolution at any time $t$ depends not only on its current state  at time $t$ but
      also on the conditional law of the current  state  given the {histories} of the states of part of the neighborhood up to
      time $t$ (see Definition \ref{def-locregtree}).   Notably, even when the original interacting process on $\mathbb{T}_\kappa$ is Markovian, the solution to the
        local equation is  a non-Markovian
        process   and it is also nonlinear in the sense that at any time, 
        its evolution  depends on the law of the process up to that time, although  in a non-standard
        way via a  conditional distribution associated with the law.   
      To provide insight into the form  of the local equation,
      in Section \ref{subsub-linegraph} below we first derive  the local equation for a particle system
      in the simplest case when the UGW tree is $\mathbb{T}_2$, or equivalently, $\mathbb{Z}$.
Then, in Section \ref{subsub-gencase} we discuss the
significant additional 
complications that arise  in the general case of a random UGW tree ${\mathcal T}$.

The only other result that we are aware of that provides an autonomous
characterization of marginals of
an infinite system of  interacting diffusions on a sparse graph
was obtained recently in  \cite{DetFouIch18},
which treats a Markovian interacting diffusive particle system  with identity diffusion coefficient
in the special case where  the interaction graph is a directed line, without any feedback of the interactions.   
In this specific setting, a coupling argument is used to obtain an autonomous characterization of the law of the trajectories of any contiguous set of particles in terms of a non-linear  diffusion process.  As we show in Section \ref{subsub-linegraph}, even on a line  such an autonomous characterization 
is  more complicated when the graph is no longer directed.

\subsubsection{The local equation for a particle system  on the line graph}	\label{subsub-linegraph}
Consider the particular diffusive particle system
\begin{align}
  dX_v(t) = \big[ \newb (X_v(t),X_{v+1}(t)) + \newb(X_v(t), X_{v-1}(t))\big] dt
  + \sigma(X_v(t))dW_v(t), \quad v \in \Z, \label{def:T2-particlesystem}
\end{align}
where $(X_v(0))_{v \in \Z}$ are i.i.d., and $\newb : (\R^d)^2 \to \R^d$ and $\sigma : \R^d \to \R^{d \times d}$ are 
    assumed to be sufficiently regular (see Assumption \ref{assumption:A}). 
Note that this is a particular case of the dynamics  \eqref{eqn-genericmod} when $G$ is equal to
  the (deterministic) $2$-regular tree $\mathbb{T}_2$, which can be identified 
  with $\mathbb{Z}$,  and the drift $b$ is linear
    in the measure variable: $b(x,\nu) =  2 \int_{\R^d} \newb (x, y) \, \nu (dy)$ for $x \in \R^d$ and
    $\nu$ a probability measure on $\R^d$. 
    We will also assume that the dependence of $\newb$ on the second variable is non-trivial so that
    we have a system of diffusions that  are truly interacting.  
  For any $t > 0$, let
  \[ X_v[t] := (X_v(s))_{s \in [0,t]} \]
  represent  the trajectory of
$X_v$ in the interval $[0,t]$, and for any subset $A \subset \mathbb{Z}$, let
$X_A[t] := (X_v[t])_{v \in A}$.
Identifying the root node  with $0 \in \mathbb{Z}$,  
we would like to understand the law of the  dynamics of the root marginal $X_0$, but
it turns out that  to obtain an autonomous description, one should instead consider the 
marginal dynamics $X_{\{-1,0,1\}} = (X_{-1}, X_0, X_1)$ of the root {\em and its neighborhood}, rather than just the root.  The characterization via the local equation entails three key ingredients.

\vskip.3cm

\noindent 
{\em (i) Markov random field structure. } 
First,  note that the dynamics of $X_0$ is completely endogenous in that it only depends on
the states of $X_{-1}$ and $X_1$, which are part of the neighborhood.  On the other hand,  the evolution of $X_{-1}$ depends on $X_{-2}$, the state 
of node $-2$, which lies outside  the set $\{-1,0,1\}$.  Therefore, in order to get an autonomous
description of the law of the dynamics of $X_{\{-1,0,1\}}$, we need to be able to express  the conditional law of $X_{-2}(t)$
given $X_{\{-1,0,1\}}$
in terms of the (joint) law of $X_{\{-1,0,1\}}$. 
As a key first step towards achieving this goal, we establish the following 
conditional independence property of the particle system \eqref{def:T2-particlesystem}: for each $t > 0$, 
\begin{align}
(X_j[t])_{j < i} \indep	 (X_j[t])_{j > i + 1} \  | \ (X_{i}[t],X_{i+1}[t]), \ \ \forall i \in \Z \label{def:T2-condind}
\end{align}
where for random elements $Z_1, Z_2, Z_3$, 
 we use $Z_1 \indep Z_2|Z_3$ to denote that  $Z_1$ and $Z_2$ are conditionally independent given $Z_3$. 
In other words, we show that for each $t > 0$, $(X_i[t])_{i \in \Z}$ is a second-order Markov chain on $\Z$.

At first glance, one might conjecture that $(X_j(t))_{j < i} \indep (X_j(t))_{j > i} \,|\, X_i(t)$ for each $i \in \Z$,
  that is, for every fixed $t > 0$, the {\em states} $(X_i(t))_{i \in \Z}$ form a \emph{first-order} Markov chain (on $\Z$).   
  However, this conjecture is not valid, because  conditioning on $X_i(t)$ clearly  provides information
  on the past $X_i[t]$ of $X_i$ and, in turn, for $s \in [0,t]$, the  state  $X_i(s)$ 
   directly influences the values of $X_{i-1}(s)$ and $X_{i+1}(s)$ and hence, of $X_{i-1}(t)$ and $X_{i+1}(t)$.  In other words, conditioning
on $X_i(t)$ correlates $X_{i-1}(t)$ and $X_{i+1}(t)$ via  the information it provides on the past of $X_i$. 
  This observation may then prompt the modified conjecture  that 
  \[ (X_j[t])_{j < i} \indep (X_j[t])_{j > i} \,|\, X_i[t], \qquad i \in \Z;  \]
  that is, for every fixed $t > 0$, the collection of \emph{trajectories} up  to time $t$, $(X_i[t])_{i \in \Z}$,
  is a  first-order 
  Markov chain.  In  particular,  one may naively expect that conditioned on $X_i[t] = \psi$,
  $X_{i-1}[t]$ and $X_{i+1}[t]$ become decoupled and satisfy the following SDE: for  $s \in [0,t]$, 
  \begin{eqnarray*}
    dX_{i-1}(s) & = & [\newb(X_{i-1}(s), \psi(s)) + \newb(X_{i-1}(s), X_{i-2}(s))] ds + \sigma (X_{i-1}(s)) dW_{i-1}(s),  \\
    dX_{i+1} (s) & =  &  [\newb(X_{i+1}(s), X_{i+2}(s)) + \newb(X_{i+1}(t), \psi(s))] ds + \sigma (X_{i+1}(s)) dW_{i+1}(s),
  \end{eqnarray*}
  where $W_{i-1}$ and $W_{i+1}$  are independent Brownian motions.  However, a more careful inspection would reveal 
  that such a reasoning is  spurious because  
   the evolution of $X_i$, and thus the random element $X_i[t]$,  directly depends on
   the values $(X_{i-1}(s), X_{i+1}(s))_{s \in  [0,t]}$,  which are in  turn  driven by the
   Brownian motions  $W_{i-1}$ and  $W_{i+1}$.  Thus, 
   conditioning on $X_i[t] = \psi$ causes   $W_{i-1}$ and  $W_{i+1}$ 
   to become correlated, showing that  $X_{i-1}$ and $X_{i+1}$ do not follow the above SDE
   and are also not   independent under this conditioning.
Thus,  the modified conjecture is also not valid. 
 Instead, as stated in \eqref{def:T2-condind}, we show that by
  conditioning on both  $X_{i}[t]$ and $X_{i+1}[t]$, the driving noise processes $W_{i-1}$ and $W_{i+2}$   remain decoupled.
 While this is not a trivial observation,  some intuition may be gleaned by noting that 
 when one conditions on both  $X_i[t]$ and $X_{i+1}[t]$,  the trajectories of $W_{i}$  and $W_{i+1}$ become
 irrelevant, and so the correlations  induced betweeen $W_{i+1}$ and $W_{i-1}$ when
 conditioning just on  $X_i[t]$, and likewise, the correlations induced betweeen $W_{i}$ and $W_{i+2}$ when
 conditioning just on $X_{i+1}[t]$, are no longer relevant.   In other words, when conditioning on
 $X_i[t]$ and $X_{i+1}[t]$,   the evolution on $[0,t]$ of 
$(X_{j})_{j < i},$ is only influenced by $X_{\{i,i+1\}}[t]$ and the independent  driving noises $(W_{j})_{j < i}$,
whereas  the evolution $(X_j)_{j > i+1},$  is only influenced by
 $X_{\{i,i+1\}}[t]$ and the independent driving noises $(W_{j})_{j >  i+1}$. 
 In particular, conditioning on both $X_{i}[t]$ and $X_{i+1}[t]$ does not alter 
 the independence of the driving noises  $W_{i-1}$ and $W_{i+2}$, although it does alter their
 distribution; they are no longer Brownian motions or even martingales.

In fact, Theorem 2.7 of \cite{LacRamWu19b} shows that  for any locally finite
    graph $G$ and $X^G$ as in \eqref{eqn-genericmod}, for every $t > 0$, 
    the trajectories  $(X^G_v[t])_{v \in V}$ form a \emph{local second-order Markov random field (MRF)}
    (assuming the initial conditions do),  in the sense that
  \begin{align}
  X_A[t] \indep X_B[t] \,|\, X_{\partial^2A}[t], \quad \forall A \subset V \ \text{ finite}, \ B \subset V \setminus (A \cup \partial^2A), \label{gen-condind}
  \end{align}
where $\partial^2A$ is the set of nodes at distance one or two from $A$ (see Section \ref{subs-graphs} for graph-theoretic terminology and Section \ref{se:2MRFs} for a discussion of MRFs). 
  We would like to emphasize, however, that the conditional independence property  \eqref{def:T2-condind} required here is not implied by 
  the local MRF property established in \cite{LacRamWu19b}. Indeed, to obtain the first conditional
   independence statement in 
   \eqref{def:T2-condind}  one would need to apply \eqref{gen-condind} with 
   $A = \{j \in \N : j < i \}$ and $B = \{j \in \N : j > i + 1\}$ in \eqref{gen-condind}.
   In particular, we need \eqref{gen-condind} to also hold for certain infinite  sets $A$.  
  This is analogous to the distinction between (tree-indexed) first-order Markov chains versus first-order local MRFs; the latter often form a proper subset of the former, as explained in \cite[Chapters 10--12]{georgii2011gibbs}. 
   More generally, 
   an extension from a \emph{local} MRF property to a \emph{global} MRF property, in which $A \subset V$ in \eqref{gen-condind} is allowed to be infinite, is highly non-trivial and can fail in general;
   see   \cite{israel1986some,Kes85,Wei80} for works in other contexts that illustrate the underlying subtleties.
   Nevertheless, we show that the \emph{global} MRF property  does hold in our setting; see Propositions
     \ref{pr:properties-GW}  and \ref{pr:invariance-GW} for a proof in the more      
     general context of random UGW trees.
     For further intuition into this second-order MRF property and explicit examples
       that illustrate why the first-order counterparts fail,  we refer the reader to
          Section 3.3 of \cite{LacRamWu19b}.     

\vskip.3cm
 
\noindent 
{\em (ii) A projection theorem and symmetry considerations. }
We now discuss the second ingredient of the proof,  
recalling that we are interested in an autonomous characterization of $X_{\{-1,0,1\}}$, where $(X_v)_{v \in \Z}$ are as in \eqref{def:T2-particlesystem}.
Using an optional projection argument known from filtering theory (see Appendix \ref{sec-mim}),  
we can conclude that (extending the probability space if necessary) there exist  independent Brownian motions
$(\widetilde{W}_{-1},\widetilde{W}_0,\widetilde{W}_1)$ such that 
 $\tX=(\tX_{-1},\tX_0,\tX_1)$  satisfies  
\[
\tX_v (t) = \tilde{b}_v (t, \tX) dt + \sigma ( \tX_v(t) ) d\widetilde{W}_v (t),  \quad v \in \{-1,0,1\}, 
\]
where, with $\C$ denoting the space of $\R^d$-valued continuous functions on $[0,\infty)$, $\tilde{b}_v:[0,\infty) \times \C^{\{-1,0,1\}} \mapsto \R^d$ is a progressively measurable version of the conditional expectation: 
\[
\tilde{b}_v (t, x) := \E\Big[ \newb( X_v(t),  X_{v+1}(t)) + \newb(X_{v}(t),X_{v-1}(t)) \, \Big| \, X_{\{-1,0,1\}}[t] = x[t] \Big], \quad v \in \{-1,0,1\}, 
\]
where we recall $x[t] = (x(s))_{s\in [0,t]}$.   
Clearly, the drift coefficient for the root or zero particle remains the same as in the original system described in
 \eqref{def:T2-particlesystem}: 
\[ \tilde{b}_0 (t,x) = \newb(x_0(t), x_{1}(t)) +\newb(x_0(t),x_{-1}(t)). \]    
On the other hand, 
$\tilde{b}_1$ and $\tilde{b}_{-1}$ do  not coincide with the corresponding drifts in the original system, but 
we can simplify the expressions for them  using the conditional independence relation of \eqref{def:T2-condind} along with  symmetries of the particle system.    
Precisely,  as justified below, we have
\begin{align*}
\tilde{b}_{-1}(t,x) &= \E\big[\newb(X_{-1}(t),X_{0}(t)) +  \newb( X_{-1}(t),  X_{-2}(t)) \, | \, X_{\{-1,0,1\}}[t] = x[t] \big] \\
   &=  \newb(x_{-1}(t),x_{0}(t)) + \E\big[ \newb( X_{-1}(t),  X_{-2}(t)) \, | \, (X_{-1},X_0)[t] = (x_{-1},x_0)[t]  \big]  \\
  &= \newb(x_{-1}(t),x_{0}(t)) + \E\big[ \newb( X_{0}(t),  X_{-1}(t)) \, | \, (X_{0},X_{1})[t] = (x_{-1},x_0)[t]  \big]. 
\end{align*}
Indeed, the crucial steps are the second line, which follows from  the conditional independence
of $X_{-2}[t]$ and $X_{1}[t]$ given $X_{\{-1,0\}}[t]$, and the third line, which follows from
the shift-invariance of  the particle system on $\mathbb{Z}$, which gives equality in law
of $(X_{-2}, X_{-1}, X_0)$ and $(X_{-1}, X_{0}, X_1)$.
We can derive an analogous expression for  $\tilde{b}_{1}(t,x)$ by using the conditional independence of $X_{2}[t]$ and $X_{-1}[t]$ given $X_{\{0,1\}}[t]$, and  the equality in law between $(X_{2},X_{1},X_0)$ and $(X_{-1},X_0,X_{1})$ which now follows from both the shift-invariance and reflection-invariance (around $0 \in \Z$) of $X$: 
\begin{align*}
  \tilde{b}_1(t,x) &= \E\big[ \newb( X_{1}(t),  X_{2}(t)) \,| \, (X_{1},X_{0})[t] = (x_{1},x_0)[t]  \big] + \newb(x_1(t),x_0(t)) \\
  &= \E\big[ \newb( X_{0}(t),  X_{-1}(t)) \, | \, (X_{0},X_{1})[t] = (x_{1},x_0)[t]  \big] + \newb(x_1(t),x_0(t)),  
\end{align*}
for $t > 0$ and $x \in \C^{\{-1,0,1\}}$. 
In summary, if we define
\begin{align}
\tilde\gamma_t(x,y) :=\E\big[ \newb(X_0(t),X_{-1}(t)) \ | \ (X_0,X_{1})[t] = (x,y)[t] \big], \quad (x,y) \in \C^2,  \label{def:intro:gammatil}
\end{align}
then we find that $\tX = \tX_{\{-1,0,1\}}$ solves the coupled system 
\begin{align}
  d\tX_{-1}(t) &= \big[ \newb(\tX_{-1}(t), \tX_0(t)) + \tilde\gamma_t(\tX_{-1},\tX_0)\big] dt + \sigma(\tX_{-1}(t))d\widetilde{W}_{-1}(t),  \nonumber \\
d\tX_{0}(t) &=  \big[ \newb(\tX_0(t), \tX_{1}(t))dt + \newb(\tX_0(t), \tX_{-1}(t)) \big] dt +
\sigma(\tX_{0}(t))d\widetilde{W}_{0}(t), \label{def:intro:T2localeq} \\
d\tX_{1}(t) &= \big[ \tilde\gamma_t(\tX_{1},\tX_0) +  \newb(\tX_1(t), \tX_{0}(t))\big] dt + \sigma(\tX_{1}(t))d\widetilde{W}_{1}(t),  \nonumber
\end{align} 
where $\widetilde{W}_{-1}, \widetilde{W}_0$ and $\widetilde{W}_1$ are  independent $d$-dimensional Brownian motions.
 Modulo some additional technical conditions,  
this is precisely the $\Tmb_2$ local equation 
associated with the particle system
\eqref{def:T2-particlesystem}; see Definition \ref{def-locregtree} with $\kappa = 2$. 
Observe that  even though  the original system  \eqref{def:T2-particlesystem} describes  a (linear) Markov process,
  its marginal $X_{\{-1,0,1\}}$, as described by the system \eqref{def:intro:T2localeq}, is a nonlinear, non-Markovian process 
  since $\tilde\gamma_t$ is a functional of the  law of $\tX_{\{-1,0,1\}}[t]$ of the process and it takes as arguments
 the past of coordinates of the process (up to time $t$).   
However,  also  note that 
this  dependence  ensures that the  coupled system  \eqref{def:intro:T2localeq}  is  autonomously defined.  
 
\vskip.3cm

{
\noindent 
{\em (iii) Proofs of well-posedness. }
The final step of  the proof is to show 
that the law of $X_{\{-1,0,1\}}$ is the unique (weak) solution to the local equation \eqref{def:intro:T2localeq}.
 Banach fixed point arguments, which are  commonly used in the analysis of more standard nonlinear Markov  processes that arise as mean-field limits,  are rendered  unsuitable by   
 the complicated appearance of  conditional laws in the local equation.
 Coupling methods, which constitute another tool to establish uniqueness of mean-field limits,
 are also hard to implement due to the lack of regularity of the conditional expectation
 functional $\tilde{\gamma}_t$ defined in \eqref{def:intro:gammatil}.   

 We develop two alternative approaches to establishing uniqueness.  In the case of bounded drift, we give a direct argument for uniqueness (on $\mathbb{T}_\kappa$ for any $\kappa \geq 2$) using relative entropy estimates in 
Section \ref{subs-altuniq1}, 
 which we sketch here in the case $\kappa = 2$ and $\sigma$ is the identity matrix.  
  We start with the useful observation  that  
  any solution $X=(X_{-1},X_{0},X_1)$ to the  local equation \eqref{def:intro:T2localeq} satisfies
  the following symmetry properties: 
  \begin{equation}
    \label{loceq-symmetry}
    (X_{-1},X_0,X_1)\stackrel{d}{=}(X_1,X_0,X_{-1}), \qquad (X_1,X_0)\stackrel{d}{=}(X_0,X_1).    
  \end{equation}
  This follows from  Lemma \ref{lem-locsym}, which  identifies symmetries in the more general setting of
    a $\kappa$-regular tree, $\kappa \geq 2$.
       Next, let $\Xo=(\Xo_{-1},\Xo_0,\Xo_1)$ and $\newY=(\newY_{-1},\newY_0,\newY_1)$ be two solutions to
\eqref{def:intro:T2localeq}, 
and let 
the associated conditional expectation functionals, as in \eqref{def:intro:gammatil}, be denoted by $\tgammax_t$ and $\tgammay_t$.
Then the difference in the drift coefficients of the SDE \eqref{def:intro:T2localeq} for $\Xo$ and $\newY$ will be governed by 
$\delta\tilde{\gamma}_t:=\tgammax_t-\tgammay_t$.  
Next, recall that $\L (Z)$ denotes the law of a random element $Z$, and let  $H$ denote the relative entropy
  functional:   for  probability measures $\nu, \tilde{\nu}$ on a common measurable space,   let
  \begin{equation}
    \label{def-relentropy}
H(\nu|\tilde{\nu}) := \int \log\tfrac{d\nu}{d \tilde{\nu}}d\nu \ \ \  \text{ if } \nu \ll \tilde{\nu}, \qquad H(\nu|\tilde{\nu})=\infty \ \ \ \text{ if } \nu \not\ll \tilde{\nu}, 
  \end{equation}
  where $\nu \ll \tilde{\nu}$ signifies $\nu$ is absolutely continuous with respect  to  $\tilde{\nu}$. 
  Then   the boundedness assumption  on the drift $b$ (which is inherited by the progressively measurable functionals $\tgammax$ and $\tgammay$, and thus
  $\delta \tilde{\gamma}$),  along with a standard calculation involving Girsanov's theorem (see Corollary \ref{co:entropyestimate}),  
  yields the relative entropy identity
\begin{align*}
H\big( \L(X[T]) \,|\,  \L(\newY[T])\big) &= \frac12 \E\left[\int_0^T\left(|\delta\tilde{\gamma}_t(X_{-1},X_0) |^2 + |\delta\tilde{\gamma}_t(X_1,X_0) |^2 \right)dt \right] \\
	&= \E\left[ \int_0^T |\delta\tilde{\gamma}_t(X_0,X_1) |^2 \,dt \right]. 
\end{align*}
Now, for $x, \newx \in \C$ and $t > 0$, let $\Pone_{x,\newx}[t]$ denote the conditional law of $X_{-1}[t]$ given $(X_0[t],X_1[t])=(x[t],\newx[t])$, and likewise, let $\Ptwo_{x,\newx}[t]$ denote the conditional law of $\newY_{-1}[t]$ given $(\newY_0[t],\newY_1[t])=(x[t],\newx[t])$. 
For $t > 0$ and $x, y \in \C$, 
set $\newb_{t,x}(\newx):=\newb(x(t), y(t))$.  Then,  letting
\[ C := \sup_{z,z' \in \R^d}|\newb(z,z')|,\]
which is finite by assumption, we see that 
\[ \delta\tilde{\gamma}_t(x,\newx)= \int_{\C}  \newb_{t,x} (z) ( \Pone_{x,\newx}[t] - \Ptwo_{x,\newx}[t])(dz) \leq
C  d_{{\rm TV}} ( \Pone_{x,\newx}[t],  \Ptwo_{x,\newx}[t]),
\]
where $d_{{\rm TV}}$ denotes the total variation distance. 
The last two displays, when combined with Pinsker's inequality (see, e.g., \cite[p. 44]{CsiKor11})  and the
chain rule for relative entropy, yield 
\begin{align*}
  H\big(\L(\Xo[T]) \,|\, \L(\Xt[T])\big) &\le 2C^2 \E\left[ \int_0^T H\big(\Pone_{\Xo_0,\Xo_1}[t]\,|\,\Ptwo_{\Xo_0,\Xo_1}[t]\big) \,dt \right]\\
  & \le   2C^2\int_0^T H\big(\L(\Xo[t]) \,|\, \L(\Xt[t])\big)\,dt.
\end{align*}
An application of Gronwall's inequality then shows that  $\L(\Xo[T]) = \L(\Xt[T])$, which proves the desired
uniqueness in law of weak solutions to the local equation.

Our second proof of uniqueness, given in Section \ref{se:pf:uniqueness-localGW}, does not require boundedness of the drift, but is less direct in the sense that it relies on well-posedness of the infinite particle system $X_\Z$ described by \eqref{def:T2-particlesystem}.  This proof exploits the conditional independence and symmetry properties described in (i) and (ii) above to essentially rebuild 
the law of $X_\Z$ using  just the joint law of the root neighborhood. 
Specifically, given a solution $(Y_{-1},Y_0,Y_1)$ to the local equation, let
$\fmeas(dy_{-1}, dy_0, dy_{1})$ denote the joint law of the root neighborhood $(Y_{-1}, Y_0, Y_1)$,
and let
$\Gamma(dy_{1}; Y_0, Y_{1})$  denote the conditional law of $Y_{-1}$ given $(Y_0,Y_{1})$.
By the first symmetry property in \eqref{loceq-symmetry}, the conditional law of $Y_{1}$ given $(Y_0,Y_{-1})$ is precisely   $\Gamma(dy_1; Y_0, Y_{-1})$. 
We then consider the unique probability measure on $\C^\Z$ with (consistent) finite-dimensional distribution on $\C^{\Z \cap [-n,n]}$ given by
\begin{align}
\fmeas(dy_{-1},dy_0,dy_1)\prod_{i=1}^{n-1} \Gamma(dy_{i+1};y_i,y_{i-1}) \Gamma(dy_{-(i+1)};y_{-i},y_{-(i-1)}) \label{intro:uniqconstr}
\end{align}
for each $n \in \N$,  where the product of the kernels reflects
the conditional independence property of $X_\Z$ stated in \eqref{def:T2-condind}.
The crux of the argument is to show that this probability measure on $\C^\Z$ is the law of a solution of the infinite SDE system \eqref{def:T2-particlesystem}; uniqueness for the local equation then follows from uniqueness for the infinite particle system.  
The full justification is much more involved but
ultimately rests upon conditional independence and symmetry arguments like those used above, as well as judicious use of Girsanov's theorem to characterize $\Gamma$ and the measures in \eqref{intro:uniqconstr}. 
It is worth emphasizing that, by purely measure-theoretic arguments, the law of \emph{any} random sequence $X_{\Z} = (X_i)_{i \in \Z}$ that is  invariant under shifts and reflections,  and also  satisfies the conditional independence property \eqref{def:T2-condind}, is uniquely determined by its root neighborhood marginal via the construction in \eqref{intro:uniqconstr}. However,
the difficulty  lies  in transferring additional properties (such as the property that the  collection $X_\Z$ satisfies a
certain SDE) from the marginal to the full configuration, and vice versa. 
}

      \subsubsection{Additional Challenges on Random Trees}
      \label{subsub-gencase}

      As we have seen above, 
      three main ingredients of the proof of  characterization of the law of marginal dynamics
      in terms of  the local equation 
      include a certain conditional independence property that is similar in spirit to the second-order MRF property, symmetry considerations, and a stochastic analytic result on projections of It\^{o} processes.
These arguments can be extended to more general dynamics and 
$\Tmb_\kappa$ for general $\kappa > 2$ in an analogous  manner, although the proofs are more involved, 
with the main  change being 
that one now exploits the  class of symmetries arising  from the automorphism group on $\Tmb_\kappa$,
which can be visualized as translation and rotation symmetries (see Section \ref{subs-locchardet} for the form of the local equation in this case).   
However, the intuition described above is somewhat limited to deterministic trees.

On random UGW trees, 
the  proof of the characterization of marginal dynamics via the local equation (described in Definition \ref{def-GWlocchar} and  Section \ref{se:statements:GWlocaleq}),  is an order of
magnitude harder, and requires new ingredients. 
Firstly, the conditional independence property   
must now be established in an annealed
sense, looking jointly at the particle system and the structure of the underlying tree,
and the statement and proof are significantly more 
involved (see Proposition \ref{pr:properties-GW}).
As for the second step, while the projection argument is similar, 
the  symmetry considerations must be 
  significantly altered, as they are not so useful in the
quenched form used for deterministic  regular trees. 
Instead, the appropriate
notion of symmetry here turns out to be {\em unimodularity}, which is defined
by a certain \emph{mass-transport principle} (elucidated in
  Section \ref{subs-unimod}). This can be viewed 
as a sort of stationarity property, which is often loosely described as the property that \emph{the root is equally likely to be any vertex} \cite{aldous-lyons}, although the precise 
formulation is more subtle.  In the course of the proof of our main result, we show in Proposition \ref{pr:unimodular} that this unimodularity property is preserved by dynamics of the form \eqref{eqn-genericmod}, which may be of interest in its own right.   

The unimodularity property is applied to establish a key identity
(see Proposition \ref{pr:invariance-GW}) that relates certain conditional expectations
related to the histories of the process at the root and its neighbors to a suitably
reweighted version of corresponding conditional expectations related to 
the histories of the process at a child of the root and its neighborhood,
leading  to a more complicated  form of
the analogue of $\tilde\gamma_t$ (as discussed in  Remark \ref{intuition-GW}).  
Section \ref{subs:auxiliary} 
contains precise statements of these key properties, which are applied  in Section \ref{se:pf:existence-localGW} to show that  the marginal distributions satisfy the UGW local equation. 
Finally, the more complicated form of the local equation on the UGW tree  also leads 
to additional subtleties in the last step of establishing well-posedness of
the local equation (see  Sections \ref{se:pf:uniqueness-localGW} and \ref{sec:alternative}). 
  In particular, both proofs now entail certain non-trivial 
  change of measure arguments that
were not necessary in the case of the deterministic regular tree; 
  for  the second proof, see Section \ref{subs-uniqueoutline} for an outline  
  and Section \ref{subs-uniqueproof} for the details and for the first proof, see Section
  \ref{subs-altuniq2}.

Precise statements of our main results are  given in Section \ref{sec-main}. In the next section, we first develop some notation.  
  
\section{Preliminaries and Notation} 
\label{sec-notat} 

In this section, we introduce common notation and definitions used throughout the paper, and which are required to state the main results. Throughout, we write $\N_0:=\N \cup \{0\}$.

\subsection{Graphs and the Ulam-Harris-Neveu labeling for trees}
\label{subs-graphs}

\subsubsection{General graph terminology}
\label{subsub-graphs}

Given a graph $G = (V,E)$, we will often abuse notation
by writing $v \in G$ for $v \in V$ to refer to a vertex or node of the graph.  In this paper, we will always assume that the graph has a finite or countably infinite
vertex set and is simple (no self-edges or multi-edges). 
Given $u, v \in V$, a path from $u$ to $v$ is a sequence of distinct vertices $u=u_0,u_1,u_2, \ldots, u_n=v$ such that 
$(u_{i-1},u_i) \in E$ for $i = 1, \ldots, n$.  The graph $G$ is said to be connected if there exists
a path between any two vertices $u, v \in V$.  
For two vertices $u,v \in V$, the distance between $u$ and $v$ is the length of the shortest path from $u$ to $v$, or $\infty$ if no such path exists.   The diameter $\mathrm{diam}(A)$ of a set $A \subset V$ is the maximal distance between vertices of $A$.
For $v \in V$, the neighborhood of $v$ in $G$ is defined to be
\[
N_v(G) :=  \{u \in V \setminus \{v\}: (u,v) \in E\}.
\]
The degree of a vertex $v$ is $|N_v(G)|$, where as usual $|A|$ denotes the cardinality
of a set $A$.  A graph is said to be \emph{locally finite} if each vertex has a finite degree. 
Given $A \subset V$, its boundary and double boundary are defined to be
\begin{align}
\begin{split}
\partial A &:= \{u \in V \setminus A: \exists v \in A \text{ such that } (u,v) \in E\}, \\
\partial^2 A &:= \partial A \cup \partial ( A \cup \partial A).
\end{split} \label{def:boundaries}
\end{align}
Note that $\partial A$ (resp.\ $\partial^2A$) is the set of vertices that are at a distance $1$ (resp.\ $1$ or $2$) from $A$. 
A \emph{clique} is a complete subgraph, that is,  a set $A \subset V$ such that $(u,v) \in E$ for every distinct $u,v \in A$. Equivalently, 
a clique is a set $A \subset V$ of diameter at most $1$.
Similarly, we say that a set $A \subset V$ is a \emph{$2$-clique} if $\mathrm{diam}(A) \le 2$.

\subsubsection{The Ulam-Harris-Neveu labeling for trees}
\label{subsub-labeling}

A tree is a (undirected) graph $G = (V,E)$ such that given any two vertices $u, v \in G$,
there is a unique path between $u$ and $v$.
It will be convenient to work with a canonical labeling scheme for
trees known as the Ulam-Harris-Neveu labeling (see, e.g.,  \cite[Section VI.2]{harris-book} or \cite{neveu1986arbres}), defined using the vertex set
\begin{equation}
  \label{vertexset}
  \V := \{\o \} \cup \bigcup_{k=1}^\infty\N^k
\end{equation}
For $u, v \in \V$, let $uv$ denote  concatenation, that is, if $u=(u_1,\ldots,u_k) \in \N^k$ and $v=(v_1,\ldots,v_j) \in \N^j$, then $uv = (u_1,\ldots,u_k,v_1,\ldots,v_j) \in \N^{k+j}$. The root $\o$ is the identity element, so $\o u = u \o = u$ for all $u \in \V$.  
 For $v \in \V \backslash \{\o \}$, we write $\mom_v$ 
for the parent of $v$; precisely, $\mom_v$ is the unique element of $\V$ such that there exists
$k \in \N$ satisfying $v=\mom_{v}k$.  We view $\V$ as a graph by  declaring  two vertices to be adjacent if one is the parent of the other. Thus, the neighborhoods of $\V$ are $N_{\o}(\V)=\N$  and $N_v(\V) = \{\mom_v\} \cup \{vk  : k \in \N\}$ for $v \in \V \backslash \{\o\}$. Note that this graph $\V$ is not locally finite.  

There is a natural partial order on $\V$. We say $u \le v$ if there exists (a necessarily unique) $w \in \V$ such that $uw=v$, and say $u < v$ when $w \neq \o$. 
A subset $\tree \subset \V$ is defined to be a \emph{tree} if:
\begin{enumerate}
\item $\o \in \tree$; 
\item If $v \in \tree$ and $u \in \V$ with $u \le v$, then $u \in \tree$; 
\item For each $v \in \tree$ there exists an integer $c_v(\tree) \ge 0$ such that, for $k \in \N$, we have $vk \in \tree$ if and only if $1 \le k \le c_v(\tree)$. 
\end{enumerate}
Note that for us a tree, by default, is locally finite. 
We also use the symbol $\tree$ to refer not only to the subset of $\V$ but also to the induced subgraph. Inductively, for $u \in \tree$, we think of the elements
$(uv)_{v=1}^{c_u(\tree)}$ as the children of the vertex labeled $u$.
For any $\tree \subset \V$ and $v \in \V$, define $N_v(\tree) = \tree \cap N_v(\V)$ to be the set of neighbors of $v$ in $\tree$ if $v \in \tree$, 
and set $N_v(\tree) = \emptyset$ if $v \notin \tree$.
It is convenient to  define also 
$\V_n$ to be  the labels of the first $n$ generations: 
\begin{equation}
  \label{def-vn}
\V_n := \{\o \} \cup \bigcup_{k=1}^n\N^k.
\end{equation}
With a minor abuse of notation, we also use $\V_n$ to denote the corresponding induced subgraph.

\subsection{Measure Spaces}
\label{subs-mspaces}

For a Polish space $\X$, we write $\P(\X)$ for the set of Borel probability measures on $\X$, endowed always with the topology of weak convergence. Note that $\P(\X)$ itself becomes a Polish space with this topology, and we equip it with the corresponding Borel $\sigma$-field. We write $\delta_x$ for the Dirac measure at a point $x \in \X$. For an $\X$-valued random variable $X$, we write $\L(X)$ to denote its law, which is an element of $\P(\X)$.
Given any measure $\nu$ on a measurable space and any  $\nu$-integrable
function $f$ on that space, we use the usual shorthand notation
$\langle \nu, f\rangle := \int f\,d\nu$. Given  $\X$-valued random elements
$Y, Y_n, n \in \N$, we write $Y_n \Rightarrow Y$ to mean that the
law of $Y_n$ converges weakly to the law of $Y$.

\subsection{Function Spaces}
\label{subs-fspaces}

For a fixed positive integer $d$,  
throughout we write 
\[
\C := C(\R_+;\R^d)
\]
for the path space of continuous functions,
endowed with the topology of uniform convergence on compacts.
For $t > 0$, we write $\C_t := C([0,t];\R^d)$, and for $x \in \C$  we write $\|x\|_{*,t} := \sup_{s \in [0,t]}|x(s)|$
and $x[t] := \{x(s), s \in [0,t]\}$ for the truncated path, viewed
as an element of $\C_t$.

\subsection{Configuration spaces}
\label{subs-conspaces}

For a set $\X$ and a graph $G=(V,E)$, we write $\X^V$ or $\X^G$ for the configuration space
$\{(x_v)_{v \in V}: x_v \in \X \mbox{ for every } v \in V\}$. We make use of a standard notation for configurations on subsets of $V$: For $x=(x_v)_{v \in V} \in \X^V$ and $A \subset V$, we write $x_A$ for the element $x_A=(x_v)_{v \in A}$ of $\X^A$.

\subsection{Space of unordered terminating sequences} \label{se:symmetricsequencespace}

As discussed in the introduction, 
we will study stochastic differential equations that  take
values in a sequence of  configuration spaces with corresponding underlying interaction 
graphs that have  different numbers of vertices. 
We want to be able to specify a single ``drift function" that takes as input finite sequences of elements of $\X$ of arbitrary length and is insensitive to the order of these elements.

 To this end, for a set $\X$, we define in this paragraph a space $\SQ(\X)$ of finite unordered $\X$-valued sequences of arbitrary length (possibly zero). First, for $k \in \N$ we define the symmetric power (or unordered Cartesian product) $S^k(\X)$ as the quotient of $\X^k$ by the natural action of the symmetric group on $k$ letters. For convenience, let $S^0(\X)=\{\onepoint\}$. Define $\SQ(\X)$ as the disjoint union,
\[
\SQ(\X) = \bigsqcup_{k=0}^\infty S^k(\X).
\] 
A typical element of  $\SQ(\X)$ will be denoted $(x_v)_{v \in V}$, for a finite (possibly empty) set $V$; if the set is empty, then by convention $(x_v)_{v \in V} = \onepoint \in S^0(\X)$. It must be stressed that, of course, the element $(x_v)_{v \in V}$ has no order. The space $\SQ(\X)$ must not be confused with what is traditionally called the \emph{infinite symmetric product space} in algebraic topology when $\X$ is endowed with a distinguished (base) point $e$, in which the points $(x_1,\ldots,x_n,e)$ and $(x_1,\ldots,x_n)$ would be identified; these two points are distinct in $\SQ(\X)$.

Suppose now that $(\X,d)$ is a metric space, and endow $\SQ(\X)$, with the usual disjoint union topology, i.e., the finest topology on $\SQ(\X)$ for which the injection $S^k(\X) \hookrightarrow \SQ(\X)$ is continuous for each $k \in \N$. A function $F : \SQ(\X) \rightarrow \Y$ to a metric space $\Y$ is continuous if and only if there is a sequence $(f_k)_{k = 0}^\infty$, where $f_0 \in \Y$ and, for each $k \in \N$, $f_k : \X^k \rightarrow \Y$ is a continuous function that is  symmetric in its $k$ variables, such that
\[
F((x_i)_{i \in \{1,\ldots,k\}}) = \begin{cases}
f_k(x_1,\ldots,x_k) &\text{for } k \in \N, \ (x_1,\ldots,x_k) \in \X^k \\
f_0 &\text{for } k=0. 
\end{cases}
\]
If $\X$ is separable and completely metrizeable, then so is $\SQ(\X)$.
Note that a sequence $(x^n_v)_{v \in V_n}$ in $\SQ(\X)$ converges to $(x_v)_{v \in V}$ if and only if for all $\epsilon > 0$ there exists $N \in \N$ such that  for all $n \ge N$ there exists a bijection $\varphi : V_n \rightarrow V$ such that $\max_{v \in V_n}d(x^n_v,x_{\varphi(v)}) < \epsilon$. (Note that this implicitly requires that $|V_n|=|V|$ for sufficiently large $n$.)
It is worth noting that continuous functions on $\SQ(\X)$ are strictly more general than weakly continuous functions on the set of empirical measures, but we refer to \cite{LacRamWu20a} for further discussion.

\section{Statements of main results}
\label{sec-main}

For a tree $\tree$, viewed as a subset of $\V$ as defined in Section \ref{subsub-labeling}, we are interested in the SDE system
\begin{align}
dX^\tree_v(t) &= 1_{\{v \in \tree\}}\Big(b(t,X^\tree_v,X^\tree_{N_v(\tree)})dt + \sigma(t,X^\tree_v)dW_v(t)\Big), \ \ \ v \in \V, \label{statements:SDE}
\end{align} 
where recall that $N_v(\tree)$ denotes the set of neighbors of $v$ in $\tree$,
  and $b$ and $\sigma$ are suitable progressively measurable coefficients as specified in  Assumption \ref{assumption:A}.  
When the tree $\tree$ is random, we always take it to be independent of the initial conditions and Brownian motions.
Note  that we include even those labels $v \in \V \setminus \tree$ that do not belong to the tree, for which the process is constant $X^\tree_v(t)=X^\tree_v(0)$; this will be convenient notation and, in the random tree case, will render the tree itself measurable with respect to the initial $\sigma$-field (as elaborated in Remark \ref{rem-treerecovery}).
Also note  that, unlike in the introduction, we
allow path-dependent coefficients $(b,\sigma)$,  both because this
arises in applications and because this results in no change in the arguments or in the form of the local equation described in Section \ref{subs-locchar}, 
which are inevitably path-dependent regardless of whether $b$ and $\sigma$ are, as discussed in Section \ref{subsub-linegraph}(ii).

We  state first our standing assumptions in
Section \ref{ass}. Then we introduce the local equation in Section \ref{subs-locchar} and
finally state our main results in Section \ref{se:statements:GWlocaleq}.
Throughout, recall the function space $\C = C(\R_+;\R^d)$ and sequence space
  $\SQ (\C)$ defined in Sections \ref{subs-fspaces} and \ref{se:symmetricsequencespace}, respectively.

\subsection{Assumptions}
\label{ass}

Fix a dimension $d \in \N$ and, for $x \in \C$ and $t > 0$, recall from
Section \ref{subs-fspaces} the notation  $\|x\|_{*,t} := \sup_{s \in [0,t]}|x(s)|$.
We assume the drift coefficient $b$, diffusion coefficient $\sigma$, and an initial distribution $\lambda_0$,  satisfy the following: 

\begin{assumption}{\textbf{A}} \label{assumption:A} $\ $

\begin{enumerate}
\item[(A.1)] The drift coefficient $b : \R_+ \times \C \times \SQ(\C) \rightarrow \R^d$ is continuous and has linear growth, in the sense that for each $T > 0$, there exists $C_T <  \infty$ such that, for any $(t,x,(x_v)_{v \in A}) \in [0,T] \times \C \times \SQ(\C)$, we have
\[
|b(t,x,(x_v)_{v \in A})| \le C_T\left(1 + \|x\|_{*,t} + \frac{1}{|A|}\sum_{v \in A}\|x_v\|_{*,t}\right),
\]
where the average is understood to be zero if $|A|=0$.  Moreover, $b$ is progressively measurable; that is, it is jointly measurable
  (which is already implied by the above continuity properties) and non-anticipative in the sense that 
  for each $t \ge 0$, $b(t,x,(x_v)_{v \in A}) = b(t,y,(y_v)_{v \in A})$ whenever $x(s)=y(s)$ and $x_v(s)=y_v(s)$ for all $s \le t$ and $v \in A$.
\item[(A.2)] The diffusion matrix $\sigma : \R_+ \times \C \rightarrow \R^{d\times d}$ satisfies the following: 
\begin{enumerate}
\item[(A.2a)] $\sigma$ is bounded and continuous. Moreover, $\sigma(t,x)$ is invertible for each $(t,x)$, and the inverse is uniformly bounded. Lastly, $\sigma$ is progressively
  measurable, which implies  that  for each $t \ge 0$, $\sigma(t,x)=\sigma(t,y)$ whenever $x(s)=y(s)$ for all $s \le t$.
\item[(A.2b)] The following driftless SDE admits a unique in law weak solution: 
\[
dX(t) = \sigma(t,X)dW(t), \quad X(0) \sim \lambda_0.
\]
\end{enumerate} 
\item[(A.3)] The initial states $(X^\tree_v(0))_{v \in \V}$ are i.i.d.\ with common distribution $\lambda_0 \in \P(\R^d)$, and $\lambda_0$ has finite second moment.
\item[(A.4)] For each non-random tree $\tree \subset \V$, 
  there exists a unique in law weak solution of the SDE system \eqref{statements:SDE} with i.i.d.\ initial positions $(X^\tree_v(0))_{v \in \V}$ with law $\lambda_0$. 
\end{enumerate}
\end{assumption}

The final condition (\ref{assumption:A}.4) regarding uniqueness in law for \eqref{statements:SDE} is not as stringent as it may appear. 
If the tree $\tree$ is finite, it follows automatically from Assumptions (\ref{assumption:A}.1)--(\ref{assumption:A}.2) and Girsanov's theorem (see Lemma \ref{le:ap:girsanov}).   For infinite graphs,  Theorem \ref{th:uniqueness-hominfSDE} below shows that Assumption (\ref{assumption:A}.4) holds if $b$ and $\sigma$ are suitably Lipschitz. The i.i.d.\ assumption  on the initial conditions in (\ref{assumption:A}.3) can be relaxed, although we do  not do so in this article; see Remark \ref{re:non-iid} for further discussion.

\begin{remark}  \label{rem-unique}
As an immediate consequence of Assumption (\ref{assumption:A}.4), it follows that
  the SDE \eqref{statements:SDE} is  unique in law even 
  when the tree $\tree$ is random,  since we always take $\tree$ to be independent of the
  initial conditions and the Brownian motions.
\end{remark}

{
\begin{theorem}  \label{th:uniqueness-hominfSDE} 
	Suppose that Assumptions (\ref{assumption:A}.1) and (\ref{assumption:A}.2a) hold. Assume also that the functions $b$ and $\sigma$ are Lipschitz, in the sense that for each $T > 0$, there exist $\bCa_T, \bCb_T < \infty$ such that, for all $t \in [0,T]$, all $x, x' \in \C$, and all $(x_u)_{u \in A},(x'_u)_{u \in A} \in \SQ(\C)$ indexed by the same finite set $A$, we have
	\begin{equation}    
	\label{Lip-b}
	|b(t,x,(x_u)_{u \in A}) - b(t,x',(x'_u)_{u \in A})| \le \bCa_T \left(\|x-x'\|_{*,t} + \frac{1}{|A|}\sum_{u \in A}\|x_u - x'_u\|_{*,t}\right),
	\end{equation} 
	where the average is understood to be zero if $|A|=0$, and
	\begin{equation}
	\label{Lip-sigma}
	|\sigma(t,x) - \sigma(t,x')| \le \bCb_T\|x-x'\|_{*,t}. 
	\end{equation}
	Then there exists a pathwise unique strong solution  for the SDE system
	\eqref{statements:SDE}, with any initial conditions $(X^\tree_v(0))_{v \in \tree}$.
\end{theorem}

\begin{proof}
This follows from standard arguments; see \cite[Theorem 3.1]{LacRamWu20a}.
\end{proof} 

Motivated by Theorem \ref{th:uniqueness-hominfSDE}, we will sometimes make the following assumption.

\begin{assumption}{\textbf{B}} \label{assumption:B}
	Suppose that Assumptions (\ref{assumption:A}.1), (\ref{assumption:A}.2a), and (\ref{assumption:A}.3) hold.
	Assume also that the functions $b$ and $\sigma$ are Lipschitz, in the sense that \eqref{Lip-b} and \eqref{Lip-sigma} hold.
\end{assumption}

We note that due to
  Theorem \ref{th:uniqueness-hominfSDE},  Assumption \ref{assumption:B} implies Assumption \ref{assumption:A}.
}

The main examples of interactions we have in mind for the drift $b$ in Assumption (\ref{assumption:A}.1) take the following forms: 

\begin{example} \label{ex:interactions1}
For a first example, suppose $b$ is of the form
\begin{align*}
b(t,x,(x_v)_{v \in A}) = \begin{cases}
\widetilde{b}_0(t,x) &\text{if } A = \emptyset,  \\
\frac{1}{|A|}\sum_{v \in A}\widetilde{b}(t,x,x_v) &\text{if } A \neq \emptyset,
\end{cases}
\end{align*}
for given functions $\widetilde{b}_0 : \R_+ \times \C \rightarrow \R^d$ and $\widetilde{b} : \R_+ \times \C \times \C \rightarrow \R^d$. Assumption (\ref{assumption:A}.1) holds if $\widetilde{b}_0$ and $\widetilde{b}$ are continuous with linear growth, in the sense that  for each $T > 0$ there exists $C_T <  \infty$ such that
\[
|\widetilde{b}_0(t,x)| + |\widetilde{b}(t,x,y)| \le C_T\left(1 + \|x\|_{*,t} + \|y\|_{*,t}\right), \ \text{ for all } (t,x,y).
\]
\end{example}

\begin{example} \label{ex:interactions2}
Generalizing Example \ref{ex:interactions1}, suppose $b$ is of the form
\begin{align*}
b(t,x,(x_v)_{v \in A}) = \begin{cases}
\widetilde{b}_0(t,x) &\text{if } A = \emptyset, \\
\widetilde{b}\left(t,x,\frac{1}{|A|}\sum_{v \in A}\delta_{x_v}\right) &\text{if } A \neq \emptyset,
\end{cases}
\end{align*}
for given functions $\widetilde{b}_0 : \R_+ \times \C \rightarrow \R^d$ and $\widetilde{b} : \R_+ \times \C \times \P(\C) \rightarrow \R^d$.
In fact, $\widetilde{b}$ needs only to be defined on the subspace of $\P(\C)$ consisting of empirical measures of finitely many points.
Assumption (\ref{assumption:A}.1) holds if $\widetilde{b}_0$ and $\widetilde{b}$ are continuous (using weak convergence or any Wasserstein metric on $\P(\C)$) with linear growth, namely if  for each $T > 0$ there exists $C_T < \infty$ such that
\[
|\widetilde{b}_0(t,x)| + |\widetilde{b}(t,x,m)| \le C_T\left(1 + \|x\|_{*,t} + \int_{\C}\|y\|_{*,t}\,dm(y)\right), \ \text{ for all } (t,x,m).
\]
\end{example}

\subsection{The local equation}
\label{subs-locchar} 

The local equation describes a novel stochastic dynamical system
and is significantly more complicated on the UGW tree than on
non-random trees, where its structure is more transparent, especially given the
discussion in Section \ref{subsub-linegraph}. 
Thus, we first introduce its definition 
on the infinite regular tree in Section \ref{subs-locchardet} and defer 
the full formulation for a UGW tree to  Section \ref{subs-loccharUGW}.
However, the reader may choose to skip directly to Section \ref{subs-loccharUGW}
without loss of continuity.

\subsubsection{The local equation for an  infinite regular  tree}
\label{subs-locchardet}

Let $\Tmb_{\kappa}$ be the infinite $\kappa$-regular tree for some integer $\kappa \ge 2$, and
note that it can  be  identified with the subset 
$\{\o\} \cup \{1,\ldots,\kappa\} \cup \bigcup_{n=2}^\infty \left( \{1,\ldots,\kappa\} \times \{1,\ldots,\kappa-1\}^{n-1} \right)$ 
of the vertex set $\V$ defined in \eqref{vertexset}.

Recall from Section \ref{subs-fspaces} that for  $t > 0$ and $x \in \C = C(\R_+;\R^d)$, we write $x[t] := \{x(s) : s \in [0,t]\}$ for the truncated path, viewed as an element of $\C_t =C([0,t];\R^d)$. The following generalizes the local equation outlined in Section \ref{subsub-linegraph} for a model on $\Tmb_2$.

\begin{definition}
  \label{def-locregtree} 
Let $\Tmb_{\kappa,1} = \{\o,1,\ldots,\kappa\}$ denote the first generation of the $\kappa$-regular tree. 
A \emph{weak solution of the $\Tmb_\kappa$ local equation with initial law $\lambda_0 \in \P(\R^d)$}  is a tuple $((\Omega,\F,\FF,\PP), \gamma,(B_v,Y_v)_{v \in \Tmb_{\kappa,1}})$ such that:
\begin{enumerate}
\item $(\Omega,\F,\PP)$ is a probability space with a filtration $\FF=(\F_t)_{t \ge 0}$.
\item $(B_v)_{v \in \Tmb_{\kappa,1}}$ are independent $d$-dimensional $\FF$-Brownian motions.
\item $(Y_v)_{v \in \Tmb_{\kappa,1}}$ are continuous $d$-dimensional $\FF$-adapted processes. 
\item $(Y_v(0))_{v \in \Tmb_{\kappa,1}}$ are i.i.d.\ with law $\lambda_0$.
\item The function $\R_+ \times \C^2 \ni (t,x_{\o},x_1) \mapsto  \gamma_t(x_{\o},x_1) \in \R^d$ is progressively measurable and satisfies
\begin{align}
 \gamma_t(Y_{\o},Y_1) = \E\Big[ b(t,Y_{\o},Y_{\{1,\ldots,\kappa\}}) \, \big| \, Y_{\o}[t], \, Y_1[t] \Big], \quad a.s., \mbox{ for }  a.e. \ t \in [0,T].   \label{def:gamma-kappareg}
\end{align}
\item The following system of stochastic equations holds: 
\begin{align}
\begin{split}
  dY_{\o}(t) & = b(t,Y_{\o},Y_{\{1,\ldots,\kappa\}}) \, dt + \sigma(t,Y_{\o}) \, dB_{\o}(t),
  \\ 
dY_i(t) & = \gamma_t(Y_i,Y_{\o}) dt + \sigma(t,Y_i) \, dB_i(t), \quad i=1,\dotsc,\kappa. 
\end{split} \label{statements:localequation-regular}
\end{align}
\item For each $i=1,\ldots,\kappa$ and $T > 0$, we have
\begin{align*}
\int_0^T\Big(|\gamma_t(Y_{\o},Y_i)|^2+|\gamma_t(Y_i,Y_{\o})|^2+|\gamma_t(\widehat{X}_{\o},\widehat{X}_i)|^2+|\gamma_t(\widehat{X}_i,\widehat{X}_{\o})|^2\Big)dt < \infty, \ \ a.s.,
\end{align*}
where $(\widehat{X}_v)_{v \in \Tmb_{\kappa,1}}$ is the unique in law
(by Assumption (\ref{assumption:A}.2b)) solution to the driftless SDE  system
\begin{align*}
d\widehat{X}_v(t) = \sigma(t,\widehat{X}_v)dB_v(t), \quad v \in \Tmb_{\kappa,1},
\end{align*}
where $(\widehat{X}_v(0))_{v \in \Tmb_{\kappa,1}}$ are i.i.d.\ with law $\lambda_0$. 
\end{enumerate}
Alternatively, we may refer to the law of the $\C^{\kappa+1}$-valued random variable $(Y_v)_{v \in \Tmb_{\kappa,1}}$ as a weak solution.
We say that the $\Tmb_\kappa$ local equation with initial law $\lambda_0$ is \emph{unique in law} if any two weak solutions induce the same law on $\C^{\kappa+1}$. 
\end{definition}

\begin{remark}
	The property (7) in Definition \ref{def-locregtree} will be used to justify certain applications of Girsanov's theorem (as in  Lemma \ref{le:ap:girsanov}). Specifically, it ensures that the joint laws of $(Y_{\o},Y_i)$ and $(\widehat{X}_{\o},\widehat{X}_i)$ are mutually absolutely continuous.	
\end{remark}

\begin{remark}
The local equation describes a ``nonlinear'' process in the sense of a McKean-Vlasov equation because the law of the solution enters the dynamics.
  However, a crucial yet unusual  feature of the local equation \eqref{statements:localequation-regular} is that the conditional expectation mapping $\gamma_t$ appears with different arguments throughout the SDE system.     In  the related paper \cite{LacRamWu20c} (see also \cite{MitchellThesis18} and \cite{SudijonoThesis19}),   
 we show that  analogous discrete-time local dynamics can be simulated efficiently.
  In future work, we plan to investigate the    analytical and numerical tractability
  of the local dynamics in the diffusion setting.  
\end{remark}

It is worth noting how the $\Tmb_\kappa$ local equation \eqref{statements:localequation-regular} simplifies when the drift $b$ takes the form described in Example \ref{ex:interactions1} above. Indeed, as shown in  Lemma \ref{lem-locsym},  
the law of any solution $(Y_{\o},Y_1,\ldots,Y_\kappa)$ is necessarily  invariant under permutations of $(Y_1,\ldots,Y_\kappa)$,  which implies
\begin{align*}
\gamma_t(Y_{\o},Y_1) &= \frac{1}{\kappa}\widetilde{b}(t,Y_{\o},Y_1) + \frac{\kappa-1}{\kappa}\widetilde\gamma_t(Y_{\o},Y_1),	
\end{align*} 
where we define 
\[
\widetilde\gamma_t(Y_{\o},Y_1) := \E\Big[ \widetilde{b}(t,Y_{\o},Y_2) \, \big| \, Y_{\o}[t], \, Y_1[t] \Big].
\] 
We may then write \eqref{statements:localequation-regular} as 
\begin{align*}
dY_{\o}(t) & = \frac{1}{\kappa}\sum_{i=1}^\kappa \widetilde{b}(t,Y_{\o},Y_i) \, dt + \sigma(t,Y_{\o}) \, dB_{\o}(t), \\
dY_i(t) & = \left(\frac{1}{\kappa}\widetilde{b}(t,Y_i,Y_{\o}) + \frac{\kappa-1}{\kappa} \widetilde\gamma_t(Y_i,Y_{\o}) \right) dt + \sigma(t,Y_i) \, dB_i(t), \ \ i=1,\dotsc,\kappa. 
\end{align*}

The main result on the characterization of the dynamics of the root and its neighborhood $\Tmb_{\kappa,1}$ via the local equation is given in Corollary \ref{cor-regtree}. It is  a simple consequence of the more
  general result, Theorem \ref{th:statements-mainlocaleq}, for UGW($\rho$) trees given in Section \ref{se:statements:GWlocaleq}.
  With that in mind,  in the next section, we first introduce the general form of the local equation for a UGW tree.

\subsubsection{The local equation for unimodular Galton-Watson trees}
\label{subs-loccharUGW}

Fix a distribution $\rho \in \P(\N_0)$ with finite non-zero first moment.  
We first formally define a UGW($\rho$) tree: 

\begin{definition}
  \label{def-UGW}
 Given $\rho \in \P(\N_0)$
 with a finite nonzero first moment, the random tree UGW($\rho$) has a root 
 with offspring distribution $\rho$, and each vertex of each subsequent generation has a
 number of offspring according to the distribution $\widehat{\rho} \in \P(\N_0)$, where $\widehat{\rho}$ is given by
\begin{equation}
  \label{def2-hatrho}
\widehat\rho(k) = \frac{(k+1)\rho(k+1)}{\sum_{n \in \N}n\rho(n)}, \ \ k \in \N_0, 
\end{equation}
and the numbers of offspring in different generations are all independent of each other. 
 Recalling the Ulam-Harris-Neveu labelling from Section \ref{subsub-labeling}, we view a UGW($\rho$) tree as a random subset of $\V$. 
\end{definition}

As discussed in Section \ref{subs-backmot}, this kind of random tree arises
as the local weak limit of many natural finite random graph models (see Examples 2.2, 2.3, and 2.4 of \cite{LacRamWu20a}).

We now give the general form of the local equation for UGW trees.
In this case, the structure of the neighborhood of the root is also random. To capture
 this, it is useful to consider the root neighborhood as a subset of the vertex set $\V_1 = \{\o\} \cup \N$.

\begin{definition}  \label{def-GWlocchar}
Given $\rho \in \P(\N_0)$ with finite nonzero first moment and $\lambda_0 \in \P(\R^d)$, a \emph{weak solution of the $\mathrm{UGW}(\rho)$ local equation with initial law $\lambda_0$} is a tuple $((\Omega,\F,\FF,\PP),\tree_1, \gamma,(B_v,Y_v)_{v \in \V_1},\Nvar)$ such that:
\begin{enumerate}
\item $(\Omega,\F,\PP)$ is a probability space with a filtration $\FF=(\F_t)_{t \ge 0}$.
\item $\tree_1$ is a random tree with the same law as the first generation of a $\mathrm{UGW}(\rho)$ tree. More explicitly, $\tree_1$ has vertex set $\{\o,1,\ldots,\kappa\}$ for some $\N_0$-valued $\F_0$-measurable random variable $\kappa$ with law $\rho$, and the edge set is $\{(\o,k) : k =1,\ldots,\kappa\}$. (If $\kappa=0$, this means the vertex set is simply $\{\o\}$, there are no edges,  and $N_{\o}(\tree_1)= \emptyset$.)
\item $\Nvar$ is an $\F_0$-measurable $\N_0$-valued random variable with law $\widehat\rho$, as defined in \eqref{def2-hatrho}.  
\item $(B_v)_{v \in \V_1}$ are independent $d$-dimensional $\FF$-Brownian motions.
\item $(Y_v)_{v \in \V_1}$ are continuous $d$-dimensional $\FF$-adapted processes.
\item $(Y_v(0))_{v \in \V_1}$ are $\F_0$-measurable and i.i.d.\ with law $\lambda_0$.  
\item The function $\R_+ \times \C^2 \ni (t,x_{\o},x_1) \mapsto \gamma_t(x_{\o},x_1) \in \R^d$ is progressively measurable and satisfies
   
\begin{equation}
  \gamma_t(Y_{\o},Y_1) = 
 \left\{
  \begin{array}{ll} 
    \displaystyle  \frac{\E\left[\left. \frac{|N_{\o}(\tree_1)|}{1+\Nvar}b(t,Y_{\o},Y_{N_{\o}(\tree_1)}) \, \right| \, Y_{\o}[t], \, Y_1[t] \right]}{\E\left[\left. \frac{|N_{\o}(\tree_1)|}{1+\Nvar} \, \right| \, Y_{\o}[t], \, Y_1[t] \right]}
    & \mbox{ on }
    \{N_{\o}(\tree_1) \neq \emptyset\},  \\
  \displaystyle  b(t,Y_{\o},\onepoint) & \mbox{ on } \{N_{\o}(\tree_1)=\emptyset\}, 
  \end{array}
  \right. 
  \label{def:gamma}
\end{equation} 
a.s., for a.e.\ $t \in [0,T]$. Recall our convention that $\onepoint$ denotes the unique element of the one-point space $\C^0$. 
\item $\tree_1$, $(Y_v(0))_{v \in \V_1}$, $\Nvar$, and $(B_v)_{v \in \V_1}$ are independent.
\item The following system of stochastic equations is satisfied:
\begin{align}
\begin{split}
dY_{\o}(t) & = b(t,Y_{\o},Y_{N_{\o}(\tree_1)}) \, dt + \sigma(t,Y_{\o}) \, dB_{\o}(t),  \\
dY_k(t) & = 1_{\{k \in \tree_1\}}\Big(\gamma_t(Y_k,Y_{\o}) \, dt + \sigma(t,Y_k) \, dB_k(t)\Big), \ \ k \in \N.
\end{split} \label{statements:localequationGW}
\end{align}
\item For each $k \in \N$ and $T > 0$, we have
\begin{align*}
\int_0^T\Big(|\gamma_t(Y_{\o},Y_k)|^2 + |\gamma_t(Y_k,Y_{\o})|^2+|\gamma_t(\widehat{X}_{\o},\widehat{X}_k)|^2+ |\gamma_t(\widehat{X}_k,\widehat{X}_{\o})|^2\Big)dt < \infty, \ \ a.s.,
\end{align*}
on the event $\{k\in\tree_1\}$,
where $(\widehat{X}_v)_{v \in \V_1}$ is the unique in law (by Assumption (\ref{assumption:A}.2b)) solution of the driftless SDE  system
\begin{align*}
d\widehat{X}_v(t) = 1_{\{v \in \tree_1\}}\sigma(t,\widehat{X}_v)dB_v(t), \quad v \in \V_1,
\end{align*}
where $(\widehat{X}_v(0))_{v \in \V_1}$ are i.i.d.\ with law $\lambda_0$.
\end{enumerate}
Alternatively, we may refer to the law of the $\C^{\V_1}$-valued random variable $(Y_v)_{v \in \V_1}$  as a weak solution.
We say that the $\mathrm{UGW}(\rho)$ local equation with initial law $\lambda_0$ is \emph{unique in law} if any two weak solutions induce the same law on $\C^{\V_1}$. 
\end{definition}

\begin{remark} \label{re:GWsimplifies-regular}
It is worth noting how Definition \ref{def-GWlocchar} reduces to Definition \ref{def-locregtree}  when the tree is the deterministic $\kappa$-regular tree, i.e., the UGW($\rho$) tree with $\rho = \delta_{\kappa}$ for an integer $\kappa \ge 2$. In this case, we have $\widehat\rho = \delta_{\kappa - 1}$,  $N_{\o}(\tree_1) = \{1,\ldots,\kappa\}$, and $\Nvar = \kappa-1$, and the definition of $\gamma_t$ in \eqref{def:gamma} reduces to \eqref{def:gamma-kappareg}.
\end{remark}

\begin{remark}  \label{intuition-GW}
The more complicated form of $\gamma_t$ in Definition \ref{def-GWlocchar}  as opposed
to Definition \ref{def-locregtree} is due to the subtler symmetries of the UGW tree in comparison with the simpler symmetries of the non-random trees $\Tmb_\kappa$ (for example, compare Lemma \ref{lem-locsym} with Lemma \ref{lem-locsymugw}). More precisely, note
that for the UGW tree $\tree$, 
on the event
  $\{|N_{\o}(\tree)| \neq \emptyset\}$, 
  the random variable $\Nvar$ is independent of
$|N_{\o}(\tree)|$ and 
  represents the number of
offspring of vertex $1$, which in a UGW($\rho$) tree has 
law $\widehat\rho$.
The following identity then provides 
  intuition behind the definition of $\gamma_t$ in
\eqref{def:gamma}:   for bounded
functions $h: \N \mapsto \R$,  
\[
\E\left[h(1+\Nvar)1_{\{N_{\o}(\tree_1) \neq \emptyset\}}\right] = \E\left[\frac{|N_{\o}(\tree)|}{1+\Nvar}h(|N_{\o}(\tree)|)1_{\{N_{\o}(\tree_1) \neq \emptyset\}}\right], 
\]
 which is easily verified by showing that both sides are equal to
  $[1 - \rho(0)] \sum_{k=0}^\infty h(k+1) \widehat{\rho}(k)$. 
This should be interpreted as explaining how to change measure, using the Radon-Nikodym derivative $|N_{\o}(\tree)|/(1+\Nvar)$, to effectively \emph{re-root} the tree to vertex $1$ instead of $\o$. See Proposition \ref{pr:invariance-GW} for a precise statement. Of course, in the $\kappa$-regular tree case discussed in Remark \ref{re:GWsimplifies-regular}, no such change of measure is necessary, because the re-rooted tree is isomorphic to the original tree.
On a more technical level, it is worth noting that the presence of the indicators in the SDE system \eqref{statements:localequationGW} ensures that $\{v \in \tree\}$ is a.s.\ $X_v[t]$-measurable for each $t > 0$ (see Remark \ref{rem-treerecovery}), and thus the conditional expectations appearing in \eqref{def:gamma} implicitly condition on the tree structure in addition to the particle trajectories. Also, 
our choice of how to define $\gamma_t(Y_{\o},Y_1)$  on the event $\{N_{\o}(\tree_1)=\emptyset\}$ is a useful convention but is irrelevant to the form of the local equation.
\end{remark}

\subsection{Characterization of marginals via the local equation}
\label{se:statements:GWlocaleq}

The following is our main result for particle systems set on UGW trees.

\begin{theorem} \label{th:statements-mainlocaleq}
Suppose Assumption \ref{assumption:A} holds. 
Let $\tree$ denote a $\mathrm{UGW}(\rho)$ tree, where $\rho \in \P(\N_0)$ has finite nonzero first and second moments. 
Let $X^\tree=(X^\tree_v)_{v \in \V}$ be the solution of the SDE system \eqref{statements:SDE}. 
Then the law of the $\C^{\V_1}$-valued random variable $(X^\tree_v)_{v \in \V_1}$ is a weak solution of the $\mathrm{UGW}(\rho)$ local equation with initial law $\lambda_0$. Moreover, the $\mathrm{UGW}(\rho)$ local equation with initial law $\lambda_0$ is unique in law.  
\end{theorem}

\begin{remark} \label{re:localeq-interpretation}
  To be absolutely clear about the meaning of Theorem \ref{th:statements-mainlocaleq},
  we must stress 
  that $(X^{\tree}_v)_{v \in \V_1}$ provides a weak solution of the local equation,
  but $(X^{\tree_1}_v)_{v \in \V_1}$ does not, where we write $\tree_1:=\tree\cap\V_1$ for the first generation of $\tree$.
    The difference is that 
  $X^{\tree_1} = (X^{\tree_1}_v)_{v \in \V}$ 
  denotes the particle system \emph{set on the one-generation tree} $\tree_1$, 
    in which the children of the root comprise the leaves of the tree, whereas
    $(X^{\tree}_v)_{v \in \V_1}$  represents the 
    root neighborhood for the particle system
    \emph{set on the potentially infinite UGW($\rho$) tree $\tree$}.
\end{remark}

Since the $\kappa$-regular tree is a special case of the UGW tree defined in Definition \ref{def-UGW} 
  (see Remark \ref{re:GWsimplifies-regular}),   the following is an immediate corollary of 
  Theorem \ref{th:statements-mainlocaleq}, where recall that $\mathbb{T}_{\kappa,1} := \{\o, 1, \ldots, \kappa\}$ represents one generation of the tree $\mathbb{T}_\kappa$.

\begin{corollary}  \label{cor-regtree}  
Suppose Assumption \ref{assumption:A} holds. 
  Let $\Tmb_{\kappa}$ denote the infinite $\kappa$-regular tree, for some $\kappa \ge 2$, and let $\Tmb_{\kappa,1}$ denote its first generation.
Let $X^{\Tmb_\kappa}=(X^{\Tmb_\kappa}_v)_{v \in \V}$ denote the solution of the SDE \eqref{statements:SDE} on the tree $\tree=\Tmb_\kappa$.
Then the law of $(X^{\Tmb_{\kappa}}_v)_{v \in \Tmb_{\kappa,1}}$ 
is a weak solution of the $\Tmb_\kappa$ local equation with initial law $\lambda_0$. Moreover, the $\Tmb_\kappa$ local equation with initial law $\lambda_0$ is unique in law.
\end{corollary}

As will be discussed in Section \ref{se:localconvergence}, combining Theorem \ref{th:statements-mainlocaleq} with the results of \cite{LacRamWu20a} yields a characterization of the limiting marginals and empirical measures of particle systems set on large finite graphs converging locally to UGW trees.

\begin{remark} 
The unimodularity condition on the  random tree, although convenient and natural in the context of local limits of random graphs, is not entirely necessary for obtaining a form of marginal dynamics.
 Indeed,  in  a related paper \cite{LacRamWu20c}, we obtain analogous results for
 interacting discrete-time Markov chains (equivalently, stochastic cellular automata),
 on standard  Galton-Watson trees. The marginal dynamics on a general Galton-Watson tree, however, involve the first two generations of the tree instead of just the first generation.  
The extra symmetry imposed by unimodularity 
 enables the reduction to a single generation, essentially because of the symmetry result of Proposition \ref{pr:invariance-GW} below. 
\end{remark}

\begin{remark} \label{re:non-iid}
We focus in this paper on i.i.d.\ initial conditions, for the sake of simplicity, but similar results are valid in greater generality. On the regular tree $\Tmb_\kappa$, if the SDE system \eqref{statements:SDE} starts from a distribution $\lambda \in \P((\R^d)^{\Tmb_\kappa})$ that is automorphism-invariant and a second order MRF (see Definition \ref{def-2MRF}), then Corollary \ref{cor-regtree}  remains valid with $(Y_v(0))_{v \in \Tmb_{\kappa,1}}$ distributed according to the $\Tmb_{\kappa,1}$-marginal of $\lambda$. A similar result should hold in the UGW case, but the requisite symmetry and conditional independence properties are much more subtle to formulate,  and  hence, deferred to future  work. 
\end{remark}

\subsection{Comments on the proof and two key auxiliary results}
\label{subs:auxiliary}

The proof of Theorem \ref{th:statements-mainlocaleq}, which is given in Section \ref{se:pfmainthm},  relies on two important properties
of the interacting particle system \eqref{statements:SDE}
which we  state in this section and which may be of independent interest.
The first result,  Proposition \ref{pr:properties-GW}, states 
the  form of the conditional independence property that we require.
 It can be viewed as a (more complicated) analogue of the second-order MRF property for $\mathbb{T}_2$-trees discussed in Section \ref{subsub-linegraph}(i), but the essential message remains the same: by conditioning on the particle trajectories at the root vertex and a child thereof, the particle trajectories in the two disjoint subtrees obtained by removing the edge between these two vertices become independent.

Recall in the following that $\mom_v$ denotes the parent vertex of any $v \in \V \backslash \{\o \}$ as defined in Section \ref{subsub-labeling}.

\begin{proposition} \label{pr:properties-GW}
Suppose Assumption \ref{assumption:A} holds, and suppose $\tree$ is a UGW($\rho$) tree, where $\rho \in \P(\N_0)$ has finite nonzero first moment.
Then, for each $t > 0$, the following hold:
\begin{enumerate}[(i)]
\item $(X^\tree_{ki}[t])_{i \in \N}$ is conditionally independent of $X^\tree_{\V_1}[t]$ given $X^\tree_{\{\o,k\}}[t]$, for any $k \in \N$.
\item For each $t > 0$, the conditional law of $(X_{ki}[t])_{i \in \N}$ given $(X^\tree_k[t],X_{\o}[t])$ does not depend on the choice of $k \in \N$. More precisely, there exists a measurable map $\mmap_t : \C^2 \rightarrow \P(\C_t^\N)$ such that, for every $k \in \N$ and every Borel set $B \subset \C_t^\N$, we have  
\[
\mmap_t\big(X^\tree_k,X^\tree_{\o}\big)(B) = \PP\big(  (X^\tree_{ki}[t])_{i \in \N} \in B \, | \, X^\tree_{\o}[t], X^\tree_k[t] \big) \ \ a.s.
\]
\end{enumerate}
\end{proposition}

The proof of Proposition \ref{pr:properties-GW} is given in Section \ref{se:proofs-GW-conditionalindependence} and relies on general definitions and properties of MRFs on finite graphs outlined in Section \ref{se:2MRFs}. We first study finite truncations of the UGW tree in Proposition \ref{pr:GW-conditionalindependence}, prove a version of this property on the truncated graph, and then carefully take limits. 

While Proposition \ref{pr:properties-GW}(ii) 
captures some of the symmetry of the UGW($\rho$) tree $\tree$, the next result, 
Proposition \ref{pr:invariance-GW}, provides one more crucial symmetry property
and is where unimodularity comes into play; this might be contrasted with the simpler symmetry considerations
used in the case of $\mathbb{T}_2$ as outlined in \eqref{loceq-symmetry} of Section \ref{subsub-linegraph}(ii). 
 Proposition \ref{pr:invariance-GW} below is where the measure change described in Remark \ref{intuition-GW} appears, which explains the form of $\gamma_t$ in Definition \ref{def-GWlocchar}.

Recall the definition of the space $\SQ(\X)$ from Section \ref{se:symmetricsequencespace}.  
For Proposition \ref{pr:invariance-GW} and  its proof,  
it is helpful to introduce some notation to emphasize when we are working with 
unordered vectors (elements of $\SQ(\X)$) versus ordered vectors. For a finite set $A$ and
a (ordered) vector $x_A = (x_v)_{v \in A} \in \X^A$, we write $\lan x_A \ran$ to denote the corresponding
element  (equivalence class) of $\SQ(\X)$. 
The canonical labeling scheme $\V$ introduced in Section \ref{subsub-labeling}
and  adopted in this section
carries with it a natural order, and we will find it helpful to use this notation $\langle \, \cdot \, \rangle$ when it is important to stress that we are dealing with an unordered vector.

\begin{proposition} \label{pr:invariance-GW}
Suppose Assumption \ref{assumption:A} holds, and suppose $\tree$ is a UGW($\rho$) tree, where $\rho \in \P(\N_0)$ has finite nonzero first moment. 
Let $t > 0$, and let $h : \C_t^2 \times \SQ(\C_t) \to \R$ be bounded and measurable. Suppose we are given
  a measurable function $\Xi_t : \C^2 \rightarrow \R$ that satisfies 
\begin{align*}
\Xi_t(X_{\o},X_1) = 1_{\{N_{\o}(\tree) \neq \emptyset\} } \frac{\E\left[\left. \frac{|N_{\o}(\tree)|}{|N_1(\tree)|}h(X_{\o}[t],X_1[t], \lan X_{N_{\o}(\tree_1)}[t] \ran ) \, \right| \, X_{\o}[t], \, X_1[t] \right]}{\E\left[\left. \frac{|N_{\o}(\tree)|}{|N_1(\tree)|} \, \right| \, X_{\o}[t], \, X_1[t] \right]}, \ \ a.s.
\end{align*}
Then, for each $k \in \N$,
\begin{align}
\Xi_t(X_k,X_{\o}) = \E\left[\left. h(X_k[t],X_{\o}[t], \,\lan X_{N_k(\tree)}[t] \ran ) \, \right| \, X_{\o}[t],X_k[t]\right], \ \ a.s., \text{ on } \ \{k \in \tree\}. \label{def:invariance-GW1}
\end{align}
\end{proposition}

The proof of Proposition \ref{pr:invariance-GW} is given in Section \ref{se:proofs-GW-invariance}.
It is worth noting that the statement of Proposition \ref{pr:invariance-GW} would be far less succinct  
if we did not define the SDE as in \eqref{statements:SDE} with the canonical labeling scheme.

\subsection{Limits of finite-graph systems} \label{se:localconvergence}

This section presents the natural application of our local equation to characterizing the limiting behavior of finite particle systems, drawing on our recent results in \cite{LacRamWu20a}. For a finite and possibly random graph $G$ with non-random vertex set $V$, we define $X^G=(X^G_v)_{v\in V}$ as the unique in law solution of the SDE
\begin{align}
dX^G_v(t) &= b(t,X^G_v,X^G_{N_v(G)})dt + \sigma(t,X^G_v)dW_v(t), \ \ v \in V, \text{ with } (X^G_v(0))_{v \in V}  \ \text{ i.i.d.} \sim \lambda_0. \label{SDE-gengraph}
\end{align}
Here $(W_v)_{v \in V}$ are independent Brownian motions, the initial law $\lambda_0 \in \P(\R^d)$ is given, and $N_v(G)$ denotes the set of vertices in $G$ which are adjacent to $v$. Moreover, we assume as always that the graph $G$, if random, is independent of $(W_v,X^G_v(0))_{v \in V}$. Note that under Assumption \ref{assumption:A}, as discussed thereafter, existence and uniqueness in law for the SDE \eqref{SDE-gengraph} hold by Girsanov's theorem. We define also the (global) empirical measure
\begin{align}
\mu^G &:= \frac{1}{|V|}\sum_{v \in V}\delta_{X^G_v}, \label{SDE-gengraph-empmeas}
\end{align}
which we view as a random element of $\P(\C)$.

Using Theorems 3.3 and 3.7 of \cite{LacRamWu20a}, we could now state a rather general theorem that applies
to any (random) graph sequence that \emph{converges in the local weak sense} to a UGW($\rho$) tree. Indeed, \cite[Theorem 3.7]{LacRamWu20a} shows that if $G_n$ \emph{converges in probability in the local weak sense} to a limiting graph $G$, then both $\mu^{G_n}$ and $X^{G_n}_{\o_n}$, where $\o_n$ is a uniformly random vertex in $G_n$, converge, with the limits characterized in terms of the root particle in the SDE system \eqref{SDE-gengraph} set on the limit graph $G$. When $G$ is a UGW($\rho$) tree, we then characterize this root particle via our local equation.
To avoid giving a full definition of local weak convergence of (marked) graphs (which can be found in \cite{LacRamWu20a}
  in Section 2.2.4, including  Definitions 2.8 and 2.10 therein, and Appendix A), 
we prefer not to state the most general result possible here,
and instead we focus on three noteworthy random graph models:
\begin{itemize}
\item The \Erdos graph $G \sim \G(n,p)$ is defined for $n \in \N$ and $p \in (0,1)$ by considering a graph with $n$
  vertices and  independently connecting each pair of distinct vertices with probability $p$ each.
\item The random $\kappa$-regular graph $G \sim \mathrm{Reg}(n,\kappa)$ is defined for $n \in \N$ by choosing a $\kappa$-regular graph (meaning each vertex has exactly $\kappa$ neighbors) uniformly at random from among all $\kappa$-regular graphs on $n$ vertices. It is well known that a $\kappa$-regular graph on $n$ vertices exists as long as $n\kappa$ is even and $n \ge \kappa+1$.
\item The configuration model $G \sim \mathrm{CM}(n,d^n)$, for any graphical sequence $d^n=(d^n_1,\ldots,d^n_n) \in \N^n$, is the uniformly random graph from among all graphs on $n$ vertices with degree sequence $d^n$; see \cite[Chapter 7]{van2009random} for more information. Of course if $d^n=(\kappa,\ldots,\kappa)$ then this reduces to the $\kappa$-regular tree.
\end{itemize}

Recall in the following theorems that $\L(Z)$ denotes the law of a random variable $Z$, and $\Rightarrow$ denotes convergence in law. The following results are all immediate corollaries of \cite[Theorem 3.7]{LacRamWu20a} (see also Examples 2.2, 2.3, and 2.4 therein) along with our Theorem \ref{th:statements-mainlocaleq}.

\begin{corollary}[\Erdos \!\!] \label{co:localconvergence-ER}
Suppose Assumption \ref{assumption:B} holds, and assume the initial distribution $\lambda_0$ has bounded support. For each $n \in \N$ suppose $G_n \sim \G(n,p_n)$ for some $p_n \in (0,1)$, and assume $\lim_{n\to\infty}np_n=\parm$ for some $\parm \in (0,\infty)$. Let $X^{G_n}$ and $\mu^{G_n}$ be as in \eqref{SDE-gengraph} and \eqref{SDE-gengraph-empmeas}, and let $\o_n$ denote a uniformly random vertex in $G_n$ for each $n$. Let $(Y_v)_{v \in \V_1}$ denote the unique in law solution of the $\mathrm{UGW}(\mathrm{Poisson}(\parm))$ local equation given by Theorem \ref{th:statements-mainlocaleq}. Then $X^{G_n}_{\o_n} \Rightarrow Y_{\o}$ in $\C$, and $\mu^{G_n} \Rightarrow \L(Y_{\o})$ in $\P(\C)$.
\end{corollary}

\begin{corollary}[Random regular graph] \label{co:localconvergence-reg}
Suppose Assumption \ref{assumption:B} holds,  and assume the initial distribution $\lambda_0$ has bounded support. Let $\kappa \ge 2$ be an integer. For each even number $n \ge \kappa+1$ suppose $G_n \sim \mathrm{Reg}(n,\kappa)$. Let $X^{G_n}$ and $\mu^{G_n}$ be as in \eqref{SDE-gengraph} and \eqref{SDE-gengraph-empmeas}, and let $\o_n$ denote a uniformly random vertex in $G_n$ for each $n$. Let $(Y_v)_{v \in \V_1}$ denote the unique in law solution of the $\mathbb{T}_\kappa$ local equation given by Corollary \ref{cor-regtree}. Then $X^{G_n}_{\o_n} \Rightarrow Y_{\o}$ in $\C$, and $\mu^{G_n} \Rightarrow \L(Y_{\o})$ in $\P(\C)$.
\end{corollary}

\begin{corollary}[Configuration model] \label{co:localconvergence-CM}
Suppose Assumption \ref{assumption:B} holds, and assume the initial distribution $\lambda_0$ has bounded support. For each $n \in \N$ suppose $d^n=(d^n_1,\ldots,d^n_n)$ is a graphical sequence, and let $G_n \sim \mathrm{CM}(n,d^n)$. Assume $\frac{1}{n}\sum_{k=1}^n\delta_{d^n_k}$ converges weakly to some $\rho \in \P(\N_0)$ with finite nonzero first moment and finite second moment, and assume also that the first moments converge: $\frac{1}{n}\sum_{k=1}^n d^n_k \to \sum_{k=0}^\infty k\rho(k)$. 
Let $X^{G_n}$ and $\mu^{G_n}$ be as in \eqref{SDE-gengraph} and \eqref{SDE-gengraph-empmeas}, and let $\o_n$ denote a uniformly random vertex in $G_n$ for each $n$. Let $(Y_v)_{v \in \V_1}$ denote the unique in law solution of the $\mathrm{UGW}(\rho)$ local equation given by Theorem \ref{th:statements-mainlocaleq}. Then $X^{G_n}_{\o_n} \Rightarrow Y_{\o}$ in $\C$, and $\mu^{G_n} \Rightarrow \L(Y_{\o})$ in $\P(\C)$.
\end{corollary}

Note that in each of these results we assert that the sequence of  random empirical measures $\{\mu^{G_n}\}_{n \in \N}$
converges in law to a \emph{non-random limit} $\L(Y_{\o})$. By standard propagation of chaos arguments (see \cite{sznitman1991topics} or \cite[Lemma 2.12]{LacRamWu20a}), it follows that if $k \in \N$ is fixed and if $v^1_n,\ldots,v^k_n$ are $k$ independent uniformly random vertices in $G_n$, then $\L(X^{G_n}_{v^1_n},\ldots,X^{G_n}_{v^k_n})$ converges weakly to the $k$-fold product measure $\L(Y_{\o}) \times \cdots \times \L(Y_{\o})$ as $n\to\infty$. The same is then true if $(v^1_n,\ldots,v^k_n)$ is chosen uniformly at random from among the ${n \choose k}$ $k$-tuples of \emph{distinct} vertices.
However, it is important to emphasize  that unlike mean-field limits, in our setting this convergence does not
  hold for any  arbitrary chosen finite  set  of vertices. 
  In particular,  if $v_n^1  = \o_n$ and  $v_n^1$  is a neighbor of $\o_n$ chosen uniformly at random  (assuming one exists,
  else set $v_n^2$ to be a uniformly at random vertex from $V \setminus \o_n$), 
  then the laws of $X^{G_n}_{v_n^1}$ and $X^{G_n}_{v_n^2}$ are not asymptotically independent but remain correlated in the  limit, with the
 limiting correlations captured by the  local  equation.

 \section{Proof of Theorem  \ref{th:statements-mainlocaleq}}
\label{se:pfmainthm}

This section is devoted to the proof of Theorem \ref{th:statements-mainlocaleq} using 
the results stated  in Propositions \ref{pr:properties-GW} and
\ref{pr:invariance-GW}.  Throughout, let $(X^\tree_v)_{v \in \V}$ be a solution to the SDE system
  \eqref{statements:SDE}.
 In Section \ref{se:pf:existence-localGW}
we first verify that the marginal $(X^\tree_v)_{v \in \V_1}$ is a weak solution of the local equation,  in particular establishing existence of a solution to the local equation.
  Then, in Section  \ref{se:pf:existence-localGW}, we show that the local equation is
  well-posed in the sense that it  has a
  unique weak solution.  
In the proofs we will use the notation $\EE$ to denote the Doleans exponential, or
\begin{align}
\EE_t(M) := \exp(M_t - \tfrac12 [M]_t), \quad t  \geq 0, \label{def:doleans-exponential}
\end{align}
for a continuous local martingale $M$, where $[M]$ denotes the (optional) quadratic variation process of $M$.
We also recall that $H$ denotes the relative entropy functional defined in \eqref{def-relentropy}.

\subsection{Verification Result}  \label{se:pf:existence-localGW}

We prove in this section the first claim of Theorem \ref{th:statements-mainlocaleq}, which asserts that the law of the root neighborhood particles $(X_v)_{v \in \V_1}$ 
provides a weak solution of the local equation of Definition \ref{def-GWlocchar}.

We first state a fairly standard integrability estimate, which explains the need for the average $1/|A|$ in the linear growth Assumption (\ref{assumption:A}.1). We defer the proof to Appendix \ref{se:ap:entropy-proof}, as it is similar to \cite[Lemma 5.1]{LacRamWu19b}.
 For any random tree $\tree$, let $(\widehat{X}_v)_{v \in \V}$ denote the unique in law (by Assumption (\ref{assumption:A}.2b)) solution of the driftless SDE  system
\begin{align}
d\widehat{X}_v^\tree(t) = 1_{\{v \in \tree\}}\sigma(t,\widehat{X}^\tree_v)dB_v(t), \quad v \in \V, \label{pf:existence-driftless}
\end{align}
where $(\widehat{X}^\tree_v(0))_{v \in \V}$ are i.i.d.\ with law $\lambda_0$, and as usual the tree, initial conditions, and Brownian motions are independent. Recall in the following that $\L(Z)$ denotes the law of a random variable $Z$, and $x_A = (x_v)_{v \in A}$ denotes a sub-configuration of $x=(x_v)_{v \in \V}$ for $A \subset \V$.

\begin{lemma} \label{le:lingrowth}
Suppose Assumption \ref{assumption:A} holds. For each $T \in (0, \infty)$ there exists a constant $C^*_T < \infty$ such that, for any random tree $\tree \subset \V$, letting $X^\tree$ be the solution of \eqref{statements:SDE}, we have
\begin{align}
\sup_{v \in \V}\E[\|X_v^\tree\|^2_{*,T} \,|\, \tree] \le C^*_T, \ \ a.s., \label{def:2ndmomentbound}
\end{align}
and also, for any finite set $A \subset \V$,
\begin{align}
H\big( \L(X_A^\tree[T]) \,\big|\, \L(\widehat{X}_A^\tree[T])\big) &\le C^*_T(1 + |A|), \label{def:entropybound1} \\
H\big( \L(\widehat{X}_A^\tree[T]) \,\big|\, \L(X_A^\tree[T])\big) &\le C^*_T(1 + |A|). \label{def:entropybound2} 
\end{align}
\end{lemma}

Now, we work for the rest of Section \ref{se:pf:existence-localGW} on a filtered probability space $(\Omega,\F,\FF,\PP)$, supporting a UGW($\rho$) tree $\tree$, independent $d$-dimensional Brownian motions $(W_v)_{v \in \V}$, and continuous $d$-dimensional processes $(X^\tree_v)_{v \in \V}$ satisfying the SDE system \eqref{statements:SDE}. As always we assume $\tree$, $(W_v)_{v \in \V}$, and $(X^\tree_v(0))_{v \in \V}$ are independent, and $(X^\tree_v(0))_{v \in \V}$ are i.i.d.\ with common law $\lambda_0$.  The offspring distribution $\rho \in \P(\N_0)$ has finite nonzero first moment and finite second moment.
For ease of notation, for the rest of Section \ref{se:pf:existence-localGW} we omit the superscript by writing $(X_v)_{v \in \V} = (X^{\tree}_v)_{v \in \V}$ and $(\widehat{X}_v)_{v \in \V} = (\widehat{X}^{\tree}_v)_{v \in \V}$.
The driftless process $\widehat{X} = (\widehat{X}^\tree_v)_{v  \in \V}$ defined in \eqref{pf:existence-driftless} may live on a different probability space that we do not specify.

\begin{remark}
  \label{rem-treerecovery}
  The dynamics \eqref{statements:SDE} include the ``fictional'' particles $v \notin \tree$ in such a way that the random tree
  $\tree$ can be recovered from $(X_v[t])_{v \in \V}$ for any $t > 0$. Indeed, almost surely, $v \notin \tree$
  if and only if there exists an interval on which $t \mapsto X_v(t)$ is constant. (Note that this holds because the diffusion coefficient is assumed non-degenerate.) More precisely, $\tree$ is measurable with respect to the ``just after time zero'' $\sigma$-field,  or
\begin{align}
\{v \in \tree\} \in \bigcap_{t > 0}\sigma(X_v(s) :  s \le t), \quad\quad \text{a.s. for each } \ v \in \V. \label{def:Tmeasurable}
\end{align}
 Here ``a.s." means that the event $\{v \in \tree\}$ belongs to the completion of the $\sigma$-field appearing on the right-hand side. 
Moreover, there exists a deterministic mapping  $\detfn : \C \rightarrow \{0,1\}$, measurable with respect to $\cap_{t > 0}\F^\C_t$ where $(\F^\C_t)_{t \ge 0}$ is the canonical filtration on $\C$, such that
\begin{align}
1_{\{v \in \tree\}} = \detfn (X_v), \quad\quad \text{a.s. for each } \ v \in \V. \label{def:Tmeasurable-strong}
\end{align}
In particular, this function $\detfn$ does not depend on $v$. These observations will be exploited several times throughout this section.
\end{remark}

 The proof is decomposed into several steps. \\
{\bf Step 1.} The first step of the proof will be to project onto the root neighborhood $\V_1$ using the projection theorem (Theorem \ref{th:brunickshreve}). 
It follows from \eqref{statements:SDE} that  $(X_v)_{v \in \V_1}$ satisfies
\begin{align*}
dX_{\o}(t) & = b(t,X_{\o},X_{N_{\o}(\tree)}) \, dt + \sigma(t,X_{\o}) \, dW_{\o}(t), \\
dX_k(t) & = 1_{\{k \in \tree_1\}}\Big(b(t,X_k,X_{N_k(\tree)}) \, dt + \sigma(t,X_k) \, dW_k(t)\Big), \quad k \in \V_1 \setminus \{{\o}\},
\end{align*}
where we write $\tree_1:=\tree\cap\V_1$ for the first generation of $\tree$.
By Theorem \ref{th:brunickshreve}, 
by extending the probability space if necessary,
we may find independent $d$-dimensional Brownian motions $(B_v)_{v \in \V_1}$
such that 
\begin{equation}
  dX_v (t) = \widetilde{b}_v(t,X_{\V_1}) \, dt +  \widetilde{\sigma}_v(t,X_{\V_1}) \, dB_v(t),
  \ \ v \in \V_1,   \label{pf:existencelocal-SDE-Y}
\end{equation}
where $\widetilde{b}_v: \R_+ \times \C^{\V_1} \mapsto \R^d$ and $\widetilde{\sigma}_v: \R_+ \times \C^{\V_1} \mapsto \R^{d \times d}$ are any progressively measurable functions satisfying
\begin{align}
  \label{proj-b}
  \widetilde{b}_v(t,X_{\V_1}) &=  \E \left[ 1_{\{v \in \tree_1\}}b(t,X_{v},X_{N_{v}(\tree)})  \, | \, X_{\V_1}[t]\right],  \\
  \label{proj-sigma}
  \widetilde{\sigma}_v\widetilde{\sigma}_v^\top(t,X_{\V_1}) &= \E\left[1_{\{v \in \tree_1\}} \sigma \sigma^\top (t,X_{v}) \,  | \, X_{\V_1}[t]\right],  
\end{align}
a.s.\ for a.e.\ $t > 0$.  Such progressively measurable functions always exist by Lemma \ref{le:optprojection}. 
Now,  the functions $\widetilde{b}_{\o}$ and $(\widetilde{\sigma}_v)_{v \in \V_1}$ can be simplified because the corresponding integrands are $X_{\V_1}[t]$-measurable.
Indeed, because $X_{N_{\o}(\tree)}$ and $\tree_1$ are
    $X_{\V_1}[t]$-measurable for each $t > 0$ (as a consequence of
 Remark \ref{rem-treerecovery}),    we may take
\begin{equation}
  \widetilde{b}_{\o}(t,X_{\V_1}) = b(t,X_{\o},X_{N_{\o}(\tree_1)}), \qquad
  \widetilde{\sigma}_v(t,X_{\V_1}) = 1_{\{v \in \tree_1 \}}\sigma (t, X_v), \quad v \in \V_1.
\end{equation}

\noindent 
{\bf  Step 2:}  Next, we simplify the form of $\widetilde{b}_{v}$ for $v \in \V_1\setminus \o,$ using symmetry and conditional independence results. Noting that $\V_1\setminus \{\o\}$  can be identified with $\N$,  
for a given $k \in \N$, we first apply the conditional independence result of Proposition \ref{pr:properties-GW}(i) to deduce that all particles except $\o$ and $k$ may be safely omitted from the conditioning in the definition of $\widetilde{b}_k$. That is, recalling also that $\{k \in \tree_1\}$ is $X_k[t]$-measurable (again by Remark \ref{rem-treerecovery}), we have
\begin{align*}
\widetilde{b}_k(t,X_{\V_1}) &= 1_{\{k \in \tree_1\}} \E\left[b(t,X_{k},X_{N_{k}(\tree)}) \, | \, X_{\o}[t], \, X_k[t]\right],  \ \ a.s., \ \ a.e.\ t > 0.
\end{align*}

Now, fix $t > 0$.  Since $b$ is progressively measurable, 
  there exists  a  measurable function $h: \C_t \times \SQ(\C_t) \mapsto \R$ such that
  $b(t, x, \bar{x}) = h (x[t], \bar{x}[t])$ for $x \in \C, \bar{x} \in \SQ(\C)$.
  Then, on $\{k \in \tree_1\}$,
  $\widetilde{b}_k(t,X_{\V_1})$ is equal to the right-hand side of \eqref{def:invariance-GW1} with this choice of $h$.
  Although $h$ is not bounded as is  required in Proposition \ref{pr:invariance-GW}, both $h(X_{\o}[t],\langle X_{N_{\o}}[t] \rangle)$ and $h(X_{k}[t],\langle X_{N_{k}}[t] \rangle)$ are square-integrable due to the linear growth of $b$ from 
     Assumption (\ref{assumption:A}.1) and Lemma \ref{le:lingrowth}, and we know also that $|N_{\o}(\tree)|$ is square-integrable as we assumed $\rho$ has finite second moment.  Hence, by truncating $h$ and taking limits, we easily extend the validity of the formula in Proposition \ref{pr:invariance-GW} to cover such an $h$.  
Ultimately, we deduce that 
  \begin{equation}
    \label{eq-tilbk}
    \widetilde{b}_k (t, X_{\V_1}) = \gamma_t (X_k, X_{\o}),   \quad \mbox{ on } \{k \in \tree_1\}, 
  \end{equation} 
  where   $\gamma_t:\C^2 \mapsto \R^d$ is a progressively measurable function satisfying
\begin{align}	\label{gammat}
\gamma_t(X_{\o},X_1) = \frac{\E\left[\left. \frac{|N_{\o}(\tree)|}{|N_1(\tree)|}b(t,X_{\o},X_{N_{\o}(\tree_1)}) \, \right| \, X_{\o}[t], \, X_1[t] \right]}{\E\left[\left. \frac{|N_{\o}(\tree)|}{|N_1(\tree)|} \, \right| \, X_{\o}[t], \, X_1[t] \right]} \qquad \text{on } \{N_{\o}(\tree_1) \neq \emptyset\}, 
\end{align}
and  $\gamma_t (X_{\o}, X_1) = b(t, X_{\o}, \onepoint)$ on  $\{N_{\o}(\tree_1) = \emptyset\}$, where we recall that $\onepoint$ denotes the element of the one-point space $\C^0$. 
 Note that $|N_1(\tree)| \geq 1$ a.s., $\E[|N_{\o}(\tree)|^2] < \infty$ and $\E\left[\int_0^T |b(t,X_{\o}[t],  X_{N_{\o}}[t])|^2 dt \right] < \infty$, which together imply
\[
\E\left[1_{\{N_{\o}(\tree) \neq \emptyset\}}\frac{|N_{\o}(\tree)|}{|N_1(\tree)|} \int_0^T |b(t,X_{\o}[t],  X_{N_{\o}}[t])| dt \right] < \infty,
\]
for each $T \in (0,\infty)$.  
Since $X$ is continuous,
the existence of a progressively measurable version of $(t,x_{\o},x_1) \mapsto \gamma_t(x_{\o},x_1)$ is then guaranteed by Lemma \ref{le:optprojection}. 

{\ }

\noindent 
{\bf Step 3.} 
It remains to check that we have all of the ingredients required by Definition \ref{def-GWlocchar} for a solution of the local equation. We begin with the integrability condition stated as property (10) in Definition \ref{def-GWlocchar}.  Note that Lemma \ref{le:lingrowth} and the linear growth of $b$ from Assumption (\ref{assumption:A}.1) ensure that, by Jensen's inequality and \eqref{gammat},  
\begin{align}
\E\left[1_{\{k \in \tree\}}\int_0^T|\gamma_t(X_k,X_{\o})|^2dt\right] &\le \E\left[1_{\{k \in \tree\}}\int_0^T|b(t,X_k,X_{N_k(\tree)})|^2dt\right] < \infty. \label{pf:verif-girsanov-1}
\end{align}
Next, recall from Remark \ref{rem-treerecovery} that $\{N_{\o}(\tree) \neq \emptyset\}$ is $X_{\o}[t]$-measurable for $t > 0$. Applying the conditional Jensen's inequality,  and invoking \eqref{gammat},  Assumption (\ref{assumption:A}.1) and
Lemma \ref{le:lingrowth} yields
\begin{align}
\E &\left[1_{\{N_{\o}(\tree) \neq \emptyset\}}\frac{|N_{\o}(\tree)|}{|N_1(\tree)|}\int_0^T|\gamma_t(X_{\o},X_k)|^2dt\right] \nonumber \\
&\qquad \quad \le \E\left[1_{\{N_{\o}(\tree) \neq \emptyset\}}\frac{|N_{\o}(\tree)|}{|N_1(\tree)|}\int_0^T|b(t,X_{\o},X_{N_{\o}(\tree_1)})|^2dt\right] \nonumber \\
&\qquad\quad \le 3C^2_T T\E\left[1_{\{N_{\o}(\tree) \neq \emptyset\}}\frac{|N_{\o}(\tree)|}{|N_1(\tree)|}\left(1 + \|X_{\o}\|_{*,T}^2 + \frac{1}{|N_{\o}(\tree)|}\sum_{k \in N_{\o}(\tree)}\|X_k\|_{*,T}^2 \right)\right] \nonumber \\
&\qquad\quad \le 3C^2_T T(1+2C^*_T)\E\left[1_{\{N_{\o}(\tree) \neq \emptyset\}}\frac{|N_{\o}(\tree)|}{|N_1(\tree)|}\right] < \infty, \label{pf:verif-girsanov-2}
\end{align}
where $C_T < \infty$ and $C^*_T < \infty$ are the constants from Assumption (\ref{assumption:A}.1) and Lemma \ref{le:lingrowth}, respectively. Recalling that $\gamma_t(X_{\o},X_1)=b(t,X_{\o},\onepoint)$ on $\{N_{\o}(\tree)=\emptyset\}$ we deduce from \eqref{pf:verif-girsanov-1} and \eqref{pf:verif-girsanov-2} that the following two integrals are a.s.\ finite, for each $k \in\N$:
\begin{align*}
&\int_0^T|\gamma_t(X_{\o},X_k)|^2dt, \quad \int_0^T|\gamma_t(X_k,X_{\o})|^2dt.
\end{align*}
The finite entropies of Lemma \ref{le:lingrowth} ensure that the laws of $(X_{\o},X_k)$ and $(\widehat{X}_{\o},\widehat{X}_k)$ are equivalent (i.e., mutually absolutely continuous) for each $k\in\N$, and therefore the following integrals are also a.s.\ finite:
\begin{align*}
&\int_0^T|\gamma_t(\widehat{X}_{\o},\widehat{X}_k)|^2dt, \quad \int_0^T|\gamma_t(\widehat{X}_k,\widehat{X}_{\o})|^2dt.
\end{align*}
Along with the definition of $\gamma_t$  (see \eqref{gammat} and the subsequent line), this verifies both
  properties  (7) and (10)  of  Definition \ref{def-GWlocchar}, with $Y = X$.

Finally, by enlarging the probability space if necessary, let $\Nvar^{\mathrm{ext}}$ be an $\F_0$-measurable $\N_0$-valued random variable with law $\widehat\rho$, independent of $(\tree,(X_v(0))_{v \in \V})$. Define $\Nvar:= |N_1(\tree)| - 1$ on the event $\{N_{\o} (\tree_1) \neq \emptyset\}$, and on the complementary event $\{N_{\o}(\tree_1) = \emptyset\}$ define $\Nvar := \Nvar^{\mathrm{ext}}$.
This way, using the definition of the UGW($\rho$) tree $\tree$, one may easily check that $\Nvar$ has law $\widehat\rho$,
$\tree_1$ is the first generation of a UGW($\rho$) tree, $(X_v(0))_{v \in \V}$ are  i.i.d.\ and $\F_0$-measurable with law $\lambda_0$, and  moreover,  $\Nvar$, $\tree_1$, and $(X_v(0))_{v \in \V}$ are independent.
  This verifies properties (1)--(3), (6), and (8) of Definition \ref{def-GWlocchar}.
  (The definition of $\Nvar$ on $\{N_{\o}(\tree_1) = \emptyset\}$ is made in this way for the sole purpose of meeting the independence requirement of Definition \ref{def-GWlocchar}(8), and $\Nvar^{\mathrm{ext}}$ serves no other purpose.)
  Combining relations \eqref{pf:existencelocal-SDE-Y}-\eqref{gammat},
  we see that  the stochastic equations
  \eqref{statements:localequationGW} are satisfied with $Y_k = X_k$ for all $k \in \V_1$, and thus  
  properties (4), (5),  and (9) of Definition \ref{def-GWlocchar} hold. 
Putting this together, we see that $(X_v)_{v \in \V_1}$ is a weak solution of the UGW($\rho$)
local equation with initial law $\lambda_0$, as in Definition \ref{def-GWlocchar}.

\subsection{Proof of well-posedness of the UGW($\rho$) local equation} 
\label{se:pf:uniqueness-localGW} 
  
Fix $\rho \in \P(\N_0)$ with finite first and second moments.
  As briefly described in Section \ref{subsub-linegraph} in the simplest case of a $2$-regular
  tree, the basic idea behind the proof of uniqueness is to use the weak solution 
to the local equation to construct a solution to the infinite particle system
\eqref{statements:SDE}  on the UGW($\rho$) tree $\tree$, and then invoke uniqueness (in law)
of the latter to deduce that of the former.  However, the construction is more 
involved when $\kappa > 2$ and substantially more complicated in the case of the random UGW tree.
To make the proof more transparent, we first provide an outline and introduce some common notation in Section \ref{subs-uniqueoutline}, then prove the main technical lemmas in Section \ref{subs-uniqueproof}, and finally, in Section \ref{subs-uniquecomplete}, show that the uniqueness property in Theorem \ref{th:statements-mainlocaleq} is a consequence of these lemmas.

\subsubsection{Outline of proof and some common terminology}
\label{subs-uniqueoutline}

Let  $((\Omega,\F,\FF,\PP),\tree_1, \gamma,(B_v,Y_v)_{v \in \V_1}, \Nvar)$ be any weak solution
to the UGW($\rho$) local equation with initial law $\lambda_0$, as specified in Definition
\ref{def-GWlocchar}.   Due to properties (2), (3), (4), and (8) of Definition \ref{def-GWlocchar},
  by extending the probability space if needed,
  we can assume without loss of generality that  $(\Omega,\F,\FF,\PP)$  
  also supports a UGW($\rho$) tree $\tree$, independent of the standard
  $d$-dimensional $\FF$-Brownian motions 
  $(B_v)_{v \in \V}$, and i.i.d.\ initial conditions
  $(Y_v(0))_{v \in \V}$,  
such that $\tree_1 = \tree \cap \V_1$ and  $\Nvar+1 = |N_1(\tree)|$ on $\{N_{\o}(\tree) \ne \emptyset\}$.
Next, again on the event $\{N_{\o} (\tree) \neq \emptyset\}$, we aim to 
  extend the local solution to $\V_2$
  in such a way that the law of 
  the particle system on the random tree  $\tree_2 := \tree \cap \V_2$ of depth $2$
  is consistent with the $\tree_2$-marginal of the interacting particle system
  \eqref{statements:SDE}, where recall that, for any $n \in \N$,  $\V_n$ was defined in \eqref{def-vn}. 
  For this it suffices to specify the conditional joint law of  the states of vertices in
  $\tree_2\setminus \V_1$  given $Y_{\V_1}$.  
  In view of the second-order MRF  property and exchangeability
  (as encapsulated in Proposition \ref{pr:properties-GW}), 
  this is equal to the product
  of the conditional joint laws of the states of the offspring 
  of each $i \in \tree_1\setminus \{\o\}$, given the states of the vertices $i$ and $\o$,
  and each of these conditial laws  is  identical in form.

Now, in the case when $\tree = \T_{\kappa}$ for some $\kappa \geq 2$, 
this conditional law can be identified from the weak solution to the local equation since, by the symmetry of the
tree, it has the same form as the conditional law, given the trajectories of vertices
$\o$ and $1$, of the remaining  children $\tree_1 \setminus \{{\o}, 1\}  = \{2, \ldots, \kappa\}$
of the root $\o$,  except that the roles of $\o$ and $1$ are now reversed, since $1$ now acts as the new root (see Figure \ref{fig:regulartree}).
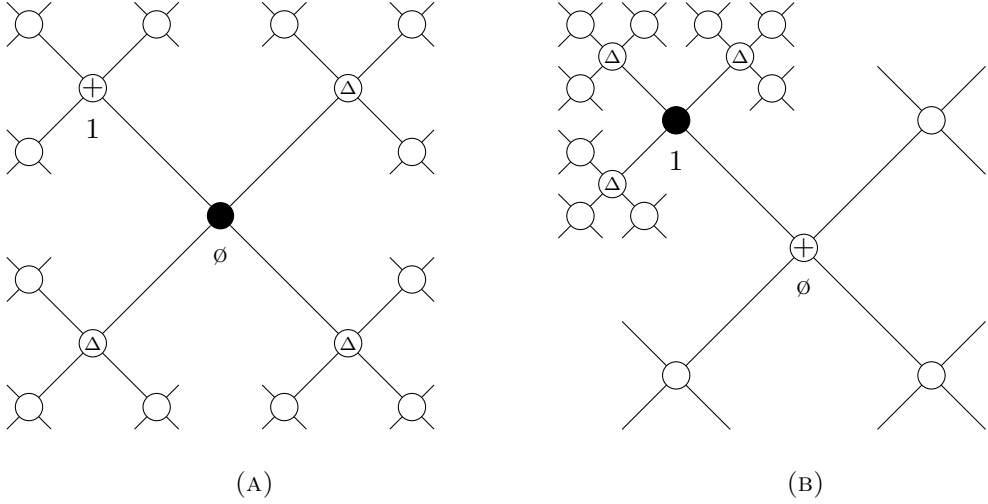
\begin{figure} 
\tikzstyle{level 1}=[sibling angle=90,level distance=8mm]
\tikzstyle{level 2}=[sibling angle=90,level distance=4mm]
\tikzstyle{level 3}=[sibling angle=90,level distance=2mm]
\tikzstyle{level 4}=[sibling angle=90,level distance=2mm]
\tikzstyle{edge from parent}=[segment angle=10,draw]
\begin{subfigure}[b]{.45\linewidth}
\begin{tikzpicture}[grow cyclic,shape=circle,cap=round,scale=3]
\node[label=below:$\o$,fill,scale=1] {} child [color=\A] foreach \A in {black,black,black}
    { node[draw,label=center:{\tiny$\Delta$},scale=1] {} child foreach \B in {black,black,black}
        { node[draw,scale=1] {} child  foreach \C in {black,black,black}
            { node {} 
            }
        }
    } child
    { node[draw,label=center:$+$,label=below:$1$,scale=1] {} child foreach \B in {black,black,black}
        { node[draw,scale=1] {} child  foreach \C in {black,black,black}
            { node {} 
            }
        }
    };
\end{tikzpicture}
\caption{}
  \end{subfigure}
\begin{subfigure}[b]{.45\linewidth}
\begin{tikzpicture}[grow cyclic,shape=circle,cap=round,scale=3]
\node[draw,label=below:$\o$,label=center:$+$,scale=1] {} child [color=\A] foreach \A in {black,black,black}
    { node[draw,scale=1] {} child foreach \B in {black,black,black} 
    	{ node{}
    	}
    } child
    { node[draw,fill,label=below:$1$,scale=1] {} child foreach \B in {black,black,black}
        { node[draw,label=center:{\tiny$\Delta$},scale=1] {} child  foreach \C in {black,black,black}
            { node[draw,scale=1] {} child  foreach \D in {black,black,black}
                { node {} 
                }
            }
        }
    };
\end{tikzpicture}
\caption{}
\end{subfigure}
\caption{The case $\tree = \T_4$, relating the conditional law  $\Lmc(Y_{\Delta} \in \cdot \, | \, Y_{\bullet}, Y_{+})$ when (A) the tree is rooted at `$\bullet$' and  (B) the tree is rerooted at `$+$'.}
\label{fig:regulartree}
 \vspace{-0.2in}
\end{figure}

In the case of the UGW($\rho$) tree, while the conditional joint laws are the same {\em given} the structure
of the tree, re-rooting the tree at  $1$ changes the distribution of the tree.
To account for this, we 
define a new ``tilted'' measure $\QQ$ on  $(\Omega, \F, \FF)$ via the relation 
\begin{align}
  \label{QQPP}
\frac{d\QQ}{d\PP} = \frac{|N_{\o}(\tree)|}{|N_1(\tree)|}1_{\{N_{\o}(\tree_1) \neq \emptyset\}}  + 1_{\{N_{\o}(\tree) = \emptyset\}}.
\end{align}
The fact that this defines a true probability measure $\QQ$ is justified in Lemma \ref{lem-comeas} below.

  We then  characterize the  joint law of $(Y_{\o}, Y_1)$ under this tilted
  measure $\widetilde{\PP}$ in Lemma \ref{lem-comeas},
and then  use the unimodularity of the tree (in particular, Proposition \ref{pr:invariance-GW}) to 
compute the conditional law on each time interval $[0,t]$  of
the trajectories of the neighborhood $N_{\o}(\tree)$ of the root given those of $\o$ and $1$
in  Lemma \ref{lem-claw}.
 Using this conditional law, which is denoted by $Z_t$, 
we extend the particle system to $\V_2$, and recursively to $\V_n$,  and denote
the latter law as $Q^n \in \P(C^{\V_n})$; see \eqref{pf:GW-defQn}. 
Finally, in Proposition \ref{lem-ext} we show that the family
$\{Q^n\}$ is consistent, in the sense that the projection of $Q^n$ to $\C^{\V_k}$ coincides with $Q^k$ for $k < n$, and that its unique extension to a law $Q \in \P(\C^{\V})$ coincides
with the unique law of a weak solution to the infinite particle  system  \eqref{statements:SDE}.

We close this discussion by  introducing some additional  notation that will be used
throughout the proof. 
Let $\nu := \PP \circ Y_{\V_1}^{-1} \in \P(\C^{\V_1})$ denote the law of (the $Y$-marginal of) the weak solution of the UGW ($\rho$) local equation, and define 
the corresponding ``tilted''
measure $\widetilde{\nu} \in \P(\C^{\V_1})$ by $\widetilde{\nu} := \QQ \circ Y_{\V_1}^{-1}$.
In other words, letting $\E^{\PP}$ denote expectation with respect to $\PP$, $\widetilde\nu$ is defined by the Radon-Nikodym derivative
\begin{align*}
\frac{d\widetilde\nu}{d\nu}(Y_{\V_1}) = \E^{\PP}\left[ \left. \frac{|N_{\o}(\tree)|}{|N_1(\tree)|}1_{\{N_{\o}(\tree) \neq \emptyset\}}  + 1_{\{N_{\o}(\tree_1) = \emptyset\}} \, \right| \, Y_{\V_1}\right],
\end{align*}
though we will make no use of this precise form. 
Also, throughout, to compute various laws and conditional laws, it will be convenient to introduce
some reference measures.
For this, we introduce again the solution 
 $(\widehat{X}_v)_{v \in \V}$ to the driftless SDE system (omitting the superscript $\tree$) 
\begin{align}
d\widehat{X}_v(t) = 1_{\{v \in \tree\}}\sigma(t,\widehat{X}_v)dB_v(t),  \quad v \in \V, \quad t > 0. \label{def:driftlessSDE}
\end{align}
Note that this SDE is unique in law due to Assumption (\ref{assumption:A}.2b).  
We also  introduce the canonical probability spaces $(\Omega^n,\F^n,\FF^n,P^{*,n})$ 
to be used throughout the proof. Here, $\Omega^n=\C^{\V_n}$, $\F^n$ is the Borel $\sigma$-field, and $\FF^n=(\F^n_t)_{t \ge 0}$ is the natural right-continuous filtration generated by the canonical coordinate processes, which
are  denoted  by $(X_v)_{v \in \V_n}$, and 
  $P^{*,n} := \PP \circ \widehat{X}_{\V_n}^{-1}$,  
and $\widetilde{P}^{*,n} := \QQ \circ \widehat{X}_{\V_n}^{-1}$ serve as references
measures that represent   the laws of the first $n$ generations of the processes defined in \eqref{def:driftlessSDE} under
  the probability measures $\PP$ and $\QQ$, respectively.
  We define $\tree_n \subset \V$ as the random tree with vertex set $\{v \in \V_n : \detfn(X_v)=1\}$, where $\detfn$ is given as in Remark \ref{rem-treerecovery}. In this way, $\tree_n$ agrees in law with $\tree \cap \V_n$, the height-$n$ truncation of the UGW($\rho$) tree $\tree$. To be clear, $(X_v)_{v \in \V_n}$ and $\tree_n$ live on the canonical space $\Omega^n$, whereas the other random variables such as $(\tree,\widehat{X},Y)$ are defined on $\Omega$.

We make special note of the conventions we  use to help the reader keep track of the various notations. 
We use a tilde for measures associated with the measure change, namely $\QQ$ and its descendants $\widetilde{P}^{*,n}$ and $\widetilde\nu$. The superscripts $*$ and $n$ on $P^{*,n}$ and $\widetilde{P}^{*,n}$ indicate that these measures are to be viewed as reference measures on the canonical space associated with $n$ generations $\Omega^n$. Lastly, the letter $\nu$ (and its decorated versions) will refer to measures constructed from the given solution $Y_{\V_1}$ of the local equation.

\begin{remark}
  It is worth emphasizing again,  as in Remark \ref{re:GWsimplifies-regular}, 
  how the argument simplifies when the tree is the deterministic $\kappa$-regular tree, i.e.,
  the UGW($\rho$) tree with $\rho = \delta_{\kappa}$ for an integer $\kappa \ge 2$. In this case, we have $\widehat\rho = \delta_{\kappa - 1}$, and we have deterministically $|N_v(\tree)| = \kappa$ for all $v \in \V$. In this case, $\QQ=\PP$, $\widetilde\nu=\nu$, and $\widetilde{P}^{*,n}=P^{*,n}$. On a first reading it may help to keep these substitutions in mind.
\end{remark}

  \subsubsection{Details of the Proof}
  \label{subs-uniqueproof}

 Once again, we break down the detailed justification  into three steps.  \\
   {\bf Step 1.}
  Our first goal is  to identify the marginal law of  $(Y_{\o}, Y_1)$ under the
  tilted  measure $\QQ$  defined in \eqref{QQPP}.  Specifically, recalling the definitions of $\nu$,  $\widetilde{\nu}$, the 
   reference measures and canonical processes introduced in 
  the last section, 
 we  define the marginal laws 
\begin{align*}
\nu^{{\o},1} &= \nu \circ (X_{\o},X_1)^{-1} = \PP \circ (Y_{\o},Y_1)^{-1}, \\
\widetilde\nu^{{\o},1} &= \widetilde\nu \circ (X_{\o},X_1)^{-1} = \QQ \circ (Y_{\o},Y_1)^{-1}, \\
\widetilde{P}^{*,{\o},1} &= \widetilde{P}^{*,1} \circ (X_{\o},X_1)^{-1},
\end{align*}
which are all elements of $\P(\C^{\{\o,1\}})$. 
We start with a lemma that uses the projection theorem (Theorem \ref{th:brunickshreve})
  to characterize the law $\widetilde{\nu}^{\o,1}$ as the weak solution to an SDE.

\begin{lemma}
  \label{lem-comeas}
  The measure  $\QQ$ specified in \eqref{QQPP} defines a probabilty measure
  on $(\Omega, \F, \FF)$. Moreover, by extending the probability space  $(\Omega, \F, \FF, \QQ)$ if necessary, we may find 
independent $d$-dimensional standard $\QQ$-Brownian motions $(\widetilde{W}_v)_{v \in \{\o, 1\}}$
   such that $(Y_{\o}, Y_1)$ satisfies the following SDE system:
\begin{align}
d Y_{\o}(t) &= \gamma_t(Y_{\o}, Y_1)dt + \sigma(t, Y_{\o})d\widetilde{W}_{\o}(t), \label{pf:def:uniquenessGW-SDEproj1}  \\
d Y_1(t) &= 1_{\{1 \in \tree_1\}} \left(\gamma_t(Y_1,Y_{\o})dt + \sigma(t,Y_1)d\widetilde{W}_1(t)\right), \label{pf:def:uniquenessGW-SDEproj2}
\end{align} 
where $\gamma_t: \R_+ \times \C^2 \mapsto \R^d$  
is the progressively measurable mapping  defined in \eqref{def:gamma}.
\end{lemma}
\begin{proof}
  To see that \eqref{QQPP} indeed defines a probability measure,
  note that $\PP(N_{\o}(\tree)=\emptyset)=\rho(0)$ and
\begin{align*}
\E^{\PP}\left[\frac{|N_{\o}(\tree)|}{|N_1(\tree)|}1_{\{N_{\o}(\tree_1) \neq \emptyset\}}\right] &= \sum_{k=1}^\infty\sum_{j=0}^\infty \frac{k}{j+1}\rho(k)\widehat\rho(j) \\
	&= \sum_{k=1}^\infty\sum_{j=0}^\infty \frac{k}{j+1}\rho(k)\frac{(j+1)\rho(j+1)}{\sum_{i=1}^\infty i \rho(i)} \\
	&= 1 - \rho(0).
\end{align*}
We stress that $\tree$ is a UGW($\rho$) tree under $\PP$ but not under  $\QQ$, although both measures give rise to the same conditional law of the particles $\widehat{X}_\V$ given the tree $\tree$.

We now turn to the proof of the second assertion
of the lemma. Observe first that $(Y_v)_{v \in \V_1}$ solves the SDE system  \eqref{statements:localequationGW},  
where $\gamma_t$ is defined as in \eqref{def:gamma}. 
Note that the change of measure from $\PP$ to $\QQ$ alters the law of the tree $\tree$ but not the Brownian motions or initial states.
We can then apply Theorem \ref{th:brunickshreve}  to construct, by again extending the probability space $(\Omega, \F, \FF, \QQ)$,
$d$-dimensional independent $\FF$-Brownian motions  $(\widetilde{W}_v)_{v \in \{\o,1\}}$ such that   $(Y_{\o}, Y_1)$ satisfy the following SDE system:  
\begin{align*}
d Y_{\o}(t) &= \widetilde{b}_{\o}(t, Y_{\o}, Y_1)dt + \widetilde{\sigma}_{\o}(t, Y_{\o},Y_1)d\widetilde{W}_{\o}(t),   \\
d Y_1(t) &= \widetilde{b}_1(t, Y_{\o}, Y_1)dt + \widetilde{\sigma}_1(t, Y_{\o}, Y_1)d\widetilde{W}_1(t), 
\end{align*} 
where $\widetilde{b}_v : \R_+ \times \C^2 \to \R^d$ and $\widetilde{\sigma}_v : \R_+ \times \C^2 \to \R^{d \times d}$ are any progressively measurable functions satisfying 
\begin{align*}
\widetilde{b}_{\o}(t,Y_{\o},Y_1) &= \E^{\QQ}\left[\left. b(t,Y_{\o},Y_{N_{\o}(\tree_1)}) \, \right| \, Y_{\o}[t], \, Y_1[t]\right], \\
\widetilde{b}_1(t,Y_{\o},Y_1) &= \E^{\QQ}\left[\left. 1_{\{1 \in \tree_1\}} \gamma_t(Y_1,Y_{\o}) \, \right| \, Y_{\o}[t], \, Y_1[t]\right], \\
\widetilde{\sigma}_{\o}\widetilde{\sigma}_{\o}^\top(t,Y_{\o},Y_1) &= \E^{\QQ}\left[\left. \sigma \sigma^\top(t,Y_{\o}) \, \right| \, Y_{\o}[t], \, Y_1[t]\right], \\
\widetilde{\sigma}_{1}\widetilde{\sigma}_{1}^\top(t,Y_{\o},Y_1) &= \E^{\QQ}\left[\left. 1_{\{1 \in \tree_1\}}\sigma\sigma^\top(t,Y_{1}) \, \right| \, Y_{\o}[t], \, Y_1[t]\right].
\end{align*}
Note again that progressively measurable versions exist by Lemma \ref{le:optprojection}.

 Now, by Remark \ref{rem-treerecovery} and in particular \eqref{def:Tmeasurable-strong},  $\{1 \in \tree_1\}$ is $Y_1[t]$-measurable for each $t > 0$. 
 Together with  the progressive measurability of
 $(t,x,x') \mapsto \gamma_t(x,x')$, this shows that 
\begin{align*}
  \widetilde{\sigma}_{\o}(t,Y_{\o},Y_1) &= \sigma(t,Y_{\o}), \\
  \widetilde{\sigma}_{1}(t,Y_{\o},Y_1) &= 1_{\{1 \in \tree_1\}}\sigma(t,Y_1), \\ 
\widetilde{b}_1(t,Y_{\o},Y_1) &= 1_{\{1 \in \tree_1\}}\gamma_t(Y_1,Y_{\o}).
\end{align*}
On the other hand, in terms of the Radon-Nikodym derivative $d\QQ/d\PP$ we can rewrite  
\begin{align*}
\widetilde{b}_{\o}(t,Y_{\o},Y_1) &= \E^{\PP}\left[\left. \frac{d\QQ}{d\PP} b(t,Y_{\o},Y_{N_{\o}(\tree_1)}) \, \right| \, Y_{\o}[t], Y_1[t] \right] \Big/ \E^{\PP}\left[\left. \frac{d\QQ}{d\PP} \, \right| \, Y_{\o}[t], Y_1[t] \right].
\end{align*}
On the $Y_1[t]$-measurable event $\{N_{\o}(\tree_1) = \emptyset\}$, we have
$\widetilde{b}_{\o}(t,Y_{\o},Y_1) = b(t,Y_{\o},\onepoint)$,
where we recall the convention that $\onepoint$ denotes the unique element of the one-point space $\C^0$. On the other hand, recalling the definitions of $d\QQ/d\PP$ and $\gamma_t$ from \eqref{QQPP} and
  \eqref{def:gamma}, respectively, on the complementary event $\{N_{\o}(\tree_1) \neq \emptyset\}$ we have
\begin{align*}
\widetilde{b}_{\o}(t,Y_{\o},Y_1) &= \E^{\PP}\left[\left. \frac{|N_{\o}(\tree_1)|}{|N_1(\tree_1)|} b(t,Y_{\o},Y_{N_{\o}(\tree_1)}) \, \right| \, Y_{\o}[t], Y_1[t] \right] \Big/ \E^{\PP}\left[\left. \frac{|N_{\o}(\tree_1)|}{|N_1(\tree_1)|} \, \right| \, Y_{\o}[t], Y_1[t] \right] \\
	&= \gamma_t(Y_{\o},Y_1).
\end{align*}
Thus, in either case, $\widetilde{b}_{\o}(t,Y_{\o},Y_1) = \gamma_t(Y_{\o},Y_1)$, and in fact this identity is precisely the purpose of the change of measure $\QQ$. 
This concludes the proof. 
\end{proof}

\noindent
{\bf Step 2.}  We now express (in Lemmas \ref{lem:Z} and \ref{lem-claw} below) 
 the (conditional) density  $d\widetilde{\nu}_t/d \widetilde{P}^{*,1}_t$ explicitly in terms
of certain local martingales that we now define.  We recall the canonical space $\Omega^n$ and canonical
processes $X = (X_v)_{v \in \V_1}$  introduced in Section \ref{subs-uniqueoutline} and
define the processes $M^n_v, R_v$,  and $R_{\o}$ on $\Omega^n$ as follows: 
\begin{align}
M^n_v &:= \int_0^\cdot (\sigma \sigma^\top)^{-1}(s, X_v)b(s,X_v,X_{N_v(\tree_n)}) \cdot dX_v(s), & n \in \N,  v \in \V_{n-1}, \nonumber  \\
R_v &:= \int_0^\cdot (\sigma \sigma^\top)^{-1}(s, X_v)\gamma_s(X_v,X_{\mom_v})  \cdot  dX_v(s), & v \in \V \backslash \{\o \}, \label{def:pf:MRmartingales} \\ 
R_{\o} &:= \int_0^\cdot (\sigma \sigma^\top)^{-1}(s, X_{\o}) \gamma_s(X_{\o},X_1) \cdot dX_{\o}(s), \nonumber
\end{align}
where we have omitted the arguments from $M^n_v$, $R_v$, and $R_{\o}$ for notational conciseness. 
It will be important later to take note of the following consistency property of $M^n_v$ when we stay away from the leaves of $\V_n$: 
\begin{align}
M^n_v((x_u)_{u \in \V_n}) = M^{n+1}_v((x_u)_{u \in \V_{n+1}}), \quad \text{ for  } v \in \V_{n-1}, \ (x_u)_{u \in \V_{n+1}} \in \C^{\V_{n+1}}. \label{pf:GW-Mconsistency}
\end{align}
Recall the Doleans exponential $\EE_t$ defined in \eqref{def:doleans-exponential}. 
 
\begin{lemma}
	\label{lem:Z}
 For each $t > 0$, we have 
	\begin{equation}
		\frac{d\widetilde{\nu}_t}{d \widetilde{P}^{*,1}_t} = \frac{d\nu_t}{d P^{*,1}_t} = \EE_t(M^1_{\o})\prod_{v \in \tree_1 \backslash \{\o \}} \EE_t(R_v). \label{pf:dnu-dp*}
	\end{equation}  
\end{lemma} 
 
\begin{proof}
The continuity of $b$ and the processes $\widehat{X}_v$ and $Y_v$ for each $v$ ensures that the following integrals are trivially a.s.\ finite:
	\begin{align*}
	\int_0^T|b(t,\widehat{X}_{\o},\widehat{X}_{N_{\o}(\tree_1)})|^2\,dt, \qquad \int_0^T|b(t,Y_{\o},Y_{N_{\o}(\tree_1)})|^2\,dt.
	\end{align*}
	We know also from condition (10) of Definition \ref{def-GWlocchar} that the following integrals are a.s.\ finite:
	\begin{align*}
	\int_0^T|\gamma_t(Y_k,Y_{\o})|^2dt, \qquad \int_0^T|\gamma_t(\widehat{X}_k,\widehat{X}_{\o})|^2dt.
	\end{align*}
	Recalling the form of the  SDE systems for $Y  = (Y_v)_{v \in \V_1}$ and $(\widehat{X}_v)_{v \in \V_1}$
        in \eqref{statements:localequationGW}   and \eqref{def:driftlessSDE}, respectively, and the definitions of 
        $\nu_t$, $\widetilde{\nu_t}$ and   ${P}^{*,1}_t$,  $\widetilde{P}^{*,1}_t$ as the laws of 
         $(Y_v)_{v  \in \V_1}$ and $(\widehat{X}_v)_{v \in \V_1}$  under $\PP$ and $\widetilde{\PP}$, respectively,  
         these facts justify an application of Girsanov's theorem in the form of Lemma \ref{le:ap:girsanov}.  
	By expanding the expression analogous to \eqref{ap:def:girsanov} in the above setting, we see
          that the Radon-Nikodym derivative of $\nu_t$ with respect to $P^{*,1}_t$ takes the form announced in the second equality in \eqref{pf:dnu-dp*}.
	The same logic (noting that $\QQ$ and $\PP$ are mutually absolutely continuous) also yields the same form for $d\widetilde{\nu}_t/d\widetilde{P}^{*,1}_t$, thus justifying the first equality in \eqref{pf:dnu-dp*}.
\end{proof}

Our next goal is to  calculate the following conditional density process for each $t > 0$:  
\begin{align}  \label{def:zt}
Z_t((\tilde{x}_{k})_{k \in \N}; x_{\o},x_1)
  & := \frac{d\widetilde\nu\big((X_{1+k}[t])_{k \in \N} \in \cdot \, | \, X_1[t]=x_1[t],\,X_{\o}[t]=x_{\o}[t]\big)}{d\widetilde{P}^{*,1}\big((X_{1+k}[t])_{k \in \N} \in \cdot \, | \, X_1[t]=x_1[t],\,X_{\o}[t]=x_{\o}[t]\big)}((\tilde{x}_{k}[t])_{k \in \N}), 
\end{align}
for $(x_{\o}, x_1) \in \C^{\o, 1}$ and $(\tilde{x}_{k})_{k \in \N} \in \C^{\N}$.
Recall the  definition of  $\widetilde{P}^{*,\o,1}_t$  just  prior to Lemma  \ref{lem-comeas} as the
  marginal of $\widetilde{P}^{*,1}$ on $\C^{\{\o,1\}}_t$.  
Since $Z_t(\cdot; X_{\o},X_1)$ is a   well-defined conditional density by Lemma \ref{lem:Z}, for $\widetilde{P}^{*,\o,1}_t$-a.e.\ $(x_{\o},x_1) \in \C_t^2$ we have
\begin{align}
  1 &= \E^{\widetilde{P}^{*,1}}\left[
    Z_t((X_{1+k})_{k \in \N}; X_{\o}, X_1)\, | \, X_{\o}[t]=x_{\o},X_1[t]=x_1\right] \nonumber \\
  &= \E^{\QQ}\left[
Z_t((\widehat{X}_{1+k})_{k \in \N}; \widehat{X}_{\o}, \widehat{X}_1)    \, | \, \widehat{X}_{\o}[t]=x_{\o},\widehat{X}_1[t]=x_1 \right]. \label{pf:QQE[Z]=1}
\end{align}
In particular, on the $\widehat{X}_1[t]$-measurable event $\{1 \notin \tree\}$, note that $Z_t((\widehat{X}_{1+k})_{k \in \N}; \widehat{X}_{\o}, \widehat{X}_1)$
is $(\widehat{X}_{\o}[t],\widehat{X}_1[t])$-measurable and must therefore equal $1$.

\begin{lemma}
  \label{lem-claw}
For each $t > 0$, we have 
\begin{align}
  Z_t((X_{1+k})_{k \in \N}; X_{\o},X_1)
  &= \frac{\EE_t(M^1_{\o})}{\EE_t(R_{\o})}\prod_{v \in N_{\o}(\tree) \backslash \{1\}}\EE_t(R_v), \qquad \widetilde{P}^{*,1}-a.s.  \label{eq:zt}
\end{align} 
Moreover,  for each $n \in \N$ and $v \in \V_n \setminus \V_{n-1}$, we have a.s.\
\begin{align}
\begin{split}
1 &= \E^{\PP}\left[ Z_t(\widehat{X}_{C_v(\tree)};\widehat{X}_{v},\widehat{X}_{\mom_v}) \, | \, \widehat{X}_v[t],\widehat{X}_{\mom_v}[t] \right]  \\
	&= \E^{\PP}\left[ Z_t(\widehat{X}_{C_v(\tree)};\widehat{X}_{v},\widehat{X}_{\mom_v}) \, | \, \widehat{X}_{\V_n}[t] \right],
\end{split}	\label{pf:Zconditionaldensity1}
\end{align}
where we write $C_v(\tree) := N_v(\tree) \backslash \{\mom_v\}$ for the children of the vertex $v$. 
  \end{lemma}
\begin{proof}
 We first compute the density $d\widetilde\nu^{{\o},1}_t / d \widetilde{P}^{*,{\o},1}_t$. 
By Lemma \ref{lem-comeas}, $\widetilde\nu^{{\o},1}$ is the law of the solution $(Y_{\o}, Y_1)$ to the SDE system defined by \eqref{pf:def:uniquenessGW-SDEproj1} and \eqref{pf:def:uniquenessGW-SDEproj2}.
 Hence,  condition (10) of Definition \ref{def-GWlocchar} justifies an application of Girsanov's theorem, in the form of Lemma \ref{le:ap:girsanov}, which yields 
\begin{align}
\frac{d\widetilde\nu^{{\o},1}_t}{d \widetilde{P}^{*,{\o},1}_t} (X_{\o}, X_1)
    &=  \begin{cases}\EE_t(R_{\o})\EE_t(R_1) &\text{if } 1 \in \tree_1, \\
\EE_t(R_{\o}) &\text{if } 1 \notin \tree_1. 
\end{cases} \label{def:dtildenu-dp*}
\end{align}
Moreover,  using Bayes' rule we obtain
\begin{align*}
  Z_t((X_{1+k})_{k \in \N}; X_{\o},X_1)
  &= \left. \frac{d\widetilde\nu_t}{d\widetilde{P}^{*,1}_t} (X_{\V_1}) \right/ \frac{d\widetilde\nu^{{\o},1}_t}{d\widetilde{P}^{*,{\o},1}_t}(X_{\o}, X_1). 
  \end{align*} 
Appealing to \eqref{def:dtildenu-dp*} and \eqref{pf:dnu-dp*}, we then obtain \eqref{eq:zt}.
Alternatively, recalling the definitions of the martingales $R_v$ and $M^1_{\o}$, shows that $Z_t$ is really a function of $(\lan X_{N_{\o}(\tree)}[t]\ran,X_{\o}[t],X_1[t])$; that is, the dependence on the coordinates $(X_{1+k}[t])_{k \in \N}$ is only through the equivalence class $\lan X_{N_{\o}(\tree)}[t] \ran$ (which is a random element of $\SQ(\C_t)$). Thus,   we can write  
\begin{align}
  Z_t((X_{1+k})_{k \in \N}; X_{\o},X_1) &=
  \widehat{Z}_t(\lan X_{N_{\o}(\tree)}  \ran; X_{\o},X_1),
  \label{pf:Z-symmetrynote}
\end{align}
where $\widehat{Z}_t: \SQ (\C) \times \C^2 \mapsto \R_+$ is defined by 
\[  \widehat{Z}_t(\lan X_{N_{\o}(\tree)}\ran ;X_{\o},X_1) := \frac{\EE_t(M^1_{\o})}{\EE_t(R_{\o})\EE_t(R_1)}\prod_{v \in N_{\o}(\tree)}\EE_t(R_v). \]

For the proof of the second (and last) assertion of the lemma,
we take advantage of some symmetries of the driftless particle system $\widehat{X}_\V$
defined in \eqref{def:driftlessSDE}. 
First note that, by inspecting \eqref{def:driftlessSDE}, and recalling the  conditional
  independence properties of the UGW tree  $\tree$ itself,
it is clear that $\widehat{X}_{C_v(\tree)}$ is conditionally independent of $\widehat{X}_{\V_n}$ given $\{v \in \tree\}$ under $\PP$, for each $n \in \N$ and $v \in \V_n \setminus \V_{n-1}$.  
 This immediately implies the second identity in \eqref{pf:Zconditionaldensity1}.
Second, we claim that in order to prove the first identity in \eqref{pf:Zconditionaldensity1} it suffices to prove it only for  the case $v=1$. This is because each non-root vertex in the UGW($\rho$) tree $\tree$ has the same offspring distribution $\widehat\rho$ under $\PP$, and thus the conditional law of $\widehat{X}_{C_v(\tree)}$ given $\{v \in \tree\}$ does not depend on the choice of $v \in \V \setminus \{\o\}$.

To prove the first identity in \eqref{pf:Zconditionaldensity1} for the case $v=1$, first recall that, as noted just after \eqref{pf:QQE[Z]=1}, on the event $\{1 \notin \tree\}$ it holds that $Z_t(\widehat{X}_{C_1(\tree)};\widehat{X}_1,\widehat{X}_{\o}) = 1$. Hence, we focus on the complementary event. Recall the notation of \eqref{pf:Z-symmetrynote}, which gives
\begin{align*}
  \E^{\PP}\left[
    Z_t(\widehat{X}_{C_1(\tree)}; \widehat{X}_1,\widehat{X}_{\o})    \, | \, \widehat{X}_1[t],\widehat{X}_{\o}[t] \right] &= \E^{\PP}\left[
  \widehat{Z}_t(\lan \widehat{X}_{N_1(\tree)}\ran; \widehat{X}_1,\widehat{X}_{\o})  \, | \, \widehat{X}_1[t],\widehat{X}_{\o}[t] \right].
\end{align*}
We are now in a position to apply Proposition \ref{pr:invariance-GW}. Indeed, Proposition \ref{pr:invariance-GW} applies not just to  the original SDE system $X_\V$ of  \eqref{statements:SDE} but also to the system $\widehat{X}_\V$ defined in \eqref{def:driftlessSDE}, simply because the latter is the special case of the former corresponding to $b \equiv 0$.
We deduce that, on the event $\{1 \in \tree\}$, we have
\begin{align*}
  \E^{\PP}\left[
    \widehat{Z}_t(\lan \widehat{X}_{N_1(\tree)}\ran; \widehat{X}_1,\widehat{X}_{\o})\, | \, \widehat{X}_1[t],\widehat{X}_{\o}[t] \right] &= \Xi_t(\widehat{X}_1,\widehat{X}_{\o}),
\end{align*}
where we define $\Xi_t : \C_t^2 \to \R$ by
\begin{align*}
  \Xi_t(\widehat{X}_{\o},\widehat{X}_1) := 1_{\{1 \in \tree\}}\frac{\E^{\PP}\left[\left. \frac{|N_{\o}(\tree)|}{|N_1(\tree)|} 
   \widehat{Z}_t(\lan \widehat{X}_{N_1(\tree)}\ran; \widehat{X}_{\o},\widehat{X}_{1})    \, \right| \, \widehat{X}_{\o}[t], \, \widehat{X}_1[t] \right]}{\E^{\PP}\left[\left. \frac{|N_{\o}(\tree)|}{|N_1(\tree)|} \, \right| \, \widehat{X}_{\o}[t], \, \widehat{X}_1[t] \right]}.
\end{align*}
Recalling from \eqref{QQPP}
that $d\QQ/d\PP = |N_{\o}(\tree)|/|N_1(\tree)|$ on $\{1 \in \tree\}$, it follows from Bayes' rule that
\begin{align*}
  \Xi_t(\widehat{X}_{\o},\widehat{X}_1)  = \E^{\QQ}\left[\left.
 \widehat{Z}_t(\lan \widehat{X}_{N_1(\tree)}\ran; \widehat{X}_{\o},\widehat{X}_{1}) 
    \, \right| \, \widehat{X}_{\o}[t], \, \widehat{X}_1[t] \right], \quad \text{ on } \{1 \in \tree\}.
\end{align*}
Reverting back from the $\widehat{Z}$ to $Z$ notation as in \eqref{pf:Z-symmetrynote}, this can be rewritten as
\begin{align*}
  \Xi_t(\widehat{X}_{\o},\widehat{X}_1)  = \E^{\QQ}\left[\left.
  Z_t((\widehat{X}_{1+k})_{k \in \N}; \widehat{X}_{\o},\widehat{X}_{1})   \, \right| \, \widehat{X}_{\o}[t], \, \widehat{X}_1[t] \right], \quad \text{ on } \{1 \in \tree\}.
\end{align*}
It follows from \eqref{pf:QQE[Z]=1} that $\Xi_t(\widehat{X}_{\o},\widehat{X}_1) = 1$ on $\{1 \in \tree\}$, which completes the proof of \eqref{pf:Zconditionaldensity1}.
\end{proof}

\noindent
    {\bf Step 3. }  We finally present the main construction of the argument, which involves establishing a one-to-one correspondence between solutions of the local equation  and solutions of the infinite  SDE system \eqref{statements:SDE} via
      a recursive construction and an extension.   Recall the definition of the law $P^{*,n} \in \P(\C^n)$ of the driftless process introduced in Section \ref{subs-uniqueoutline}, and as usual, let $P^{*,n}_t$ denote its projection onto $\P(\C^n_t)$.
  For each $t > 0$ and $n \ge 1$, define a probability measure $\nufinal^n_t \in \P(\C_t^{\V_n})$ via the density
\begin{align}
  \frac{d\nufinal^n_t}{dP^{*,n}_t}((x_v)_{v \in \V_n}) &= \frac{d\nu_t}{dP^{*,1}_t}((x_v)_{v \in \V_1})\prod_{v \in \tree_{n-1} \backslash \{\o \}}
Z_t((x_{vk})_{k \in \N}; x_v,x_{\mom_v}),
  \label{pf:GW-defQn}
\end{align}
with $Z_t$ as defined in \eqref{def:zt}.
We now establish the following.

\begin{proposition}  \label{lem-ext}
We have $\nufinal^1_t = \nu_t$ for each $t > 0$. Moreover,  $\{\nufinal^n_t : t > 0, \ n \in \N\}$ is a well defined and \emph{consistent} family of probability measures in the sense that for $t > s  \geq 0$ and $n \ge k$ the projection of $\nufinal^n_t$ from $\C^{\V_n}_t$ to $\C^{\V_k}_s$ is precisely $\nufinal^k_s$. 
  Furthermore, the unique extension  $\nufinal \in \P(\C^{\V})$ of
  $\{\nufinal^n\}$ to $\P(\C^{\V})$ coincides with  the (unique) law of a weak
  solution of the SDE system \eqref{statements:SDE} with $\tree$ given as a UGW($\rho$) tree. 
\end{proposition}
\begin{proof}
  Note that $\nufinal^1_t = \nu_t$ for $t > 0$ follows immediately from the definition \eqref{pf:GW-defQn}. 
For the next assertion, note that (as justified below) for each $t > 0$ and $n \in \N$,
\begin{align*}
\E^{P^{*,n+1}} &\left[\frac{d\nufinal^{n+1}_t}{dP^{*,n+1}_t}(X_{\V_{n+1}}[t]) \, \Big| \, X_{\V_n}[t]\right] \\
&= \E^{P^{*,n+1}}\left[ \frac{d\nufinal^{n}_t}{dP^{*,n}_t}(X_{\V_{n}}[t])\prod_{v \in \tree_n \setminus \tree_{n-1}}
Z_t(X_{C_v(\tree)}; X_v,X_{\mom_v}) 
\, \Big| \, X_{\V_n}[t]\right] \\
&= \frac{d\nufinal^{n}_t}{dP^{*,n}_t}(X_{\V_{n}}[t])\prod_{v \in \tree_n \setminus \tree_{n-1}}\E^{P^{*,n+1}}\left[ Z_t(X_{C_v(\tree)}; X_v,X_{\mom_v})  \, \Big| \, X_{\V_n}[t]\right] \\
&= \frac{d\nufinal^{n}_t}{dP^{*,n}_t}(X_{\V_{n}}[t]).
\end{align*}
Indeed, the last line uses the relation \eqref{pf:Zconditionaldensity1}, and the penultimate
line uses the fact that  for $n \in \N$, $(\widehat{X}_{C_v(\tree)})_{v \in \V_n \setminus \V_{n-1}}$ are conditionally independent given $\widehat{X}_{\V_n}$,  which follows from the conditional independence structure of the tree itself; $(C_v)_{v \in \V_n \setminus \V_{n-1}}$ are conditionally independent of each other given $(1_{\{v \in \tree\}})_{v \in \V_n \setminus \V_{n-1}}$.
Iterating this, we find for each $t > 0$ and $n \ge k$ with $n,k \in \N$ that
\begin{align}
\E^{P^{*,n}}\left[\frac{d\nufinal^n_t}{dP^{*,n}_t}(X_{\V_n}[t]) \, \Big| \,  X_{\V_k}[t]\right] &= \frac{d\nufinal^k_t}{dP^{*,k}_t}(X_{\V_k}[t]), \ \ a.s. \label{pf:GWconsistency2}
\end{align}
In particular,
\begin{equation}
	\E^{P^{*,n}}\left[\frac{d\nufinal^n_t}{dP^{*,n}_t}(X_{\V_n}[t]) \right] = \E^{P^{*,1}}\left[\frac{d\nufinal^1_t}{dP^{*,1}_t}(X_{\V_1}[t]) \right] = 1 \label{pf:GWconsistency0}
\end{equation}
and $\nufinal^n_t$ is a well-defined probability measure.

Next, we rewrite the Radon-Nikodym derivative \eqref{pf:GW-defQn} in a more useful form.
Recalling the definitions of the martingales $M^n_v$ and $R_v$ given in \eqref{def:pf:MRmartingales},
the consistency equations \eqref{pf:GW-Mconsistency} and the relation \eqref{eq:zt}, it is straightforward to check that for each
$v \in {\mathcal T}_{n-1} \setminus \{{\o}\}$, 
\begin{align*}
Z_t((X_{vk})_{k \in \N}; X_v, X_{\mom_v}) 
   &= \frac{\EE_t(M^n_v)}{\EE_t(R_v)}\prod_{u \in N_v(\tree) \backslash \{\mom_v\}}\EE_t(R_u) \\
	&= \frac{\EE_t(M^n_v)}{\EE_t(R_v)}\prod_{u \in C_v(\tree)}\EE_t(R_u),
\end{align*}
where we again abbreviate $C_v(\tree) = N_v(\tree) \backslash \{\mom_v\}$.
Combining this relation with \eqref{pf:GW-defQn} and the form of $d\nu_t/dP^{*,1}_t$ given in  \eqref{pf:dnu-dp*}, we obtain 
 \begin{align*}
\frac{d\nufinal^n_t}{d P^{*,n}_t} &= \EE_t(M^n_{\o}) \prod_{v \in \tree_1 \backslash \{\o \}}\EE_t(R_v) \prod_{v \in \tree_{n-1} \backslash \{\o \}} Z_t((X_{vk})_{k \in \N};X_v,X_{\mom_v}) \\
	&= \EE_t(M^n_{\o}) \prod_{v \in \tree_1 \backslash \{\o \}}\EE_t(R_v) \prod_{v \in \tree_{n-1} \backslash \{\o \}} \left(\frac{\EE_t(M^n_v)}{\EE_t(R_v)}\prod_{u \in C_v(\tree)}\EE_t(R_u)\right).
\end{align*}
For each $v \in \tree_{n-1}\backslash \{\o \}$, the factor $\EE_t(R_v)$ appears exactly once in the numerator and once in the denominator. Hence, the above reduces to
\begin{align}
\frac{d\nufinal^n_t}{dP^{*,n}_t} &= \prod_{v \in \tree_{n-1}}\EE_t(M^n_v)\prod_{v \in \tree_n \backslash \tree_{n-1}}\EE_t(R_v) \nonumber \\
	&= \EE_t\Bigg(\sum_{v \in \tree_{n-1}}M^n_v + \sum_{v \in \tree_n \setminus \tree_{n-1}} R_v\Bigg), \label{pf:uniquenessGW-RNformQn}
\end{align}
where the second equality follows from the fact that the local martingales $\{M^n_v : v \in \V_{n-1}\} \cup \{R_v : v \in \V_n \setminus \V_{n-1}\}$ are orthogonal.
Combining this with \eqref{pf:GWconsistency0} gives the martingale property
\begin{align}
\E^{P^{*,n}}\left[\frac{d\nufinal^n_t}{dP^{*,n}_t}(X_{\V_n}[t]) \, \Big| \,  X_{\V_n}[s]\right] &= \frac{d\nufinal^n_s}{dP^{*,n}_s}(X_{\V_n}[s]), \ \ a.s., \label{pf:GWconsistency1}
\end{align}
for $t > s > 0$.  

Together, equations \eqref{pf:GWconsistency2} and \eqref{pf:GWconsistency1} prove the stated 
consistency property of the family $\{\nufinal^n\}$.  
Due to the Daniell-Kolmogorov theorem, we deduce from this that there is a unique $Q \in \P(\C^\V)$ whose restriction to $\C^{\V_n}_t$ is $\nufinal^n_t$ for each $n \in \N$ and $t > 0$.

We now turn to the proof of the last statement  of the proposition, which asserts that  $Q$ is the unique
law  of a weak solution to the SDE system \eqref{statements:SDE}.  
To this end, for each $n \ge 1$ and $t > 0$, we identify $\nufinal^n_t$  as the law of an SDE solution as follows.
Recalling the definition  $P^{*,n} = \PP \circ \widehat{X}_{\V_n}^{-1}$, where
$\widehat{X}_v$ satisfies \eqref{def:driftlessSDE}, and the definitions of $M^n_v$ and $R_v$, 
we deduce from \eqref{pf:GWconsistency0}, \eqref{pf:uniquenessGW-RNformQn}, and Girsanov's theorem
that $\nufinal^n$ is precisely the law of a weak solution $(Y_v)_{v \in \V_n}$ of the SDE system
\begin{align}
dY_{v}(t) &= 1_{\{v \in \widetilde{\tree}\}}\Big(b(t,Y_{v},Y_{N_v(\widetilde{\tree})})dt + \sigma(t,Y_v)dB_v(t)\Big), \quad v \in \V_{n-1} \label{pf:uniquenessGW-finalprojSDE-Y} \\
dY_v(t) &= 1_{\{v \in \widetilde{\tree}\}}\Big(\gamma_t(Y_v,Y_{\mom_v})dt + \sigma(t,Y_v)dB_v(t)\Big), \quad v \in \V_n \backslash \V_{n-1}, \nonumber
\end{align}
where $(B_v)_{v \in \V_n}$ are independent Brownian motions,  $(Y_v(0))_{v \in \V_n}$ are i.i.d.\ with law $\lambda_0$, and $\widetilde{\tree}$ is an independent  UGW($\rho$) tree.

Now, define $\widehat{\nufinal}^n \in \P(\C^\V)$ so that the projection onto $\C^{\V_n}$ is precisely $\nufinal^n$ and the coordinates on $\V \backslash \V_n$ are (arbitrarily) chosen to be identically zero, with probability $1$. It is immediate that $\widehat{\nufinal}^n$ converges weakly to $Q$, due to the consistency property of $\{\nufinal^n_t : t \ge 0, \, n \in \N\}$ established above. On the other hand, we argue that if
$\{\widehat{\nufinal}^n\}$ converges to some limit, then this limit must be the law of a weak solution of the infinite SDE system  \eqref{statements:SDE}.  Indeed, if $\{\widehat{\nufinal}^n\}$ converges to $Q = \L((\widehat{Y}_v)_{v \in \V})$, then we may pass to the limit in \eqref{pf:uniquenessGW-finalprojSDE-Y} (using again the weak continuity of stochastic integrals provided by Kurtz and Protter \cite[Theorem 2.2]{Kurtz-Protter} and the continuity of $b$ and $\sigma$ in Assumption \ref{assumption:A}) to find that for each $n$, the $\V_{n-1}$-coordinates $(\widehat{Y}_v)_{v \in \V_{n-1}}$ satisfy the same SDE system as in \eqref{pf:uniquenessGW-finalprojSDE-Y}. As this holds for each fixed $n$, we conclude that $(\widehat{Y}_v)_{v \in \V_{n-1}}$ satisfies the infinite SDE system \eqref{statements:SDE}.  
In light of the uniqueness in law of solutions of  \eqref{statements:SDE} (see Assumption (\ref{assumption:A}.1) and Remark \ref{rem-unique}), 
we conclude that $Q=\L((X_v)_{v \in \V})$, where $(X_v)_{v \in \V}$ was the unique in law solution of \eqref{statements:SDE} on the UGW($\rho$) tree $\tree$. This completes the proof, as we know from the beginning of the proof that the $\V_1$-marginal of $Q$ is precisely $\nufinal^1=\nu$.
\end{proof}

\subsubsection{Completing the proof of uniqueness in Theorem \ref{th:statements-mainlocaleq}}
\label{subs-uniquecomplete}
The lemmas of the previous section contain the proof of the uniqueness assertion in Theorem \ref{th:statements-mainlocaleq}. Indeed, we began in Section \ref{subs-uniqueoutline} with an arbitrary weak solution $((\Omega,\F,\FF,\PP),\tree_1, \gamma,(B_v,Y_v)_{v \in \V_1})$ 
to the UGW($\rho$) local equation with initial law $\lambda_0$. In Proposition \ref{lem-ext}, recalling the notation $\nu=\L((Y_v)_{v \in \V_1})$, we deduced that necessarily $\nu=\L((X_v)_{v \in \V_1})$, where $(X_v)_{v \in \V}$ solves the SDE system  \eqref{statements:SDE}.  We know from Assumption (\ref{assumption:A}.1) (and Remark \ref{rem-unique}) that the SDE system \eqref{statements:SDE} is unique in law. Hence, the law of $(Y_v)_{v \in \V_1}$ does not depend on the choice of weak solution to the UGW($\rho$) local equation.

{

\subsection{Alternative proof of uniqueness in law of solutions  to the local equation} 
\label{sec:alternative}

  In this section we provide an alternative proof of the uniqueness property stated in Theorem \ref{th:statements-mainlocaleq},
 in the case when  the drift $b$ is bounded. 
  In contrast to the proof given  in the previous section, this proof does not 
  refer to the infinite particle system \eqref{statements:SDE}.
  To lead up to the proof of uniqueness for the UGW tree,
which is given in Section \ref{subs-altuniq2}, 
we  first consider the simpler case of the 
$\kappa$-regular tree in Section \ref{subs-altuniq1}. Throughout, we fix $\lambda_0 \in {\mathcal P}(\R^d)$.

\subsubsection{Alternative proof of uniqueness for the $\kappa$-regular tree}
\label{subs-altuniq1}

We first  establish  certain symmetry properties 
that are satisfied by any solution  to the local equation. 
Let $((\omegax,\Fmcx,\Fmbx,\Pmbx), \gammax, (B,X))$ 
 be  any weak solution of the  local equation on the $\kappa$-regular tree 
with initial law $\lambda_0$,  as stated in Definition \ref{def-locregtree},  and let 
 $\Ex$ denote expectation with respect to $\Pmbx$. 
Note that, in particular, this implies 
\begin{equation}
  \label{solX}
\begin{array}{rcl}
	dX_{\o}(t) & = & b(t,X_{\o},X_{\{1,\dotsc,\kappa\}})\,dt + \sigma(t,X_{\o})\,dB_{\o}(t), \\
	dX_{i}(t) & = & \gammax_t(X_{i},X_{\o})\,dt + \sigma(t,X_{i})\,dB_{i}(t), \quad i=1,\dotsc,\kappa,
\end{array}
\end{equation}
with 
\[ 
	\gammax_t(x,y)  =  \Emb[b(t,X_{\o},X_{\{1,\dotsc,\kappa\}}) \, | \, X_{\o}[t]=x[t], \, X_1[t]=y[t]], \quad x,y \in \C.
\] 
Then we have the following result. 

\begin{lemma}  \label{lem-locsym} 
  If $b$ is bounded,  the law of $(X_{\o},X_1,\dotsc,X_\kappa)$ is invariant under permutations of $(X_1,\dotsc,X_\kappa)$; that is,
  for any permutation $\perm$
  of $\{1, \ldots, \kappa\}$, it follows that 
  \begin{equation}
    \label{locsym1}
    \L ((X_{\o},X_1,\dotsc,X_\kappa))  = \L ((X_{\o},X_{\perm(1)},\dotsc,X_{\perm (\kappa)})). 
  \end{equation}
 Furthermore, for every  $i \in \{1, \ldots, \kappa\}$, 
  \begin{equation}
    \label{locsym2}
    \L ((X_{\o}, X_i)) = \L ((X_i, X_{\o})). 
    \end{equation}
\end{lemma}
\begin{proof} 
   We first note that 
   for any fixed  progressively measurable functional
   $\gammax$, since   the SDE  in \eqref{solX} is symmetric and the driving Brownian motions and
   initial conditions are i.i.d., for any permutation
   $\perm$ of $\{1, \ldots, \kappa\}$, 
   $\{(X_{\o}, X_{\perm(1)},  \ldots, X_{\perm(\kappa)}), (B_{\o}, B_{\perm(1)}, \ldots, B_{\perm(\kappa)})\}$ also forms a weak solution to the SDE. 
   If $\gammax$ is also bounded, then 
    due to  Assumption \ref{assumption:A}
   and the boundedness of $b$, 
   (existence and) uniqueness in law of the SDE \eqref{solX}  follows from Girsanov's theorem. 
   Since,  by its definition the particular  $\gammax$ defined above is a bounded progressively measurable functional
   (due to the boundedness of $b$), this immediately proves \eqref{locsym1}.  

    Next, to see why \eqref{locsym2} holds, fix $i \in \{1,\dotsc,\kappa\}$.  Note first that \eqref{locsym1} and the definition of $\gamma_t$ imply $\E[ b(t,X_{\o}, X_{\{1, \ldots, \kappa\}})|X_{\o}[t], X_i[t]] =\gammax_t (X_{\o}[t], X_i[t])$.   
    Applying the projection result in Theorem \ref{th:brunickshreve} and the elementary identity
$\E[\gammax_t (X_{i}[t], X_{\o}[t])|X_{\o}[t], X_i[t]] =\gammax_t (X_{i}[t], X_{\o}[t])$, 
by extending the probability space if necessary, we may find independent $d$-dimensional Brownian motions
$(W_{\o},W_i)$ such that  
\begin{align*}
	dX_{\o}(t) & = \gammax_t(X_{\o},X_i)\,dt + \sigma(t,X_{\o})\,dW_{\o}(t), \\
	dX_i (t) & = \gammax_t(X_{i},X_{\o})\,dt +  \sigma(t,X_i)\,dW_i(t).
\end{align*}
For any fixed bounded functional $\gammax$ (and  hence, for the particular $\gammax$ specified in the local equation) this SDE is unique in law (invoking, as above,  Assumption (\ref{assumption:A}.2b) and Girsanov's theorem). Combined with the fact that the SDE is symmetric, namely   $\{(X_i, X_{\o}), (W_i, W_{\o})\}$ is also a solution to this SDE, this implies that  \eqref{locsym2} also  holds.  
\end{proof}

Now, let  $((\omegay,\Fmcy,\Fmby,\Pmby), \gammay, (\By,\newy))$  be another  weak solution of the $\kappa$-regular tree local equation
with the same initial law $\lambda_0$. 
For $x, y \in \C^2$ and $t > 0$,  letting
\begin{align*}
	\Pone_{x,y}[t] & := \L((X_1,\dotsc,X_\kappa)[t] \,|\, X_{\o}[t]=x[t], X_1[t]=y[t]), \\
	\Ptwo_{x,y}[t] & := \L((\newy_1,\dotsc,\newy_\kappa)[t] \,|\, \newy_{\o}[t]=x[t], \newy_1[t]=y[t]),
\end{align*}
we can write 
\begin{equation*}
	\gammax_t(x,y) = \lan \Pone_{x,y}[t], b(t,x,\cdot) \ran, \quad \gammay_t(x,y) = \lan \Ptwo_{x,y}[t], b(t,x,\cdot) \ran.
\end{equation*}
Then by Assumption \ref{assumption:A}, the boundedness of $b$, and Corollary \ref{co:entropyestimate},  we have  
\begin{align*}
	H(\L(X[t]) \,|\,\L( \newy[t])) & = \frac{1}{2} \Emb \left[\int_0^t \sum_{i=1}^\kappa |\sigma^{-1}(s,X_{i})(\gammax_s(X_i,X_{\o}) - \gammay_s(X_i,X_{\o}))|^2 \,ds \right]\\
	& = \frac{\kappa}{2} \Emb \left[\int_0^t |\sigma^{-1}(s,X_{1})(\gammax_s(X_1,X_{\o}) - \gammay_s(X_1,X_{\o}))|^2 \,ds \right] \\
	& = \frac{\kappa}{2} \Emb \left[\int_0^t |\sigma^{-1}(s,X_{1}) \lan \Pone_{X_1,X_{\o}}[s] - \Ptwo_{X_1,X_{\o}}[s], b(t,X_1,\cdot) \ran|^2 \,ds\right] \\
	& = \frac{\kappa}{2} \Emb \left[\int_0^t |\sigma^{-1}(s,X_{\o}) \lan \Pone_{X_{\o},X_1}[s] - \Ptwo_{X_{\o},X_1}[s], b(t,X_{\o},\cdot) \ran|^2 \,ds\right], 
\end{align*}
where the second  and last lines use the symmetry properties \eqref{locsym1} and \eqref{locsym2}, respectively.  
It then follows from Pinsker's inequality (see, e.g., \cite[p. 44]{CsiKor11}) that
\begin{equation*}
	H(\L(X[t]) \,|\,\L( \newy[t])) \le \kappa\|\sigma^{-1}b\|_\infty^2 \Embx \left[ \int_0^t H(\Pone_{X_{\o},X_1}[s] \,|\, \Ptwo_{X_{\o},X_1}[s]) \,ds \right].
\end{equation*}
Using Fubini's theorem and the chain rule of relative entropy, the right-hand side equals 
\begin{equation*}
\kappa\|\sigma^{-1}b\|_\infty^2 \int_0^t \Emb\big[ H(\L(X[s]) \,|\, \L(\newy[s])) - H(\L((X_{\o},X_1)[s]) \,|\, \L((\newy_{\o},\newy_1)[s])) \big] \,ds.
\end{equation*}
By non-negativity of relative entropy, we finally deduce that
\begin{align*}
H(\L(X[t]) \,|\,\L( \newy[t]))  \le \kappa\|\sigma^{-1}b\|_\infty^2 \int_0^t H(\L(X[s]) \,|\, \L( \newy[s])) \,ds.
\end{align*} 
It then follows from Gronwall's inequality that
\begin{equation*}
	H(\L(X[t]) \,|\,\L( \newy[t])) = 0, \qquad \forall \, t \ge 0,
\end{equation*}
which in particular implies $\L(X) = \L(\newy)$.
This proves the desired uniqueness in law.

\subsubsection{Alternative proof of uniqueness for the UGW$(\rho)$ tree}
\label{subs-altuniq2}

In this section we give an alternative proof of uniqueness for the UGW$(\rho)$ local equation, under the additional assumptions that $b$ is bounded and the offspring distribution has a finite moment generating function:
\begin{equation}
\sum_{k =0}^\infty e^{ck}\rho(k) < \infty , \qquad \forall c > 0. \label{altuniq:rho-tail2}
\end{equation}
Note by a standard Chernoff bound that this is equivalent to the condition
\begin{equation}
\lim_{r\to\infty} e^{cr}\sum_{k=r}^\infty \rho(k) =0 , \qquad \forall c > 0. \label{altuniq:rho-tail}
\end{equation}
This covers the case of UGW trees with uniformly bounded degrees, as well as the important case of Poisson offspring distribution.

Let $((\omegax,\Fmcx,\Fmbx,\Pmbx), \treex_1, \gammax, (B,X), \Nvarx)$ 
be a weak solution of the
UGW$(\rho)$ local equation with initial law $\lambda_0$, as specified in Definition \ref{def-GWlocchar}.  Note that $\rho \in \P(\N)$ has a nonzero first moment as stated  therein, and  
a finite moment of every order by \eqref{altuniq:rho-tail2}.   
Properties (7) and (9) of Definition \ref{def-GWlocchar} state that 
\begin{equation}
  \label{sol-Xugw}
  \begin{array}{rcl} 
	dX_{\o}(t) & = & b(t,X_{\o},X_{N_{\o}(\treex_1)})\,dt + \sigma(t,X_{\o})\,dB_{\o}(t), \\
	dX_{k}(t) & = & 1_{\{k \in \treex_1\}} \Big(\gammax_t(X_{k},X_{\o})\,dt + \sigma(t,X_{k})\,dB_{k}(t)\Big), \quad k \in \N,
  \end{array}
\end{equation}
with 
\[ 	\gammax_t(x,y)  = 
	\left\{
	\begin{array}{ll} 
	\displaystyle  \frac{\E\left[\left. \frac{|N_{\o}(\treex_1)|}{1+\Nvarx}b(t,X_{\o},X_{N_{\o}(\treex_1)}) \, \right| \, X_{\o}[t]=x[t], \, X_1[t]=y[t] \right]}{\E\left[\left. \frac{|N_{\o}(\treex_1)|}{1+\Nvarx} \, \right| \, X_{\o}[t]=x[t], \, X_1[t]=y[t] \right]}
	& \mbox{ on }
	\{N_{\o}(\treex_1) \neq \emptyset\},  \\
	\displaystyle  b(t,X_{\o},\onepoint) & \mbox{ on } \{N_{\o}(\treex_1)=\emptyset\}.
	\end{array}
	\right. 
        \]
        Once again, we start by establishing useful symmetry properties of any weak solution. As in the first proof of uniqueness on the
UGW tree, it is convenient to introduce the  ``tilted'' measure $\QQx$ on $(\omegax, \Fmcx)$ by 
\begin{equation}
  \label{com}
	\frac{d\QQx}{d\Pmbx} = \frac{|N_{\o}(\treex_1)|}{1+\Nvarx}1_{\{N_{\o}(\treex_1) \neq \emptyset\}}  + 1_{\{N_{\o}(\tree) = \emptyset\}},
\end{equation} 
and write $\widetilde{\L}(\cdot)$ and $\widetilde{\E}$ for the law and expectation, respectively, under $\QQx$.  
\begin{remark}
  \label{rem-com}
  The following two properties of $\QQx$ are noteworthy:
  \begin{enumerate}
  \item
   The change of measure from $\Pmb$ to $\QQx$ alters the law of $\treex_1$ and $\Nvarx$, but not the Brownian motion
   or initial states, which (by property (8) of  Definition \ref{def-GWlocchar})
   are independent of $\treex_1$ and $\Nvarx$ under both $\Pmb$ and $\QQx$.   
 \item  On the event $\{N_{\o}(\treex_1) \neq \emptyset\}$, we have the identity
   \begin{align}
     \label{rem-identity}
 	\gammax_t(X_{\o}[t],X_1[t]) = \widetilde{\E}\left[\left. b(t,X_{\o},X_{N_{\o}(\treex_1)}) \, \right| \, X_{\o}[t], \, X_1[t] \right], \ \ a.s., 
 \end{align} 
      \end{enumerate}
\end{remark}

We now establish some invariance properties  of the UGW($\rho$) tree  local equation. 

\begin{lemma}
  \label{lem-locsymugw}
 If $b$  is bounded, then  for every $k \in \N$ and permutation $\perm$ of $\{1, \ldots, k\}$, it  follows that
 \begin{equation}
   \label{eq:alternative-1}
   \widetilde{\L} ( (X_{\o}, X_1, \ldots, X_k) \, | \, |N_{\o}(\treex_1)| = k ) =
  \widetilde{\L} ( (X_{\o}, X_{\perm(1)}, \ldots, X_{\perm(k)}) \, | \, |N_{\o}(\treex_1)| = k). 
  \end{equation}
  Furthermore,  for each bounded measurable function $g: \C^2 \to \R$, we have 
    \begin{equation}
\label{eq:alternative-2}
  \widetilde{\Emb} \left[ g (X_k, X_{\o})\,|\,\{k \in \treex_1\} \right] =
  \widetilde{\Emb} \left[ g (X_{\o}, X_k) \, | \, \{k \in \treex_1\} \right].
  \end{equation} 
\end{lemma}
\begin{proof}
   Note that by Remark \ref{rem-com}(1),   the Brownian motions and initial
   conditions are independent of $\treex_1$, and also by properties (4) and (6) of
   Definition \ref{def-GWlocchar}, 
   the Brownian motions  and initial conditions  are i.i.d. 
Hence,  for any $k \in \N$ and fixed bounded progressively measurable
  functional $\gammax$,  conditioned  on the event $\{|N_{\o}(\treex_1)| = k\}$, 
  the symmetry of the SDE \eqref{sol-Xugw} immediately shows that   
  $\{(X_{\o}, X_{\perm(1)}, \ldots, X_{\perm(k)}), (B_{\o}, B_{\perm(1)}, \ldots, B_{\perm(k)})\}$ is also 
  a weak solution to the SDE \eqref{sol-Xugw}.
   However, by Assumption \ref{assumption:A} and the boundedness of $b$, on the event $\{|N_{\o}(\treex_1)| = k\}$,  
  for any fixed bounded progressively measurable
  functional $\gammax$,  the (existence and) uniqueness in law of  weak solutions to the SDE  \eqref{sol-Xugw}   follows by Girsanov's theorem.
  In particular, since (when $b$ is bounded) the $\gammax$ arising in the weak solution is a progressively measurable
  bounded functional,   the last two statements imply \eqref{eq:alternative-1}. 

  On the other hand,  to show \eqref{eq:alternative-2}, we fix $k \in \N$ and apply the
  projection result in Theorem \ref{th:brunickshreve} to project the SDE \eqref{sol-Xugw} onto $(X_{\o},X_k)$, under the measure $\QQx$. Note first that the identity \eqref{rem-identity} in  Remark \ref{rem-com}(2)
  along with  the identity $\{N_{\o}(\treex_1) \neq \emptyset\} = \{1 \in \treex_1\}$ and the relation \eqref{eq:alternative-1} imply
   \begin{align}
 	\gammax_t(X_{\o}[t],X_k[t]) = \widetilde{\E}\left[\left. b(t,X_{\o},X_{N_{\o}(\treex_1)}) \, \right| \, X_{\o}[t], \, X_k[t] \right], \ \ a.s., \text{ on } \{k \in \treex_1\}.
 \end{align}
   Hence,  invoking the boundedness of  $\sigma$ from Assumption (\ref{assumption:A}.2a) and the
     assumed boundedness of $b$, and thus $\gammax_t$, to verify the  
   condition \eqref{ap:sigmasquareintegrable} of Theorem \ref{th:brunickshreve}, 
by extending the $\QQx$-probability space if necessary,  
we may find independent $d$-dimensional Brownian motions $(W_{0},W_k)$ such that 
\begin{align*}
	dX_{0}(t) & =  \gammax_t(X_{0},X_{k}) \,dt + \sigma(t,X_{0})\,dW_{0}(t), \\
	dX_{k}(t) & = 1_{\{k \in \treex_1\}} \Big(\gammax_t(X_{k},X_{0}) \,dt +  \sigma(t,X_k)\,dW_k(t)\Big), 
\end{align*}
where we have also used the fact that $\{k \in \tree_1\}$ is  $X_k[t]$-measurable for each $t > 0$ (as in Remark \ref{rem-treerecovery}). 
Again,   for any fixed bounded progressively measurable
functional $\gammax$,   by Girsanov's theorem (the boundedness of $b$ and Assumption \ref{assumption:A}),
this SDE is unique in law.
Since, in addition  the SDE is symmetric given $\{k \in \treex_1\}$,
it follows that for each $k \in \N$, 
\[ 
\widetilde{\L}((X_{\o}, X_{k}) | \{ k \in \treex_1\} ) =
\widetilde{\L}((X_{k}, X_{\o}) | \{ k \in \treex_1\}).
\]
This proves the identity in  \eqref{eq:alternative-2}. 
\end{proof}

Now, let  $((\omegay,\Fmcy,\Fmby,\Pmby), \treey_1, \gammay, (\By,\newY), \Nvary)$ be another weak solution of the  
UGW$(\rho)$ local equation with initial law $\lambda_0$, and let $\QQy$ be an absolutely continuous measure
to $\Pmby$, defined as in  \eqref{com}, but with $\QQ, \Pmb, \treex_1, \Nvarx$ replaced
with $\QQy, \Pmby, \treey_1, \Nvary$, respectively. Also, let $\Ly$ be the law under $\Pmby$ and  let $\tLy$ and $\Emby$ be the law and expectation under $\QQy$, respectively.  
For $t > 0$ and $x,y \in \C^2$, define the conditional laws 
\begin{align*} 
  \tPonet_{x,y}[t] & := \widetilde{\L}(X_{N_{\o}(\treex_1)} [t] \,|\, X_{\o}[t]=x[t], X_1[t]=y[t]), \\
  \tPtwot_{x,y}[t] & := \tLy(\newy_{N_{\o}(\treey_1)} [t]  \,|\, \newY_{\o}[t]=x[t], \newY_1[t]=y[t]).
\end{align*}
Recalling the form of $\gammax_t$ and the form of the change of measure in \eqref{com} we can write  
\begin{equation}
  \label{gamma-rep}
	\gammax_t(x,y) = \lan \tPonet_{x,y}[t], b(t,x,\cdot)  \ran, \quad
	\gammay_t(x,y) = \lan \tPtwot_{x,y}[t], b(t,x,\cdot) \ran.
\end{equation}
By properties (2), (3) and (8) of the local equation in Definition  \ref{def-GWlocchar}, 
$\L(\treex_1, \Nvarx)$ and $\Ly(\treey_1, \Nvary)$  both represent the joint law of the root neighborhood and  the number
of offspring of a  neighbor of the root in a 
UGW($\rho$) tree, and so by the definitions of the changes of measure $\QQx$ and $\QQy$, $\widetilde{\L}(\treex_1, \Nvarx)$ also coincides with
$\tLy(\treey_1, \Nvary)$. 
When combined with the chain rule for relative entropy, this  shows that    
\begin{align*}
 H\big(\widetilde{\L}(X[t], \treex_1,\Nvarx) \,|\, \tLy(\newY[t], \treey_1,\Nvary) \big) 
 &=  H\big(\widetilde{\L}(\treex_1,\Nvarx)  \,|\,  \tLy (\treey_1,\Nvary) \big)
 + \widetilde{\Emb} \left[H\big( \nuone_{\treex_1,\Nvarx}[t] \,|\, \nutwo_{\treex_1,\Nvarx}[t]\big)\right] \\
 &= \widetilde{\Emb} \left[H\big( \nuone_{\treex_1,\Nvarx}[t] \,|\, \nutwo_{\treex_1,\Nvarx}[t]\big)\right],
\end{align*}
where, for $t>0$, $m \in \N_0$, and trees $\tau_1 \subset \V_1$, we define
\begin{align*}
	\nuone_{t_1,m}[t] := \widetilde{\L}(X[t] \,|\, \treex_1 = \tau_1, \Nvarx = m), \quad \nutwo_{t_1,m}[t] := \tLy(\newy[t] \,|\, \treey_1 = \tau_1, \Nvary = m).
\end{align*}
Now, since for each solution, the initial condition and driving Brownian motions are independent of the tree by Remark \ref{rem-com}(1), conditioned on $\treex_1 = \tau_1, \Nvarx = m$, $X$
simply satisfies the SDE \eqref{sol-Xugw} with $\treex_1$ and $\Nvarx$ replaced by 
$\tau_1$ and $m$, respectively, and an exactly analogous statement holds for
the conditional dynamics of $\newy$ given $\treey_1$ and $\Nvary$.  Together with 
Assumption \ref{assumption:A} and the boundedness of $b$,
this allows us to  invoke the entropy identity of Corollary \ref{co:entropyestimate} to obtain 
\begin{align*}
  &  \widetilde{\Emb} \left[H\big( \nuone_{\treex_1,\Nvarx}[t] \,|\, \nutwo_{\treex_1,\Nvarx}[t]\big)\right] \\
	& = \frac12 \widetilde{\Emb} \left[\int_0^t \sum_{k=1}^\infty 1_{\{k \in \treex_1\}} |\sigma^{-1}(s,X_{k})(\gammax_s(X_k,X_{\o}) - \gammay_s(X_k,X_{\o}))|^2 \,ds \right]\\
	&  = \frac12 \widetilde{\Emb}\left[ \int_0^t \sum_{k=1}^\infty 1_{\{k \in \treex_1\}} |\sigma^{-1}(s,X_{\o})(\gammax_s(X_{\o},X_{k}) - \gammay_s(X_{\o},X_{k}))|^2 \,ds\right] \\
	& = \frac12 \widetilde{\Emb}\left[ \int_0^t \sum_{k=1}^\infty 1_{\{k \in \treex_1\}} |\sigma^{-1}(s,X_{\o})(\gammax_s(X_{\o},X_{1}) - \gammay_s(X_{\o},X_{1}))|^2 \,ds\right] \\
	&   = \frac12 \widetilde{\Emb} \left[\int_0^t |N_{\o}(\treex_1)| |\sigma^{-1}(s,X_{\o}) \lan \tPonet_{X_{\o},X_{1}}[s] - \tPtwot_{X_{\o},X_{1}}[s], b(t,X_{\o},\cdot) \ran|^2 \,ds \right],
\end{align*}
where the second equality uses \eqref{eq:alternative-2}, the third equality  uses \eqref{eq:alternative-1}, and the fourth equality uses \eqref{gamma-rep}.
Set $C := \|\sigma^{-1}b\|^2_\infty$, 
let $r > 0$, and introduce the indicator of the event $\{|N_{\o}(\treex_1)| \le r\}$ and its complement, to bound the above by
\begin{align*}
\frac{r}{2} \widetilde{\Emb} \left[\int_0^t |\sigma^{-1}(s,X_{\o}) \lan \tPonet_{X_{\o},X_{1}}[s] - \tPtwot_{X_{\o},X_{1}}[s], b(t,X_{\o},\cdot) \ran|^2 \,ds \right] +  2Ct \widetilde{\Emb} \left[|N_{\o}(\treex_1)|1_{\{|N_{\o}(\treex_1)|>r\}}\right]. 
\end{align*} 
Now, letting $d_{\rm{TV}}$ denote the total variation distance, one has for each $s \in [0,t]$, 
\begin{align*}
  |\sigma^{-1}(s,X_{\o})\lan   \tPonet_{X_{\o},X_1}[s] - \tPtwot_{X_{\o},X_1}[s], b(s,X_{\o},\cdot) \ran|^2 
  & \leq C d_{\rm{TV}}^2 (\tPonet_{X_{\o},X_1}[s], \tPtwot_{X_{\o},X_1}[s]), 
\end{align*}
and so Pinsker's inequality (see, e.g., \cite[p. 44]{CsiKor11}) implies 
\[
  |\sigma^{-1}(s,X_{\o})\lan   \tPonet_{X_{\o},X_1}[s] - \tPtwot_{X_{\o},X_1}[s], b(s,X_{\o},\cdot) \ran|^2 
  \leq 2 C H\big(\tPonet_{X_{\o},X_1}[s] \,|\, \tPtwot_{X_{\o},X_1}[s]\big). \] 
Combine the last six  displays to obtain 
\begin{align}
  & \displaystyle H\big(\widetilde{\L}(X[t], \treex_1,\Nvarx) \,|\, \tLy(\newY[t], \treey_1,\Nvary)\big)
  \nonumber \\ & \,
 \qquad  \qquad \displaystyle  \le Cr \widetilde{\E}\left[ \int_0^t H( \tPonet_{X_{\o},X_1}[s] \,|\,  \tPtwot_{X_{\o},X_1}[s]) \,ds \right] + 2Ct \widetilde{\Emb} \left[|N_{\o}(\treex_1)|1_{\{|N_{\o}(\treex_1)|>r\}}\right]. \label{pf:UGWuniq2}
\end{align} 
Moreover, the chain rule and the data processing inequality of relative entropy (see \cite[Appendix E]{GerMenSto20}) imply that for  $s \in [0,t]$, 
\begin{align*}
  \widetilde{\Emb}\left[ H\big(\tPonet_{X_{\o},X_1}[s] \,|\, \tPtwot_{X_{\o},X_1}[s]\big)\right] &= H\big(\widetilde{\L}(X[s]) \,|\, \tLy(\newY[s])\big) -
  H\big(\widetilde{\L}((X_{\o},X_1)[s]) \,|\, \widetilde{\L}^\prime((\newY_{\o},\newY_1)[s])\big)]   \\
	& \le H\big(\widetilde{\L}(X[s]) \,|\, \tLy(\newY[s])\big) \\
	& \le  H\big(\widetilde{\L}(X[s],\tree_1,\Nvar) \,|\, \tLy(\newY[s],\treey_1,\Nvary)\big). 
\end{align*}
Substitute this into \eqref{pf:UGWuniq2}, and apply Gronwall's inequality to deduce that  for every $r >  0$, 
\begin{equation}
	H\big(\widetilde{\L}(X[t],\treex_1,\Nvarx) \,|\, \tLy(\newY[t],\treey_1,\Nvary)\big) \le 2Ct e^{Crt} \widetilde{\Emb} \left[|N_{\o}(\treex_1)|1_{\{|N_{\o}(\treex_1)|>r\}}\right], \quad \forall  \ t \geq 0. \label{pf:alt:trunc1}
\end{equation}
Recalling the definition of $\QQ$ in \eqref{com} and applying the Cauchy-Schwarz inequality, we have
\begin{equation*}
 \widetilde{\Emb} \left[|N_{\o}(\treex_1)|1_{\{|N_{\o}(\treex_1)|>r\}}\right] \le \E\left[|N_{\o}(\treex_1)|^21_{\{|N_{\o}(\treex_1)| > r\}}\right] \le \left(\E\left[|N_{\o}(\treex_1)|^4\right] \,\PP(|N_{\o}(\treex_1)|>r)\right)^{1/2}.
\end{equation*}
Substituting this into \eqref{pf:alt:trunc1}, sending $r \to \infty$ and noting that \eqref{altuniq:rho-tail2}-\eqref{altuniq:rho-tail} imply that 
   $\E\left[|N_{\o}(\treex_1)|^{4}\right] < \infty$ and 
  $e^{Crt} \left(\PP(|N_{\o}(\treex_1)|>r)\right)^{1/2} \rightarrow 0$,  it follows that  
\begin{equation*}
	H\big(\widetilde{\L}(X[t],\treex_1,\Nvarx) \,|\, \tLy(\newY[t],\treey_1,\Nvary)\big) = 0, \qquad \forall \, t \ge 0.
\end{equation*}
This means $\widetilde{\L}(X,\treex_1,\Nvarx) = \tLy(\newY,\treey_1,\Nvary)$, and thus
$\L(X,\treex_1,\Nvarx) = \Ly(\newY,\treey_1,\Nvary)$. 
This completes the (alternative) proof of  uniqueness in law of weak solutions to the UGW$(\rho)$ local equation with a given initial law $\lambda_0$. }

\section{Second-order Markov random fields} \label{se:2MRFs}

 The rest of the paper is devoted to justifying the two key Propositions \ref{pr:properties-GW} and \ref{pr:invariance-GW}. 
We begin by summarizing some general properties of Markov random fields (MRFs) which will play a key role in the former proposition.  
Throughout this section, we work with a fixed Polish space $\X$ and a fixed (non-random) graph $G=(V,E)$, assumed to have finite or countable vertex set. We assume that $G$ is simple (no self-loops or multi-edges), but it need not be locally finite (so that we may use $G=\V$).
We fix a reference measure $\lambda \in \P(\X)$. 
The goal of this section is to summarize how conditional independence properties
of a measure $\mu \in \P(\X^V)$  can be deduced from factorization properties of
its density with respect to the product measure $\lambda^V$.

We recall the basic graph-theoretic definitions given in Section \ref{subsub-graphs}, in particular
the notion of boundary and double boundary of a set $A$ of vertices in a graph $G = (V,E)$ defined in \eqref{def:boundaries}.
In what follows,  for any random elements $Y_i$, $i = 1, 2, 3$, we write  $Y_1 \indep Y_2 \,| \, Y_3$
to denote that $Y_1$ is conditionally independent of $Y_2$ given $Y_3$.

\begin{definition}[Second-order MRF] \label{def-2MRF}
A collection of $\X$-valued random elements $(Y_v)_{v \in G}$ is said to form a (global) second-order MRF with respect to $G$ if
for any sets $A \subset V$, $B \subset V \setminus (A \cup \partial^2 A)$, 
we have the following conditional independence structure:  
\[
Y_A \indep Y_B \ \ | \ \ Y_{\partial^2 A}. 
\]
\end{definition}

Note that a first-order MRF (with respect to $G$), sometimes also referred to as a  Gibbs measure, would require the same to hold but with $\partial A$ in place of $\partial^2A$.

We state here a variant of a well known theorem, which can be found in various forms in \cite[Theorem 2.30]{georgii2011gibbs} and \cite[Proposition 3.8, Theorem 3.9]{lauritzen1996graphical}, for first-order MRFs on finite graphs.
We do not state the more difficult converse, often attributed to Hammersley-Clifford, as we will not need it.  Recall that a 2-clique of a graph is a set of vertices of diameter at most 2. 

\begin{theorem} \label{th:hammersleyclifford2}
 Assume the graph $G$ is finite. 
Assume $\mu \in \P(\X^V)$ is absolutely continuous with respect to $\lambda^V$. Suppose there exists a set $\K$ of $2$-cliques of $G$ such that the density of $\mu$ with respect to $\lambda^V$ factorizes in the form
\begin{equation}    \label{eq-mufactor}
\frac{d\mu}{d\lambda^V}(x_V) = \prod_{K \in \K}f_K(x_K),
\end{equation}
for some measurable functions $f_K : \X^K \rightarrow \R_+$, for $K \in \K$.
Then $\mu$ is a second-order MRF. 
\end{theorem}
\begin{proof}
Let $A \subset V$, and let $\varphi : \X^{\partial^2 A} \to \R_+$ denote the marginal density of $X_{\partial^2 A}$. Let $B = (A \cup \partial^2 A)^c$. Then the conditional density of $(X_{A},X_B)$ given $X_{\partial^2 A}$ is precisely
\begin{align*}
\frac{1}{\varphi(x_{\partial^2 A})}\prod_{K \in \K}f_K(x_K).
\end{align*}
No $2$-clique of $G$ that intersects $A$ can also intersect $B$, and vice versa, because any pair of vertices $u\in A$ and $v \in B$ have distance at least 3. Thus, for $x_{\partial^2 A}$ frozen, the above conditional density as a function of $(x_A,x_B)$ factorizes into a function of $x_A$ times a function of $x_B$. This implies $X_A$ and $X_B$ are conditionally independent given $X_{\partial^2 A}$.
\end{proof}

The second-order MRF property is more intuitive, but the factorization property of Theorem \ref{th:hammersleyclifford2} will be quite useful in our analysis. Hence, we give it a name:

\begin{definition} \label{def:2cliquefactorization}
 We say that $\mu \in \P(\X^V)$ \emph{admits a $2$-clique factorization with respect to $\lambda^V$} if the density $d\mu/d\lambda^V$ exists and takes the form \eqref{eq-mufactor}, for some set $\K$ of $2$-cliques of $G$.
\end{definition}

It is clear that Theorem \ref{th:hammersleyclifford2} admits a generalization to $m$-order MRFs, defined in the obvious way for $m \in \N$, where one must assume the density factorizes over $m$-cliques, but we have no use for such a generalization.

\section{Proof of the conditional independence property} \label{se:proofs-GW-conditionalindependence}

We now turn to the proof of the conditional independence property stated in
Proposition \ref{pr:properties-GW}, which played a crucial role in the proof of existence 
for Theorem \ref{th:statements-mainlocaleq}.  The strategy is to first establish
the property on certain finite truncations of the tree, and then use an approximation argument. 
Specifically, in Section \ref{subs-truncated} 
  we first establish the desired conditional independence property on a truncation  
  of the infinite tree $\V$ to one of finite depth  and width by  explicitly identifying the joint density with respect to a product measure and then invoking Theorem \ref{th:hammersleyclifford2}.  
 In Section \ref{subs-condinftree} we then implement a rather delicate limiting argument to show that
the conditional independence property  is preserved when  the infinite
  tree is approximated by 
  trees of finite depth and width.

\subsection{Truncated systems}
\label{subs-truncated}

We begin by studying the particle system set on the truncated (finite) tree
$\tree_n := \tree \cap \V_{n,n}$, where $\tree$ is a UGW($\rho$) tree, and 
\begin{align*}
\V_{m,n} := \{\o\} \cup \bigcup_{k=1}^m \{1,\ldots,n\}^k, \quad \text{ for } n,m \in \N.
\end{align*}  
That is, $\V_{m,n}$ is the set of labels of trees of height $m$ with at most $n$ offspring per generation.
Let $(X^n_v)_{v \in \V} := (X^{\tree_n}_v)_{v \in \V}$ be a solution to  the SDE system
\begin{equation}
  \label{drifted-x}
  dX_v^n (t) = 1_{\{v \in \tree_n\}} \left( b(t,X_v^n, X_{N_v(\tree_n)}^n) dt + \sigma (t, X_v^n) dW_v(t)\right), \qquad v \in \V, 
\end{equation}
where  $(X_v^n(0))_{v \in \V}$ are i.i.d.\ with law $\lambda_0$, and as usual the tree $\tree$, the initial conditions $(X_v^n(0))_{v \in \V}$, and the driving Brownian motions $(W_v)_{v \in \V}$ are independent. Also, for $v \in \V \backslash \tree_n$, note as usual that the particles are constant over time,
with $X_v^n(t)=X_v^n(0)$ for all $t > 0$.
Let $P^n \in \P((\{0,1\} \times \C)^{\V})$ denote the law of 
\begin{equation}
\label{def-pn}
\Bigl(1_{\{ v \in \tree_n \}}, \, X^n_v\Bigr)_{v \in \V}.
\end{equation}
We will identify $P^n$ by way of its Radon-Nikodym derivative with respect to a certain reference measure (in the process 
showing that the SDE \eqref{drifted-x} is unique in law).  
In this case, as a reference measure  
we use $\W \in \P((\{0,1\} \times \C)^{\V})$, defined as the law of $(\xi_v,\widehat{X}_v)_{v \in \V}$, where $(\xi_v)_{v \in \V}$ are independent Bernoulli($1/2$) random variables, and where $\widehat{X}$ solves the driftless SDE  system
\begin{equation}
  \label{driftless-tx}
  d \widehat{X}_v(t) = \xi_v \sigma (t, \widehat{X}_v) dB_v(t), \quad \widehat{X}_v(0) \sim \lambda_0,  \ \ \ v \in \V,
\end{equation} 
with $(B_v)_{v \in \V}$ as independent standard $d$-dimensional Brownian motions, and with $(B_v)_{v \in \V}$, $(\xi_v)_{v \in \V}$, and with $(\widehat{X}_v(0))_{v \in \V}$  independent. Note that the  SDE \eqref{driftless-tx} is well-posed due to Assumption (\ref{assumption:A}.4).
Note in particular that $\W$ is an i.i.d.\ product measure.

To show that $P^n$ of \eqref{def-pn} is a second-order MRF, we will study how its density with respect to $\W$ factorizes, and then apply Theorem \ref{th:hammersleyclifford2}. As a first step, we identify the density of the $\{0,1\}^{\V_{n,n}}$-marginal:

\begin{lemma} \label{le:RNderivative}
Suppose $\rho$ has a finite nonzero first moment.
The law of $(1_{\{v \in \tree\}})_{v \in \V_{n,n}}$ on $\{0,1\}^{\V_{n,n}}$ is absolutely continuous with respect to that of $(\xi_v)_{v \in \V_{n,n}}$. Moreover, the Radon-Nikodym derivative is of the form
\begin{align}
F_n((a_v)_{v \in \V_{n,n}}) = f_{\o}(a_{\o},(a_k)_{k=1}^n)\prod_{v \in \V_{n-1,n} \backslash \{\o\}}f_1(a_v,(a_{vk})_{k=1}^n), \label{pf:GW-CI-2-1}
\end{align}
for measurable functions $f_{\o},f_1 : \{0,1\}^{n+1} \to \R_+$.
\end{lemma}
\begin{proof}
This is an easy consequence of the conditional independence structure of the tree $\tree$ and the fact that, aside from the root, every vertex has an identical offspring distribution.
\end{proof}

Next, we establish the desired second-order MRF 
property for $P^n$.
 We make use of the following notation. For $t > 0$, a set $A \subset \V$, and a probability measure $Q$ on $(\{0,1\} \times \C)^\V$, we write $Q_t$ and $Q_t[A]$ for the projections onto $(\{0,1\} \times \C_t)^\V$ and $(\{0,1\} \times \C_t)^A$, respectively. For example, $Q_t[A]$ is the image of $Q$ through the map $(a_v,x_v)_{v \in \V} \mapsto (a_v,x_v[t])_{v \in A}$.

\begin{proposition} \label{pr:GW-conditionalindependence}
Suppose Assumption \ref{assumption:A} holds, and assume the offspring distribution $\rho$ has a finite nonzero first moment.
Then, for each $t > 0$ and $n \ge 3$, the following hold:
\begin{enumerate}[(i)]
\item $(1_{\{v \in \tree_n\}},X^n_v[t])_{v \in \V}$ is a global second-order MRF. 
\item $(X^n_v[t])_{v \in \V}$ is a global second-order MRF.  
\end{enumerate} 
\end{proposition}
\begin{proof}
The property (ii) easily follows from (i), after noting as in Remark \ref{rem-treerecovery} that $1_{\{v \in \tree_n\}}$ is measurable
with respect to $X^n_v[t]$. Hence, we only prove (i).

Fix $t > 0$ and $n \geq 3$. Because the coordinates of $\V\setminus\V_{n,n}$ are all independent of those in $\V_{n,n}$, it clearly suffices to show that $(1_{\{v \in \tree_n\}},X^n_v[t])_{v \in \V_{n,n}}$ is a global second-order MRF. By Definition \ref{def-2MRF}, we must show that
\begin{align}
(1_{\{v \in \tree_n\}},X^n_v[t])_{v \in A} \, \, \indep \, \, (1_{\{v \in \tree_n\}},X^n_v[t])_{v \in B} \,\, \Bigr| \, \, (1_{\{v \in \tree_n\}},X^n_v[t])_{v \in \partial^2A}, \label{pf:GW-CI-1}
\end{align}
for any sets $A,B \subset \V_{n,n}$ with $B \cap (A \cup \partial^2A) = \emptyset$, where  $\partial^2$ denotes the double boundary operation in the tree $\V_{n,n}$. 
Recall that  $P^n_t[\V_{n,n}]$ is the restriction of the law $P^n$ of the random process
$(1_{\{v \in \tree_n\}}, \, X_v^n)_{v \in \V}$
in \eqref{def-pn} to  $(\{0,1\} \times \C_t)^{\V_n}$, and similarly for $\W_t[\V_{n,n}]$,
where $\W \in \P((\{0,1\} \times \C)^{\V})$ is the law of the process
  $(\xi_v, \widehat{X}_v)_{v \in \V}$ defined just prior to \eqref{driftless-tx}.  
To prove (i), we show that the density $dP^n_t[\V_{n,n}]/d\W_t[\V_{n,n}]$ admits a $2$-clique factorization in the sense of Definition \ref{def:2cliquefactorization}. 
To show this, we will use Girsanov's theorem to identify a conditional density
\emph{given the realization of the tree}, and then note that $dP^n_t[\V_{n,n}]/d\W_t[\V_{n,n}]$ is nothing but the product of this conditional density with the density of the law of $(1_{\{v \in \tree\}})_{v \in \V_{n,n}}$ with respect to the law of $(\xi_v)_{v \in \V_{n,n}}$, the form of which was identified in Lemma \ref{le:RNderivative}.

To identify this conditional density, we need a bit more notation. Define 
\[
\D_{n,n} := \{(1_{\{v \in T\}})_{v \in \V_{n,n}} \in \{0,1\}^{\V_{n,n}} : T \subset \V_{n,n} \text{ is a tree}\}.
\]
Define
$\widehat{\tree}_n : \D_{n,n} \rightarrow 2^{\V_{n,n}}$ by setting $\widehat{\tree}_n((1_{\{v \in T\}})_{v \in \V_{n,n}})=T$ for each tree $T \subset \V_{n,n}$,
and extend $\widehat{\tree}_n$ to all of $\{0,1\}^{\V_{n,n}}$ by (arbitrarily) setting $\widehat{\tree}_n(a) :=\{\o \}$ for $a \notin \D_{n,n}$.
Note that $(1_{\{v \in \tree_n\}})_{v \in \V_{n,n}}$ belongs a.s.\ to $\D_{n,n}$ and that 
$(\xi_v)_{v \in \V_{n,n}}$ is measurable with respect to $\widehat{\tree}_n((\xi_v)_{v \in \V_{n,n}})$ on the
event $\{(\xi_v)_{v \in \V_{n,n}} \in \D_{n,n}\}$.
We may additionally extend the domain $\widehat{\tree}_n$ to all of $(\{0,1\} \times \C)^{\V_{n,n}}$ by the identification $\widehat{\tree}_n((a_v,x_v)_{v \in \V_{n,n}}) = \widehat{\tree}_n((a_v)_{v \in \V_{n,n}})$.
Intuitively, under the measure $P^n_t[\V_{n,n}]$, 
  $\widehat{\tree}_n$ will represent the truncated random UGW($\rho$) tree $\tree_n$, with the advantage  that $\widehat{\tree}_n$ is defined on the canonical space $(\{0,1\} \times \C)^{\V_{n,n}}$.

Given these definitions, we may now identify the density of $P^n_t[\V_{n,n}]$ with respect to $\W_t[\V_{n,n}]$, conditionally on $\widehat{\tree}_n$.
Since $\V_{n,n}$ is a finite set,   
we may apply Girsanov's theorem in the form of Lemma \ref{le:ap:girsanov} (which is applicable since  \eqref{ap:asmp:girsanov} is satisfied due to  Assumption (\ref{assumption:A}.1) and Remark \ref{re:ap:girsanov}): recalling the definition of $X^n$ in \eqref{drifted-x}, the conditional density of $P^n_t[\V_{n,n}]$ with respect to $\W_t[\V_{n,n}]$ given $\widehat{\tree}_n$ is
\begin{align}	\label{dpndw}
\frac{dP^n_t[\V_{n,n}](\cdot\, | \, \widehat{\tree}_n)}{d\W_t[\V_{n,n}](\cdot\, | \, \widehat{\tree}_n)} &= \prod_{v \in \V_{n,n}}\EE_t(M^n_v),
\end{align}
where $\EE_t$ is the Doleans exponential defined in \eqref{def:doleans-exponential}, and
$M^n_v = M_v^n((a_v, x_v)_{v \in \V_{n,n}})$ is given by
\begin{align*}
M^n_v(t)((a_v,x_v)_{v \in \V_{n,n}}) &:= 1_{\{v \in \widehat{\tree}_n\}}\int_0^t (\sigma\sigma^\top)^{-1}b(s,x_v,x_{N_v(\widehat{\tree}_n)}) \cdot dx_v(s),
\end{align*}
where we suppressed the arguments $(a_v)_{v \in \V_{n,n}}$ of $\widehat{\tree}_n$.
Observe that for each $v_0 \in \V_{n,n}$, 
$M^n_{v_0}$ depends on $(a_v, x_v)_{v \in \V_{n,n}}$ only
through $a_{v_0}$, $x_{v_0}$ and 
$(a_v, x_v)_{v \in N_{v_0}(\V_{n,n})}$, recalling that $N_v(\V_{n,n})$ denotes the set of neighbors of $v$ within the tree $\V_{n,n}$.

Letting $F_n$ be as in Lemma \ref{le:RNderivative}, the entire (joint) density takes the form 
\begin{align*}
\frac{dP^n_t[\V_{n,n}]}{d\W_t[\V_{n,n}]}((a_v,x_v)_{v \in \V_{n,n}}) &= F_n((a_v)_{v \in \V_{n,n}})\frac{dP^n_t[\V_{n,n}](\cdot\, | \, \widehat{\tree}_n)}{d\W_t[\V_{n,n}](\cdot\, | \, \widehat{\tree}_n)}((a_v,x_v)_{v \in \V_{n,n}}).
\end{align*}
Together, \eqref{dpndw} and  Lemma \ref{le:RNderivative} imply that this can be rewritten as
\[
\frac{dP^n_t[\V_{n,n}]}{d\W_t[\V_{n,n}]}((a_v,x_v)_{v \in \V_{n,n}}) = \prod_{v \in \V_{n,n}}g_v^n((a_v, x_v),  (a_u,x_u)_{u \in N_v(\V_{n,n})}),
\]
for appropriate functions $(g_v^n)_{v \in \V_{n,n}}$.  More precisely, with $f_{\o}$ and $f_1$ as in Lemma \ref{le:RNderivative},
we have
\[
g^n_v((a_v, x_v), (a_u,x_u)_{u \in N_v(\V_{n,n})}) = \begin{cases}
 f_{\o}(a_{\o},(a_k)_{k=1}^n)\EE_t(M^n_{\o}) &\text{if } v = \o, \\
f_1(a_v,(a_{vk})_{k=1}^n)\EE_t(M^n_v) &\text{if } v \in \V_{n-1,n} \backslash \{\o\}, \\
\EE_t(M^n_v) &\text{if } v \in \V_{n,n} \backslash \V_{n-1,n}. 
\end{cases}
\]
Observing   that  for each $v \in \V_{n,n}$, the set $\{v\} \cup N_v(\V_{n,n})$ is a $2$-clique in $\V_{n,n}$, property (i) now follows from  Theorem \ref{th:hammersleyclifford2}.  
\end{proof}

\subsection{Convergence to the infinite system}
\label{subs-condinftree}

With the second-order MRF property now established for the truncated systems $P^n$, we wish to pass to the limit $n\rightarrow\infty$ to deduce a similar property for the infinite system.
We begin by checking that the law $P^n$ of $(1_{\{v \in \tree_n\}},X^n_v)_{v \in \V}$ converges to the law $P$ of $(1_{\{v \in \tree\}},X_v)_{v \in \V}$ and also that conditional laws converge in a suitable sense, where we recall that $X^n$ and $X$, respectively, denote the solutions of the SDE systems \eqref{drifted-x} and \eqref{statements:SDE}.

\begin{lemma} \label{le:infinitegraphlimit-GW}
Suppose Assumption \ref{assumption:A} holds.  Assume also that $\rho$ has a finite nonzero first moment. 
Then $P^n \rightarrow P$ weakly on $(\{0,1\} \times \C)^\V$.
Moreover, for any $k \in \N$, any $t > 0$, and any bounded continuous function $\varphi : \C_t^{\Vmb} \rightarrow \R$, we have
\begin{align}
\E\big[\varphi(X^n_{\Vmb\setminus\{\o,k\}}[t],X^n_{\{\o,k\}}[t]) \,\big|\, X^n_{\{\o,k\}}[t]\big] \Rightarrow \E\big[\varphi(X_{\Vmb\setminus\{\o,k\}}[t],X_{\{\o,k\}}[t]) \,\big|\, X_{\{\o,k\}}[t]\big],	\label{def:infinitegraphlimit-GW}
\end{align} 
where we recall that $\Rightarrow$ denotes convergence in law.
\end{lemma}
\begin{proof} 
Recall $\tree_n = \tree \cap \V_{n,n}$, where $\tree$ is a UGW($\rho$) tree, and $\V_{n,n}$ is as defined in \eqref{def-vn}, 
and note that $(1_{\{v \in \tree_n\}})_{v \in \V}$ therefore converges in law to $(1_{\{v \in \tree\}})_{v \in \V}$ in $\{0,1\}^\V$.  It is straightforward to check that the family of $\C$-valued random variables $\{X^n_v : v \in \V, \, n \in \N\}$ is tight, by standard arguments or by using the relative entropy estimates of Lemma \ref{le:lingrowth}.
Hence, $\{(1_{\{v \in \tree_n\}},X^n_v)_{v \in \V} : n \in \N\}$ is a tight family of $(\{0,1\} \times \C)^\V$-valued random variables. Let $(1_{\{v \in \tree\}},X^\infty_v)_{v \in \V}$ denote any weak limit point, and assume by Skorokhod representation that it is in fact an a.s.\ limit. For $m \in \N$, we have $N_v(\tree_n)=N_v(\tree)$ for all $n > m+1$ and $v \in \V_m$, and using weak convergence of stochastic integrals (see \cite[Theorem 2.2]{Kurtz-Protter}) we deduce that $(X^\infty_v)_{v \in \V_m}$ satisfies
\begin{align*}
dX^\infty_v(t) &= 1_{\{v \in \tree\}}\left(b(t,X^\infty_v,X^\infty_{N_v(\tree)})dt + \sigma(t,X^\infty_v)dW^\infty_v(t)\right), \quad v \in \V_m,
\end{align*}
for some independent Brownian motions $(W^\infty_v)_{v \in \V_m}$. As this is true for each $m$, we deduce that $(X^\infty_v)_{v \in \V}$ and $(X_v)_{v \in \V}$ solve the same SDE system
\eqref{statements:SDE}. The SDE \eqref{statements:SDE} is unique in law by Assumption (\ref{assumption:A}.1) (and Remark \ref{rem-unique}), and so the law of $(1_{\{v \in \tree\}},X^\infty_v)_{v \in \V}$ must be $P:=\L((1_{\{v \in \tree\}},X_v)_{v \in \V})$, which shows that $P^n \to P$.

The second claim requires more care, and we will ultimately appeal to \cite[Theorem 2.1]{CrimaldiPratelli2005}, which gives a criterion for the weak convergence of conditional expectations.
We introduce the following systems that are parallel to $X^n$ and $X$ but are driftless for nodes in $\V_2$. 
Let $Q^n \in \P((\{0,1\} \times \C)^{\V})$ denote the law of 
\begin{equation}
\label{def-qn}
\Bigl(1_{\{ v \in \tree \}}, \, Y^n_v\Bigr)_{v \in \V},
\end{equation}
where $(Y_v^n(0))_{v \in \V}=(X_v(0))_{v \in \V}$ and $(Y^n_v)_{v \in \V}$ solves the SDE system
\begin{align}
  \label{drifted-yn}
  \begin{aligned}
  	dY_v^n (t) & = 1_{\{v \in \tree_n\}} \left( b(t,Y_v^n, Y_{N_v(\tree_n)}^n) dt + \sigma (t, Y_v^n) dW_v(t)\right), \qquad v \in \V \setminus \Vmb_2, \\
  	dY_v^n (t) & = 1_{\{v \in \tree_n\}} \sigma (t, Y_v^n) dW_v(t), \qquad v \in \Vmb_2. 
  \end{aligned}
\end{align}
Recall that the tree $\tree$, the initial conditions $(X_v(0))_{v \in \V}$, and the driving Brownian motions $(W_v)_{v \in \V}$ are independent.
{To see that the SDE \eqref{drifted-yn} is unique in law (and hence, $Q^n$ is well-defined), condition on the (finite) tree $\tree_n$,
  use the independence properties just stated, 
  the fact that the driftless SDE is unique in law by Assumption (\ref{assumption:A}.4) and  Lemma \ref{le:ap:girsanov}
  (along with Remark \ref{re:ap:girsanov} and Assumption (\ref{assumption:A}.1)).}   
  
Similarly, let $Q \in \P((\{0,1\} \times \C)^{\V})$ denote the law of 
\begin{equation}
\label{def-q}
\Bigl(1_{\{ v \in \tree \}}, \, Y_v\Bigr)_{v \in \V},
\end{equation}
where $(Y_v(0))_{v \in \V}=(X_v(0))_{v \in \V}$ and $(Y_v)_{v \in \V}$ solves the SDE system
\begin{align}
  \label{drifted-y}
  \begin{aligned}
  	dY_v (t) & = 1_{\{v \in \tree\}} \left( b(t,Y_v, Y_{N_v(\tree)}) dt + \sigma (t, Y_v) dW_v(t)\right), \qquad v \in \V \setminus \Vmb_2, \\
  	dY_v (t) & = 1_{\{v \in \tree\}} \sigma (t, Y_v) dW_v(t), \qquad v \in \Vmb_2. 
  \end{aligned}
\end{align}
That the SDE \eqref{drifted-y} is unique in law (and thus $Q$  is well-defined) 
  can be deduced by applying Lemma \ref{le:ap:girsanov2}  with $X^2 = (X^2_v)_{v \in \V}$ equal to the solution to  the SDE
  \eqref{statements:SDE}, which is unique in law by Remark \ref{rem-unique}, 
  and $X^1 = (Y_v)_{v \in \V}$ as above, 
   noting that  the two  differ only for $v$ in the finite set $\Vmb_2$, and that
 condition \eqref{asmp:ap:girs2} of Lemma \ref{le:ap:girsanov2} holds by Remark \ref{re:ap:girsanov} and Assumption (\ref{assumption:A}.1).  
  
It is easily checked that $Q^n \to Q$ weakly, using the same argument which showed that $P^n \to P$ above.  
Fix $t > 0$. 
We may now apply Girsanov's theorem, in the precise infinite-dimensional form developed in Lemma \ref{le:ap:girsanov2},
whose application is justified by the uniqueness in law of the SDEs in \eqref{drifted-yn} and \eqref{drifted-y}
and the fact that 
  the condition \eqref{asmp:ap:girs2}  holds on account of  Remark \ref{re:ap:girsanov} and Assumption (\ref{assumption:A}.1),   
to obtain 
\begin{align*}
\frac{dP^n_t}{dQ^n_t}((1_{\{v \in \tree_n\}},Y^n_v)_{v \in \V}) & = \EE_t \left( \sum_{v \in \Vmb_2} \int_0^\cdot 1_{\{v \in \tree_n\}} \sigma(s,Y^n_v)^{-1} b(s,Y^n_v,Y^n_{N_v(\tree_n)}) \cdot dW_v(s)\right), \\
\frac{dP_t}{dQ_t}((1_{\{v \in \tree\}},Y_v)_{v \in \V}) & = \EE_t \left( \sum_{v \in \Vmb_2} \int_0^\cdot 1_{\{v \in \tree\}} \sigma(s,Y_v)^{-1} b(s,Y_v,Y_{N_v(\tree)}) \cdot dW_v(s)\right).
\end{align*}
Note that the summations are a.s.\ finite, since all but finitely many of the indicators $1_{\{v \in \tree_n\}}$
and $1_{\{v \in \tree\}}$ are zero for $v \in \V_2$.

From the weak convergence $Q^n \rightarrow Q$ (of the laws of $(1_{v \in \tree_n}, Y^n_v)_{v \in \V}$ to
  that of $(1_{v \in \tree}, Y_v)_{v \in \V}$) and using weak convergence of stochastic integrals (see \cite[Theorem 2.2]{Kurtz-Protter}), we easily deduce the following weak convergence in $(\{0,1\} \times \C_t)^\V \times \R$:
\begin{align}	\label{eq:cvg_cond2}
\begin{split}
&\left( \Bigl(1_{\{ v \in \tree_n \}}, \, Y^n_v[t]\Bigr)_{v \in \Vmb}, \ \frac{dP^n_t}{dQ^n_t}((1_{\{v \in \tree_n\}},Y^n_v)_{v \in \V}) \right) \\
	&\qquad\Rightarrow \left( \Bigl(1_{\{ v \in \tree \}}, \, Y_v[t]\Bigr)_{v \in \Vmb}, \ \frac{dP_t}{dQ_t}((1_{\{v \in \tree\}},Y_v)_{v \in \V}) \right). 
\end{split}  
\end{align}
To use this to deduce the desired convergence of related  conditional distributions, we now verify an additional 
condition in  \cite[Theorem 2.1]{CrimaldiPratelli2005}.  
Fix $k \in \N$ and a bounded continuous function $g$ on $(\{0,1\} \times \C_t)^{\V\setminus\{\o,k\}}$.
It is clear from the form of \eqref{drifted-yn} that $(1_{\{ v \in \tree_n \}}, Y^n_v[t])_{v \in \V \setminus\{\o,k\}}$ and $(Y^n_v[t])_{v \in \{\o,k\}}$ are conditionally independent given $\{k \in \tree\}$, and similarly with $(\tree_n,Y^n)$ replaced by $(\tree,Y)$, when $n > k$.
For $n > k$, we have $\{k \in \tree_n\}=\{k \in \tree\}$, and thus
\begin{align*}
\E & \left[g\left((1_{\{ v \in \tree_n \}}, \, Y^n_v[t])_{v \in \Vmb\setminus\{\o,k\}}\right) \Big| (1_{\{ v \in \tree_n \}}, \, Y^n_v[t])_{v \in \{\o,k\}} \right]  \\
	& \qquad = \E \left[g\left((1_{\{ v \in \tree_n \}}, \, Y^n_v[t])_{v \in \Vmb\setminus\{\o,k\}}\right) \Big| 1_{\{ k \in \tree \}} \right]  \\
	& \qquad = \frac{1_{\{ k \in \tree \}}}{\PP(k \in \tree)} \E \left[g\left((1_{\{ v \in \tree_n \}}, \, Y^n_v[t])_{v \in \Vmb\setminus\{\o,k\}}\right) 1_{\{ k \in \tree \}} \right]   \\
	& \qquad \quad + \frac{1_{\{ k \notin \tree \}}}{\PP(k \notin \tree)} \E \left[g\left((1_{\{ v \in \tree_n \}}, \, Y^n_v[t])_{v \in \Vmb\setminus\{\o,k\}}\right) 1_{\{ k \notin \tree \}} \right]  \\
	& \qquad \Rightarrow \frac{1_{\{ k \in \tree \}}}{\PP(k \in \tree)} \E \left[g\left((1_{\{ v \in \tree \}}, \, Y_v[t])_{v \in \Vmb\setminus\{\o,k\}}\right) 1_{\{ k \in \tree \}} \right]  \\
	& \qquad \quad + \frac{1_{\{ k \notin \tree \}}}{\PP(k \notin \tree)} \E \left[g\left((1_{\{ v \in \tree \}}, \, Y_v[t])_{v \in \Vmb\setminus\{\o,k\}}\right) 1_{\{ k \notin \tree \}} \right]  \\
	& \qquad = \E \left[g\left((1_{\{ v \in \tree \}}, \, Y_v[t])_{v \in \Vmb\setminus\{\o,k\}}\right) \Big| 1_{\{ k \in \tree \}} \right] \\
	& \qquad = \E \left[g\left((1_{\{ v \in \tree \}}, \, Y_v[t])_{v \in \Vmb\setminus\{\o,k\}}\right) \Big| (1_{\{ v \in \tree \}}, \, Y_v[t])_{v \in \{\o,k\}} \right]. 
\end{align*}
 This and \eqref{eq:cvg_cond2} are precisely the two conditions assumed in \cite[Theorem 2.1]{CrimaldiPratelli2005}, which we may now apply to deduce that
\begin{align*}
&\E\left[g\left((1_{\{ v \in \tree_n \}}, \, X^n_v[t])_{v \in \V}\right) \Big| (1_{\{ v \in \tree_n \}}, \, X^n_v[t])_{v \in \{\o,k\}} \right] \\
	&\qquad\quad\Rightarrow \E\left[g\left((1_{\{ v \in \tree \}}, \, X_v[t])_{v \in \V}\right) \Big| (1_{\{ v \in \tree \}}, \, X_v[t])_{v \in \{\o,k\}} \right],
\end{align*}
for each bounded continuous function $g$ on $(\{0,1\} \times \C)^\V$. Specializing to functions on $\C^\V$ yields the claim \eqref{def:infinitegraphlimit-GW}. 
\end{proof}

\subsection{Proof of Proposition \ref{pr:properties-GW}}

We finally prove Proposition \ref{pr:properties-GW}, starting with claim (i).
Fix $k \in \N$ and let $C_k = \{ki : i \in \N\}$.
Fix two bounded continuous functions $f$ and $g$ on $\C_t^{C_k}$ and $\C_t^{\V_1}$.
From Lemma \ref{le:infinitegraphlimit-GW} we have that
\begin{align*}
	& s_1 \E[f(X^n_{C_k}[t]) \, | \, X^n_{\{\o,k\}}[t]] + s_2 \E[g(X^n_{\V_1}[t]) \, | \, X^n_{\{\o,k\}}[t]] \\
	& \Rightarrow s_1 \E[f(X_{C_k}[t]) \, | \, X_{\{\o,k\}}[t]] + s_2 \E[g(X_{\V_1}[t]) \, | \, X_{\{\o,k\}}[t]]
\end{align*}
for every $s_1, s_2 \in \Rmb$.
Therefore, by the Cram\'{e}r-Wold theorem, 
\begin{align*}
	& \Big(\E[f(X^n_{C_k}[t]) \, | \, X^n_{\{\o,k\}}[t]], \ \E[g(X^n_{\V_1}[t]) \, | \, X^n_{\{\o,k\}}[t]] \Big) \\
	& \Rightarrow \Big(\E[f(X_{C_k}[t]) \, | \, X_{\{\o,k\}}[t]], \ \E[g(X_{\V_1}[t]) \, | \, X_{\{\o,k\}}[t]]\Big).
\end{align*}
 By Proposition \ref{pr:GW-conditionalindependence}(i), $X^n_{C_k}[t]$ and $X^n_{\V_1}[t]$ are conditionally independent given $X^n_{\{\o,k\}}[t]$ for each $n$; indeed, apply Definition \ref{def-2MRF} of a second-order MRF with the set $A$ given as the set of all descendants of $k$, so that $\partial^2A=\{\o,k\}$. Thus, we have 
\begin{align*}
	\E[f(X_{C_k}[t])g(X_{\V_1}[t])] &= \lim_{n\to\infty}\E[f(X^n_{C_k}[t])g(X^n_{\V_1}[t])] \\
	&= \lim_{n\to\infty}\E\left[\E\big[f(X^n_{C_k}[t]) \, | \, X^n_{\{\o,k\}}[t]\big] \, \E\big[g(X^n_{\V_1}[t]) \, | \, X^n_{\{\o,k\}}[t]\big]\right] \\
	&= \E\left[ \E\big[f(X_{C_k}[t]) \, | \, X_{\{\o,k\}}[t]\big] \, \E\big[g(X_{\V_1}[t]) \, | \, X_{\{\o,k\}}[t]\big]\right].
\end{align*}
As this holds for any pair of bounded continuous functions $(f,g)$, we conclude as desired that $X_{C_k}[t]$ and $X_{\V_1}[t]$ are conditionally independent given $X_{\{\o,k\}}[t]$.

To prove part (ii) of Proposition \ref{pr:properties-GW}, we use a symmetry argument. Fix $k \in \N$, and let $\varphi : \V \to \V$ denote the transposition of the subtrees rooted at $1$ and $k$, defined by setting $\varphi(1u)=ku$ and $\varphi(ku)=1u$ for all $u \in \V$ as well as $\varphi(v)=v$ for all $v \in \V$ which satisfy neither $v \ge 1$ nor $v \ge k$ (i.e., for all $v \in \V$
  that are not descendants of $1$ or $k$). Due to the recursive structure of the tree $\tree \sim \mathrm{UGW}(\rho)$, we have $\L(\tree \, | \, k \in \tree) = \L(\varphi(\tree) \, | \, k \in \tree)$. Using uniqueness of the SDE system in Assumption (\ref{assumption:A}.4), we deduce that $\L(X_{\o},X_1,(X_{1j})_{j \in \N} \, | \, k \in \tree) = \L(X_{\o},X_k,(X_{kj})_{j \in \N} \, | \, k \in \tree)$.
Now, fix $t > 0$ and let $\mmap_t : \C \times \C \to \P(\C_t^\N)$ denote a version of the conditional law of $(X_{1j}[t])_{j \in \N}$ given $(X_1[t],X_{\o}[t])$. Then,  for bounded measurable functions $f,g,h$, 
  we combine this symmetry property with  the conditional independence of Proposition \ref{pr:properties-GW}(i) proven above to obtain
\begin{align*}
\E\left[f(X_{\o}[t])g(X_k[t])h((X_{kj}[t])_{j \in \N})1_{\{k \in \tree\}}\right] &= \E\left[f(X_{\o}[t])g(X_1[t])h((X_{1j}[t])_{j \in \N})1_{\{k \in \tree\}}\right] \\
	&= \E\left[f(X_{\o}[t])g(X_1[t])\lan \mmap_t(X_1,X_{\o}),\, h\ran \, 1_{\{k \in \tree\}}\right] \\
	&= \E\left[f(X_{\o}[t])g(X_k[t])\lan \mmap_t(X_k,X_{\o}),\, h\ran \, 1_{\{k \in \tree\}}\right].
\end{align*}
Indeed, the second step followed from the conditional independence of $(X_{1j}[t])_{j \in \N}$ and $\{k \in \tree\}$ (which is $X_k[t]$-measurable by Remark \ref{rem-treerecovery}) given $(X_{\o}[t],X_1[t])$.
This shows that
\begin{align}
\lan \mmap_t(X_k,X_{\o}), \, h \ran = \E\left[h((X_{kj}[t])_{j \in \N}) \, \big| \, X_k[t],X_{\o}[t]\right], \ \ a.s. \ \  \text{on } \{k \in \tree\}. \label{pf:symmetry-GW-1}
\end{align}
Recalling how $\mmap_t$ was defined above, the proof would now be complete if not for the qualification ``on $\{k \in \tree\}$," so we lastly take care of the complementary set.
Let $Y_v(t)=X_v(0)$ for all $t \ge 0$ and $v \in \V$, and note that $Y_v = X_v$ a.s.\ on $\{v \notin \tree\}$ by construction. Note also that  $(Y_v)_{v \in \V}$ are i.i.d. On the event $\{k \notin \tree\}$, we know $X_{kj} \equiv Y_{kj}$ for all $j \in \N$, and so
\begin{align}
\E\left[h((X_{kj}[t])_{j \in \N}) \, \big| \, X_k[t],X_{\o}[t]\right] = \E\left[h((Y_{kj}[t])_{j \in \N})\right] = \E\left[h((Y_{1j}[t])_{j \in \N})\right], \ \ a.s. \label{pf:symmetry-GW-2}
\end{align}
Repeating this independence argument with $k=1$ and using the definition of $\Lambda_t$, we find
\begin{align}
\lan \mmap_t(X_1,X_{\o}), \, h \ran &= \E\left[h((Y_{1j}[t])_{j \in \N})\right], \ \ a.s. \ \  \text{on } \{1 \notin \tree\}.\label{pf:symmetry-GW-3}
\end{align}
Recalling from Remark \ref{rem-treerecovery} that there is a measurable function $\detfn$ such that $1_{\{v \in \tree\}} = \detfn(X_v)$ a.s.\ for each $v$, it is straightforward to deduce  from \eqref{pf:symmetry-GW-2} and \eqref{pf:symmetry-GW-3} that the same identity \eqref{pf:symmetry-GW-1} holds also on the event $\{k \notin \tree\}$. \hfill \qed

\section{Proof of the symmetry property}
\label{se:proofs-GW-invariance}

The last remaining point is to prove Proposition \ref{pr:invariance-GW}, which was the second key ingredient in the 
first (verification) part of Theorem \ref{th:statements-mainlocaleq}. As a first step, in Section \ref{se:proof-GW-exchangeability} we show that the children of the root are exchangeable, in a suitable conditional sense. Then,
in Section \ref{subs-unimod}, 
we use unimodularity to prove Proposition \ref{pr:invariance-GW}.
Recall here that for a finite set $A$ and for $x_A \in \X^A$ we write $\lan x_A \ran$ for the corresponding element (equivalence class) in $\SQ(\X)$.

\subsection{Conditional exchangeability at the generation level} \label{se:proof-GW-exchangeability}

We first show how to use  Proposition \ref{pr:properties-GW} to derive a useful conditional exchangeability property.

\begin{lemma} \label{le:GW-exchangeability}
  Suppose Assumption \ref{assumption:A} holds, and assume that $\rho \in \P(\N_0)$ has a finite nonzero first moment.
  For each $t > 0$ and each bounded measurable function $\h: \C_t^2 \times \SQ(C_t)^2 \to \R$,
  it holds almost surely on the event $\{1 \in \tree\}$ that 
  \begin{align}
    \notag 
    \E&\left[\left. \frac{1}{|N_{\o}(\tree)|}\sum_{k \in N_{\o}(\tree)}\h\big(X_{\o}[t],X_k[t],\lan X_{N_{\o}(\tree)}[t] \ran,\lan X_{N_k(\tree)}[t] \ran \big) \, \right| \, X_{\o}[t], \, \lan X_{N_{\o}(\tree)}[t]\ran\right] \\
    \label{claim-exchange}
	&= \E\left[\left. h\big(X_{\o}[t],X_1[t],\lan X_{N_{\o}(\tree)}[t] \ran,\lan X_{N_1(\tree)}[t] \ran \big) \, \right| \, X_{\o}[t], \, \lan X_{N_{\o}(\tree)}[t]\ran\right]. 
\end{align} 
\end{lemma}
\begin{proof}
 	We first prove \eqref{claim-exchange} assuming that $h$ has the following form: there exists a
    bounded measurable mapping $f : \C_t^2 \to \R$ such that 
    \begin{equation}
      \label{form-g}
      h(x,y,\tilde x, \tilde y) = f(x,y), \qquad x,y \in \C_t,  \ \ \  \tilde x, \tilde y \in \SQ(C_t). 
    \end{equation}  
    Fix $k,n \in \N$ with $k \le n$, and let $\varphi : \V \to \V$ denote the transposition of the subtrees rooted at $1$ and $k$, defined by setting $\varphi(1u)=ku$ and $\varphi(ku)=1u$ for all $u \in \V$ as well as $\varphi(v)=v$ for all $v \in \V$ which satisfy neither $v \ge 1$ nor $v \ge k$ with respect to the Ulam-Harris-Neveu labeling (i.e., for all $v \in \V$  that are  neither descendants of  $1$  nor $k$).
     Due to the recursive structure of the tree $\tree \sim \mathrm{UGW}(\rho)$, we have $\L(\tree \, | \, |N_{\o}(\tree)| = n) = \L(\varphi(\tree) \, | \, |N_{\o}(\tree)| = n)$. 
    Using uniqueness of the SDE system in Assumption (\ref{assumption:A}.4), we deduce that 
    $$\L(X_{\o}[t], X_1[t] \, | \, X_{\o}[t], \, \lan X_{N_{\o}(\tree)}[t]\ran, \, |N_{\o}(\tree)| = n) = \L(X_{\o}[t], X_k[t] \, | \, X_{\o}[t], \, \lan X_{N_{\o}(\tree)}[t]\ran, \, |N_{\o}(\tree)| = n).$$
	From this we have
	\begin{align}
          \notag
	  \frac{1}{n} \sum_{k=1}^n f(X_{\o}[t],X_k[t]) & = \E\left[\left. \frac{1}{n} \sum_{k=1}^n f(X_{\o}[t],X_k[t]) \, \right| \, X_{\o}[t], \, \lan X_{N_{\o}(\tree)}[t]\ran, \, |N_{\o}(\tree)| = n\right] \\
          \label{favg}
		& = \E\left[\left. f(X_{\o}[t],X_1[t]) \, \right| \, X_{\o}[t], \, \lan X_{N_{\o}(\tree)}[t]\ran, \, |N_{\o}(\tree)| = n\right].
	\end{align}
	In other words, it holds a.s.\ on $\{1 \in \tree\} = \{N_{\o}(\tree) \neq \emptyset\}$ that 
	\begin{align*}
		\frac{1}{|N_{\o}(\tree)|} \sum_{k \in N_{\o}(\tree)} f(X_{\o}[t],X_k[t]) 
		& = \E\left[\left. f(X_{\o}[t],X_1[t]) \, \right| \, X_{\o}[t], \, \lan X_{N_{\o}(\tree)}[t]\ran, \, |N_{\o}(\tree)|\right].
	\end{align*} 
	Because $|N_{\o}(\tree)|$ is a.s.\ $\lan X_{N_{\o}(\tree)}[t]\ran$-measurable for each $t>0$,
        this implies 
	\begin{align*}
	\frac{1}{|N_{\o}(\tree)|} \sum_{k \in N_{\o}(\tree)} f(X_{\o}[t],X_k[t]) = \E\left[\left. f(X_{\o}[t],X_1[t]) \, \right| \, X_{\o}[t], \, \lan X_{N_{\o}(\tree)}[t]\ran\right],
	\end{align*} 
	again on the event $\{1 \in \tree\}$.  Thus,  the proof is complete for $h$ of the form \eqref{form-g}.

        We now prove \eqref{claim-exchange} for general $h$.
        Since both sides of  \eqref{claim-exchange} are conditional on $X_{\o}[t]$ and $\langle X_{N_{\o}(\tree)}[t] \rangle$, by  general measure-theoretic considerations, it suffices to prove
          the relation \eqref{claim-exchange} for $h(x,y,\tilde{x},\tilde{y})=g(y,\tilde{y})$ depending only on the variables that are not being conditioned upon. That is, 
it suffices to show that for all bounded measurable functions $g : \C_t \times \SQ(\C_t) \to \R$ we have
\begin{align}
\E&\left[\left. \frac{1}{|N_{\o}(\tree)|}\sum_{k \in N_{\o}(\tree)}g( X_k[t] ,\lan X_{N_k(\tree)}[t] \ran) \, \right| \, X_{\o}[t], \, \lan X_{N_{\o}(\tree)}[t]\ran\right] \nonumber \\
	&= \E\left[\left. g( X_1[t],\lan X_{N_1(\tree)}[t] \ran) \, \right| \, X_{\o}[t], \, \lan X_{N_{\o}(\tree)}[t]\ran\right], \ \ \ a.s., \text{ on } \{1 \in \tree\}. \label{pf:symmetry-goal1}
\end{align}
To prove this, recall first from Proposition \ref{pr:properties-GW}(ii) that there is a measurable function $\mmap_t : \C_t^2 \to \P(\C_t^\N)$ such that
\[
\mmap_t(X_k[t],X_{\o}[t]) = \L((X_{ki}[t])_{i \in \N} \, | \, X_k[t],X_{\o}[t]), \ \ \ a.s., \text{ on } \{k \in \tree\}.
\]
Using the conditional independence of Proposition \ref{pr:properties-GW}(i), we have also
\begin{align}
\mmap_t(X_k[t],X_{\o}[t]) = \L((X_{ki}[t])_{i \in \N} \, | \, X_{\V_1}[t]), \ \ \ a.s., \text{ on } \{k \in \tree\}. \label{pf:symmetry-2mrf}
\end{align}
Noting again that $|N_{\o}(\tree)|$ is $\lan X_{N_{\o}(\tree)}[t] \ran$-measurable, we may use the tower property of conditional expectation (and other relations specified below) to obtain, on $\{1 \in \tree\}$,
\begin{align*}
\E&\left[\left. \frac{1}{|N_{\o}(\tree)|}\sum_{k \in N_{\o}(\tree)} g( X_k[t] ,\lan X_{N_k(\tree)}[t] \ran) \, \right| \, X_{\o}[t], \, \lan X_{N_{\o}(\tree)}[t]\ran\right] \\
&= \E\left[\left. \frac{1}{|N_{\o}(\tree)|}\sum_{k \in N_{\o}(\tree)} \E\big[g( X_k[t] ,\lan X_{N_k(\tree)}[t] \ran) \, | \, X_{\V_1}[t]\big] \,  \, \right| \, X_{\o}[t], \, \lan X_{N_{\o}(\tree)}[t]\ran\right].
\\
	&= \E\left[\left. \frac{1}{|N_{\o}(\tree)|}\sum_{k \in N_{\o}(\tree)} \left\lan \mmap_t(X_k[t],X_{\o}[t]), \, g( X_k[t] , \lan\cdot\ran )\right\ran  \, \right| \, X_{\o}[t], \, \lan X_{N_{\o}(\tree)}[t]\ran\right] \\
	&= \E\left[\left. \lan \mmap_t(X_1[t],X_{\o}[t]), \, g( X_1[t] ,\lan\cdot\ran )\ran  \, \right| \, X_{\o}[t], \, \lan X_{N_{\o}(\tree)}[t]\ran\right],
\end{align*}
where the second equality used \eqref{pf:symmetry-2mrf} and our short-hand notation
  $\langle \nu, f\rangle = \int f d\nu$  for any measure $\nu$ and $\nu$-integrable function $f$, and the last equality used the relation 
\eqref{favg} with $f(x_{\o}, x_k)= \langle \Lambda_t (x_k,x_{\o}), g (x_k, \langle \cdot \rangle) \rangle$ for $x_{\o}, x_k \in \C_t^2$. 
Now, apply \eqref{pf:symmetry-2mrf} once again to rewrite the right-hand side as
\begin{align*}
\E &\Big[   \E\big[g( X_1[t] , \lan X_{N_1(\tree)}[t]\ran ) \, | \, X_{\V_1}[t]\big]  \, \Big| \, X_{\o}[t], \, \lan X_{N_{\o}(\tree)}[t]\ran\Big] \\
	&= \E\left[\left.  g( X_1[t] , \lan X_{N_1(\tree)}[t]\ran )  \, \right| \, X_{\o}[t], \, \lan X_{N_{\o}(\tree)}[t]\ran\right],  \ \ \text{ on } \{1 \in \tree\}.
\end{align*}
This shows \eqref{pf:symmetry-goal1}, thus completing the proof of the lemma. 
\end{proof}

\subsection{Unimodular random graphs}
\label{subs-unimod}

So far we only needed the notion of a unimodular Galton-Watson tree, which could
be defined simply as in Definition \ref{def-UGW}. 
However, the final step of the proof of Proposition \ref{pr:invariance-GW} uses
 crucially the notion of \emph{unimodularity}
 on general graphs,  which we now briefly define; refering to \cite{aldous-lyons} for a more thorough discussion. 
 For this,  we will need to introduce 
the notation for (doubly) rooted (marked) graphs.  We recall the general graph terminology introduced
in Section \ref{subs-graphs}.

A \emph{rooted graph} $(G,o)$ is a connected
graph equipped with a distinguished vertex $o$, where we assume $G$ has finite or countable vertex set and is locally finite, meaning each vertex has finitely many neighbors.
An \emph{isomorphism} from one rooted graph $(G_1,o_1)$ to another $(G_2,o_2)$ is a bijection $\varphi$ from the vertex set of $G_1$ to that of $G_2$ such that $\varphi(o_1)=o_2$ and such that $(u,v)$ is an edge in $G_1$ if and only if $(\varphi(u),\varphi(v))$ is an edge in $G_2$.
We say two rooted graphs  are \emph{isomorphic} if there exists an isomorphism \ between them, and we let $\G_*$  denote the set of isomorphism classes of rooted graphs. Similarly, a \emph{doubly rooted graph} $(G,o,o')$ is a rooted graph $(G,o)$ with an additional distinguished vertex $o'$ (which may equal $o$). Two doubly rooted graphs $(G_i,o_i,o_i')$ are isomorphic if there is an isomorphism from $(G_1,o_1)$ to $(G_2,o_2)$ which also maps $o_1'$ to $o_2'$. We write $\G_{**}$ for the set of isomorphism classes of doubly rooted graphs.

There are analogous definitions for \emph{marked rooted graphs}.
An \emph{$\X$-marked rooted graph} is a tuple $(G,x,o)$, where $(G,o)$ is a rooted graph and $x=(x_v)_{v \in G} \in \X^G$ is a vector of marks, indexed by vertices of $G$.
We say that two marked rooted graphs $(G_1,x^1,o_1)$ and $(G_2,x^2,o_2)$  are \emph{isomorphic} if there exists an isomorphism $\varphi$
between the rooted graphs $(G_1,o_1)$ and $(G_2,o_2)$ that maps the marks of one to the marks of the other (i.e., for which $x^1_{\varphi(v)} = x^2_v$  for all $v \in G$). Let  $\G_*[\X]$ denote the set of isomorphism classes of $\X$-marked rooted graphs.
A double rooted marked graph is defined in the obvious way, and $\G_{**}[\X]$ denotes the set of isomorphism classes of doubly rooted marked graphs.

These spaces of graphs come with natural topologies. For $r \in \N$ and $(G,o)\in\G_*$, let $B_r(G,o)$ denote the induced subgraph of $G$ (rooted at $o$) containing only those vertices with (graph) distance at most $r$ from the root $o$. The distance between $(G_1,o_1)$ and $(G_2,o_2)$ is defined as the value $1/(1+\bar{r})$, where $\bar{r}$ is the supremum over
$r \in \N_0$ such that $B_r(G_1,o_1)$ and $B_r(G_2,o_2)$ are isomorphic,
where we interpret $B_0(G_i,o_i) = \{o_i\}$.  The distance between two marked graphs $(G_i,x^i,o_i)$, $i = 1, 2$, is likewise defined as the value $1/(1+\bar{r})$, where $\bar{r}$ is the supremum over $r \in \N_0$ such that there exists an isomorphism $\varphi$ from $B_r(G_1,o_1)$ to $B_r(G_2,o_2)$ such that $d(x^1_v,x^2_{\varphi(v)}) \le 1/r$ for all $v \in B_r(G_1,o_1)$. We equip $\G_{**}$ and $\G_{**}[\X]$ with similar metrics, just using the union of the balls at the two roots, $B_r(G,o) \cup B_r(G,o')$, in place of the ball around a single root $B_r(G,o)$. Metrized in this manner, the spaces $\G_*$ and $\G_{**}$ are Polish spaces, as are $\G_*[\X]$ and $\G_{**}[\X]$ if $\X$ is itself a Polish space. See \cite[Lemma 3.4]{Bordenave2016}  (or  \cite[Appendix A]{LacRamWu20a}) for a proof
that $\G_*[\X]$ is a Polish space. 
Each space $\G_*[\X]$ and $\G_{**}[\X]$ is equipped with its Borel $\sigma$-algebra.

We are now ready to introduce the definition of unimodularity for general graphs.

\begin{definition}  \label{def-unimodular}
For a metric space $\X$,  
we say that a $\G_*[\X]$-valued random element $(G,X,o)$ is \emph{unimodular}
if the following \emph{mass-transport principle} holds: for every (non-negative) bounded Borel measurable function $F : \G_{**}[\X] \rightarrow \R_+$,
\begin{align}
\E\left[\sum_{o' \in G}F(G,X,o,o')\right] = \E\left[\sum_{o' \in G}F(G,X,o',o)\right].   \label{def:unimodularity}
\end{align}
A $\G_*$-valued random variable $(G,o)$ is said to be \emph{unimodular} if the same identity holds, but
with $X$ removed, that is,  if for every bounded Borel measurable function  $F:\G_{**} \rightarrow \R_+$,
\[
\E \left[ \sum_{o' \in G}  F(G,o,o') \right] = \E \left[ \sum_{o' \in G}  F(G,o',o) \right].
\]
\end{definition}

Recalling the canonical Ulam-Harris-Neveu labeling introduced in Section \ref{subsub-labeling},
as described therein, 
a (countable, locally finite) tree may always be viewed  as a subset of $\V$ satisfying the appropriate properties. Recall that $\o\in\V$ denotes the root of any tree in this canonical labeling,  and let $\mathbb{T}_*$ denote the collection
  of subsets of $\V$  described in Section \ref{subsub-labeling} that define a rooted tree. 
A tree $\tree \in  \mathbb{T}_*$ induces an element $(\tree,\o)$ of $\G_*$, and we say a random
($\mathbb{T}_*$-valued) tree $\tree$ is unimodular if $(\tree,\o)$ is a unimodular random graph in the sense of Definition \ref{def-unimodular}.
 
Recall from Assumption (\ref{assumption:A}.4) and Remark \ref{rem-unique} that there is a unique solution $X^\tree=(X_v^\tree)_{v \in \V}$ to the system \eqref{statements:SDE} for any tree $\tree \in \mathbb{T}_*$.
We may then view $(\tree,(X^\tree_v)_{v \in \tree},\o)$ as a rooted graph marked by the trajectories of the process $X^\tree$, i.e., as a $\G_*[\C]$-valued random element.  

\begin{proposition} \label{pr:unimodular}
Suppose Assumption \ref{assumption:A} holds. Let $\tree$ be any unimodular ($\mathbb{T}_*$-valued) random tree, and let $X^\tree=(X^\tree_v)_{v \in \V}$ be the unique solution of the SDE system \eqref{statements:SDE}.
Then the $\G_*[\C]$-valued random variable $(\tree,(X^\tree_v)_{v \in \tree},\o)$ is unimodular.
\end{proposition}
\begin{proof}
It will help to temporarily free ourselves from the canonical labels of $\V$. For any (countable, locally finite) tree $\tree$ (labeled in any manner), consider the SDE system
\begin{align}
dX^\tree_v(t) = b(t,X^\tree_v,X^\tree_{N_v(\tree)})dt + \sigma(t,X^\tree_v)dW_v(t), \quad v \in \tree, \label{pf:unimod-SDE}
\end{align}
where $N_v(\tree)$ denotes the neighbors of $v$ in $\tree$, $(W_v)_{v \in \tree}$ are independent Brownian motions, and $(X_v(0))_{v \in \tree}$ are i.i.d.\ with law $\lambda_0$. Note that this SDE system is unique in law by Assumption (\ref{assumption:A}.4), as the tree $\tree$ can always be viewed up to isomorphism as a subset of $\V$.
For any non-random doubly rooted tree $(\tree,o_1,o_2)$, the unique solution of \eqref{pf:unimod-SDE}
gives rise to a $\C^\tree$-valued random variable $X^\tree=(X^\tree_v)_{v \in \tree}$, which in turn induces a $\G_{**}[\C]$-valued random variable $(\tree,(X^\tree_v)_{v \in \tree},o_1,o_2)$, whose law we denote by $Q[\tree,o_1,o_2]$.

We claim first that $Q[\tree,o_1,o_2]=Q[\tree',o_1',o_2']$ whenever $(\tree,o_1,o_2)$ and $(\tree',o_1',o_2')$ are isomorphic as doubly rooted graphs.
To see this, let $\varphi : \tree \to \tree'$ denote any isomorphism. It is clear from the structure of the SDE \eqref{pf:unimod-SDE} that the $\C^\tree$-valued random elements $(X^{\tree'}_{\varphi(v)})_{v \in \tree}$ and $(X^{\tree}_v)_{v \in \tree}$ solve the same SDE  and thus have the same law, due to the aforementioned uniqueness in law. In particular, the $\G_{**}[\C]$-valued random variables $(\tree,(X^\tree_v)_{v \in \tree},o_1,o_2)$ and $(\tree',(X^{\tree'}_v)_{v \in \tree'},o_1',o_2')$ have the same law.

This shows that $Q[\tree,o_1,o_2]$ depends on $(\tree,o_1,o_2)$ only through its isomorphism class. We may thus view $Q$ as a (measurable) 
map from the set $\mathbb{T}_{**} \subset \G_{**}$ of doubly rooted trees to $\P(\G_{**}[\C])$.
(For a justification of the measurability of $Q$, see Remark \ref{rem-Q} below.) 
For a bounded measurable function $F : \G_{**} \to \R_+$, the function $\mathbb{T}_{**} \ni (\tree,o_1,o_2) \mapsto \langle Q[(\tree,o_1,o_2)],\, F\rangle \in \R_+$ is also bounded and measurable, and we extend it to be zero on $\G_{**}\setminus \mathbb{T}_{**}$. Then, for a given unimodular ($\mathbb{T}_*$-valued) random tree $\tree$, we have
(as justified subsequently) 
\begin{align*}
\E\left[\sum_{o \in \tree}F(\tree,(X^\tree_v)_{v \in \tree},\o,o)\right] &= \E\left[\sum_{o \in \tree}\E\big[F(\tree,(X^\tree_v)_{v \in \tree},\o,o) \,|\, \tree \big]\right] = \E\left[\sum_{o \in \tree} \big\langle Q[(\tree,\o,o)],\, F\big\rangle \right] \\
&	\hspace{2.0in} = \E\left[\sum_{o \in \tree} \big\langle Q[(\tree,o,\o)],\, F\big\rangle \right] \\
&	\hspace{2.0in}	= \E\left[\sum_{o \in \tree}\E\big[F(\tree,(X^\tree_v)_{v \in \tree},o,\o) \,|\, \tree\big]\right] \\
&	\hspace{2.0in}	= \E\left[\sum_{o \in \tree}F(\tree,(X^\tree_v)_{v \in \tree},o,\o)\right].
\end{align*}
Indeed, the second and fourth steps used the fact that a random tree $\tree \subset \V$ in the SDE system \eqref{statements:SDE} is always assumed to be independent of the Brownian motions and initial conditions, which ensures that the conditional law of $(\tree,(X^\tree_v)_{v \in \tree},\o,o)$ given $\tree$ is precisely $Q[\tree,\o,o]$.
\end{proof}

  \begin{remark}
    \label{rem-Q}
For completeness, we sketch here a proof of the measurability of $Q$ introduced in the last proof. 
For $r \in \N$ and $(G,o_1,o_2)$ for which the graph distance $d_G(o_1,o_2)$ is at most $r$, let $B_r(G,o_1,o_2) \in \G_{**}$ denote the union of the balls of radius $r$ around $o_1$ and $o_2$. The topology of the subspace $\{(\tree,o_1,o_2) \in \mathbb{T}_{**} : (\tree,o_1,o_2)=B_r(\tree,o_1,o_2)\}$ is discrete for each $r$, so the map $(\tree,o_1,o_2) \mapsto Q[B_r(\tree,o_1,o_2)]$ is trivially measurable for each $r$. To complete the proof, it suffices to argue that $\lim_{r\to\infty}Q[B_r(\tree,o_1,o_2)]=Q[(\tree,o_1,o_2)]$ for each $(\tree,o_1,o_2) \in \mathbb{T}_{**}$. If we fix a doubly rooted tree $(\tree,o_1,o_2)$ (with labels, i.e., not an element of $\G_{**}$ but rather a representative from an equivalence class therein), then straightforward weak convergence arguments show that, for each $k \in \N$, $(X^{B_r(\tree,o_1,o_2)}_v)_{v \in B_k(\tree,o_1,o_2)}$ converges in law to $(X^{(\tree,o_1,o_2)}_v)_{v \in B_k(\tree,o_1,o_2)}$ as $r\to\infty$, which proves the claim.
  \end{remark}

\begin{remark} \label{re:UGWunimodular}
  It is well known that a UGW($\rho$) tree $(\tree,\o)$ is unimodular (hence the name), for $\rho \in \P(\N_0)$ with finite nonzero first moment, and from Proposition \ref{pr:unimodular} we then deduce that $(\tree,(X^\tree_v)_{v \in \tree},\o)$ is unimodular. A direct proof of the mass-transport principle for $(\tree,\o)$ is attributed to \cite{lyons1995conceptual}, but
  one can argue instead by approximation by finite uniformly rooted graphs; see \cite[Example 10.2]{aldous-lyons} or \cite[Proposition 2.5]{dembo-montanari}.
\end{remark}

\subsection{Proof of Proposition \ref{pr:invariance-GW}} 
As in the statement of Proposition \ref{pr:invariance-GW}, let $h: \C_t^2 \times \SQ(\C_t) \mapsto \R$ be bounded
  and measurable.  To prove the proposition, we may assume without loss of generality that in addition $h \ge 0$. 
Fix $t > 0$, and let $g : \C_t^2 \to \R_+$ be any bounded measurable function. Because $t$ is fixed, throughout this proof we will omit the argument $[t]$ for the sake of readability, with the understanding that every appearance of $X_v$ below should be written more precisely as $X_v[t]$. 
Recall once more that for a finite set $A$ and for $x_A \in \X^A$ we write $\lan x_A \ran$ for the corresponding element (equivalence class) in $\SQ(\X)$.
We will take advantage of the unimodularity of $(\tree,X,\o)$ shown in Proposition \ref{pr:unimodular}, by applying the mass-transport principle with
\[
F(G,x,\o,o) := g(x_{\o},x_o)h(x_o,x_{\o},\lan x_{N_o(G)}\ran)1_{\{o \in N_{\o}(G)\}}/|N_{\o}(G)|.
\]
Note that $F$ is well defined on $\G_{**}[\C_t]$ because it is invariant under isomorphisms of $(G,x,\o,o)$. 
We recall also that $\{v \in \tree\}$ is measurable with respect to $X_v$ for each $v \in \V$, as explained in Remark \ref{rem-treerecovery}, which in particular implies that $\{1 \in \tree\}$ and $|N_{\o}(\tree)|$ are $\langle X_{N_{\o}(\tree)}\rangle$-measurable, and $|N_k(\tree)|$ is $\lan X_{N_k(\tree)} \ran$-measurable.
The following calculation will use Lemma \ref{le:GW-exchangeability} and the aforementioned measurability properties in the first and last equality, 
unimodularity as in \eqref{def:unimodularity} with $F$ as above 
in the third equality,
and the fact that $\o \in N_v (\tree)$ if and only if $v \in N_{\o}(\tree)$ in the fourth equality (recalling also our convention that $\frac{1}{|N_{\o}(\tree)|}\sum_{k \in N_{\o}(\tree)}=0$ when $N_{\o}(\tree)=\emptyset$):
\begin{align}
\E&\left[g(X_{\o},X_1)h(X_1,X_{\o},\lan X_{N_1(\tree)}\ran)1_{\{1 \in \tree\}} \right]  \nonumber \\
	&\quad = \E\left[\frac{1}{|N_{\o}(\tree)|}\sum_{k \in N_{\o}(\tree)}g(X_{\o},X_k)h(X_k,X_{\o},\lan X_{N_k(\tree)}\ran) \right] \nonumber \\
	&\quad = \E\left[\sum_{v \in \tree}g(X_{\o},X_v)h(X_v,X_{\o},\lan X_{N_v(\tree)}\ran)1_{\{v \in N_{\o}(\tree)\}}\frac{1}{|N_{\o}(\tree)|} \right] \nonumber \\
	&\quad = \E\left[\sum_{v \in \tree}g(X_v,X_{\o})h(X_{\o},X_v,\lan X_{N_{\o}(\tree)}\ran)1_{\{\o \in N_{v}(\tree)\}}\frac{1}{|N_{v}(\tree)|} \right] \nonumber \\
	&\quad = \E\left[\frac{1}{|N_{\o}(\tree)|}\sum_{k \in N_{\o}(\tree)}g(X_k,X_{\o})h(X_{\o},X_k,\lan X_{N_{\o}(\tree)}\ran) \frac{|N_{\o}(\tree)|}{|N_k(\tree)|} \right] \nonumber \\
	&\quad = \E\left[ g(X_1,X_{\o})h(X_{\o},X_1,\lan X_{N_{\o}(\tree)}\ran) \frac{|N_{\o}(\tree)|}{|N_1(\tree)|} 1_{\{1 \in \tree\}}\right].  \label{pf:unimod-key1}
\end{align}

If $\varphi_h : \C_t^2 \to \R$ is defined by
\begin{align*}
\varphi_h(X_{\o},X_1) =  1_{\{1 \in \tree\} }\E\left[\left. \frac{|N_{\o}(\tree)|}{|N_1(\tree)|}h(X_{\o},X_1,\lan X_{N_{\o}(\tree)}\ran) \, \right| \, X_{\o}, X_1\right],
\end{align*}
then \eqref{pf:unimod-key1} can be rewritten as 
\begin{align}
\E\left[g(X_{\o},X_1)h(X_1,X_{\o},\lan X_{N_1(\tree)}\ran)1_{\{1 \in \tree\}} \right]   &= \E\left[ g(X_1,X_{\o})\varphi_h(X_{\o},X_1) 1_{\{1 \in \tree\}}\right]. \label{pf:unimod-key1.5}
\end{align}
Similarly, define $\varphi_1 : \C_t^2 \to \R$ by
\begin{align*}
\varphi_1(X_{\o},X_1) =   1_{\{1 \in \tree\}} \E\left[\left. \frac{|N_{\o}(\tree)|}{|N_1(\tree)|} \, \right| \, X_{\o}, X_1\right].
\end{align*}
Apply the identity \eqref{pf:unimod-key1.5}, with $h$ replaced by the constant function $1$ and with $g(x_{\o},x_1)$ replaced by $g(x_1,x_{\o})\varphi_h(x_{\o},x_1)$, to obtain 
\begin{align*}
\E&\left[ g(X_1,X_{\o})\varphi_h(X_{\o},X_1) 1_{\{1 \in \tree\}}\right] = \E\left[ g(X_{\o},X_1)\varphi_h(X_1,X_{\o})\varphi_1(X_{\o},X_1) 1_{\{1 \in \tree\}}\right].
\end{align*}
Substitution of this identity into the right-hand side of \eqref{pf:unimod-key1.5} yields 
\begin{align}
\E&\left[g(X_{\o},X_1)h(X_1,X_{\o},\lan X_{N_1(\tree)}\ran)1_{\{1 \in \tree\}} \right] = \E\left[ g(X_{\o},X_1)\varphi_h(X_1,X_{\o})\varphi_1(X_{\o},X_1) 1_{\{1 \in \tree\}}\right]. \label{pf:unimod-key2}
\end{align}
The fact that this holds for any $g$ implies that, a.s.\ on $\{1 \in \tree\}$,
\begin{align*}
\E\left[\left. h(X_1,X_{\o},\lan X_{N_1(\tree)}\ran) \, \right| \, X_{\o}, X_1\right] = \varphi_h(X_1,X_{\o})\varphi_1(X_{\o},X_1). 
\end{align*}
On the other hand, applying \eqref{pf:unimod-key2} with $h$ replaced by the constant function $1$, we deduce that $\varphi_1(X_1,X_{\o})\varphi_1(X_{\o},X_1) = 1$ a.s.\ on $\{1 \in \tree\}$, and so
\begin{align*}
\E\left[\left. h(X_1,X_{\o},\lan X_{N_1(\tree)}\ran) \, \right| \, X_{\o}, X_1\right] = \frac{\varphi_h(X_1,X_{\o})}{\varphi_1(X_1,X_{\o})}.
\end{align*}
Now recalling the definition of $\Xi_t$ given in the statement of Proposition \ref{pr:invariance-GW}, 
(still omitting $[t]$ from the notation), it follows that 
\[
\Xi_t(X_{\o},X_1) = 1_{\{1 \in \tree\}}\frac{\varphi_h(X_{\o},X_1)}{\varphi_1(X_{\o},X_1)}. 
\]
Thus, the last two displays 
 establish  \eqref{def:invariance-GW1} with $k=1$. In light of the symmetry provided by Proposition \ref{pr:properties-GW}(ii), this is enough to complete the proof. \hfill\qedsymbol

\vspace{.1in}
\textbf{Acknowledgments:} We would like to thank the reviewer for feedback that improved the exposition of the paper.

\appendix

\section{A projection theorem}
\label{sec-mim}

Here we state and prove a result, used crucially in deriving the local equation, which can be seen as a \emph{projection} or \emph{mimicking} theorem for It\^o processes.
Theorem \ref{th:brunickshreve} below seems to be reasonably well known, particularly in filtering theory, appearing (in various different forms) for instance in \cite[Theorem 7.17]{liptser-shiryaev}, \cite[Corollary 3.11]{brunick2013mimicking}, and \cite[Section VI.8]{rogers-williams} but we give a short and mostly self-contained proof.
Theorem \ref{th:brunickshreve} can be seen also as a path-dependent counterpart of the famous mimicking theorem of Gy\"{o}ngy \cite{gyongy1986mimicking}.

We begin with a technical lemma to clear up any concerns about the existence of suitable versions of conditional expectations, of the sort that appear in the definitions of $\gamma_t$ in \eqref{def:gamma-kappareg} and \eqref{def:gamma}.
As usual, write $\C=C(\R_+;\R^d)$ and $\C_t=C([0,t];\R^d)$ for the spaces of $\R^d$-valued paths, for $t > 0$, and $x[t]$ for the path up to time $t$ of any $x \in \C$. Recall that we call a function $f$ from $\R_+ \times \C$ to a measurable space $S$ \emph{progressively measurable} if it is jointly measurable and satisfies $f(t,x)=f(t,y)$ whenever $t \ge 0$ and $x,y \in \C$ satisfy $x[t]=y[t]$.

\begin{lemma} \label{le:optprojection}
Suppose $\Gamma=(\Gamma(t))_{t \ge 0}$ and $Y=(Y(t))_{t \ge 0}$ are stochastic processes with values in $\R^k$ and $\R^d$, respectively.  Suppose $Y$ is continuous, and $\E[\int_0^T|\Gamma(t)|dt] < \infty$ for each $T > 0$. Then there exists a progressively measurable function $\gamma : \R_+ \times \C \to \R^k$ such that
\begin{align*}
\gamma(t,Y) = \E[\Gamma(t)\,|\,Y[t]], \quad \text{a.s., for a.e. } t \ge 0.
\end{align*}
\end{lemma}
\begin{proof}
Apply \cite[Proposition 5.1]{brunick2013mimicking}, taking the Polish-space-valued process $Z_t$ therein to be the $\C$-valued process $Y[t]$, to find a Borel measurable function $\widehat\gamma : \R_+ \times \C \to \R^k$ such that
\begin{align*}
\widehat\gamma(t,Y[t]) = \E[\Gamma(t)\,|\,Y[t]], \quad \text{a.s., for a.e. } t \ge 0.
\end{align*}
Then set $\gamma(t,x)=\widehat \gamma(t,x[t])$ for $(t,x) \in \R_+ \times \C$.
\end{proof}

\begin{theorem} \label{th:brunickshreve}
Let $(\Omega,\F,\FF,\PP)$ be a filtered probability space supporting an $\FF$-Brownian motion $W$ of dimension $m$ as well as a continuous $\FF$-adapted process $X$ of dimension $d$ such that $X$ admits the differential
\begin{align*}
dX(t) &= b(t)dt + \sigma(t)dW(t),
\end{align*}
where $b$ and $\sigma$ are $\FF$-progressively processes taking values in $\R^d$, and $\R^{d \times m}$, respectively,
with 
\begin{align}
\E\left[\int_0^t\left(|b(s)| + \mathrm{Tr}[\sigma\sigma^\top(s)]\right) ds\right] < \infty, \quad \text{ for } t > 0. \label{ap:sigmasquareintegrable}
\end{align}
Let $\widetilde{b} : \R_+ \times \C \to \R^d$ and $\widetilde{\sigma} : \R_+ \times \C \to \R^{d\times d}$  be any progressively measurable functions satisfying
\begin{align*}
  \widetilde{b}(t,X[t]) &= \E\big[b(t) \, | \, X[t]\big], \qquad  \widetilde{\sigma}\widetilde{\sigma}^\top(t,X[t]) = \E\big[\sigma \sigma^\top (t) \, | \, X[t]\big], \quad \text{a.s., for a.e. } t \ge 0.
\end{align*} 
Let $\FF^X=(\F^X_t)_{t \ge 0}$ denote the filtration generated by $X$, defined by $\F^X_t=\sigma(X[t])$.
Then there exists an extension $(\check \Omega, \check \F, \check \FF, \check \PP)$ of the probability space $(\Omega,\F, \FF^X, \PP)$ 
supporting a standard $d$-dimensional $\check\FF$-Brownian motion $\widetilde{W}$ such that 
\begin{align*} 
d X(t) = \widetilde{b}(t,X)dt + \widetilde{\sigma}(t,X)d\widetilde{W}(t), \quad t \geq 0.
\end{align*}
\end{theorem}
\begin{proof}
Let $C^\infty_c(\R^d)$ denote the set of smooth functions on $\R^d$ with compact support. Write $\nabla$ and $\nabla^2$ for the gradient and Hessian operators, respectively. By  It\^o's formula and the condition  \eqref{ap:sigmasquareintegrable}, for each $\varphi \in C^\infty_c(\R^d)$ the process
\[
\varphi(X(t)) - \int_0^t\left(b(u) \cdot \nabla\varphi(X(u)) + \frac12\mathrm{Tr}[\sigma\sigma^\top(u)\nabla^2\varphi(X(u))]\right)du
\]
is a $\FF$-martingale. 
In particular, if $t > s$, and if $Z$ is any bounded $\F_s$-measurable random variable then
\begin{align*}
0 &= \E\left[Z\left(\varphi(X(t)) - \varphi(X(s)) - \int_s^t\left(b(u) \cdot \nabla\varphi(X(u)) + \frac12\mathrm{Tr}[\sigma\sigma^\top(u)\nabla^2\varphi(X(u))]\right)du \right)\right].
\end{align*}
Now, If $Z$ is measurable with respect to  $\F^X_s \subset \F_s$, then we may use Fubini's theorem
and the tower property of conditional expectations to obtain 
\begin{align*}
0 &= \E\left[Z\left(\varphi(X(t)) - \varphi(X(s)) - \int_s^t\left(\widetilde{b}(u,X) \cdot \nabla\varphi(X(u)) + \frac12\mathrm{Tr}[\widetilde{\sigma}\widetilde{\sigma}^\top(u,X(u))\nabla^2\varphi(X(u))]\right)du \right)\right].
\end{align*}
This shows that the process 
\[
\varphi(X(t)) - \int_0^t\left(\widetilde{b}(u,X) \cdot \nabla\varphi(X(u)) + \frac12\mathrm{Tr}[
 \widetilde{\sigma}\widetilde{\sigma}^\top(u,X)\nabla^2\varphi(X(u))]\right)du
\]
is a $\FF^X$-martingale, for every $\varphi \in C^\infty_c(\R^d)$.

The claim now follows from the usual construction of weak solutions from solutions to  martingale problems (e.g., using the arguments
 in Proposition 5.4.6 and Theorem 3.4.2 of \cite{karatzas-shreve} or  \cite[Theorem (20.1), p.\ 160]{rogers-williams}).
\end{proof}

\section{Forms of Girsanov's theorem}

We develop here two simple forms of Girsanov's theorem tailored to the needs of proofs of results in this paper. No aspects of these results should come as a surprise to specialists, but we were unable to locate a reference that covered our precise requirements, which fall beyond the scope of the standard Novikov condition.
Our drift $b$ in Assumption \ref{assumption:A} has linear growth, and thus, at least for the first lemma below, fairly standard results could cover some of our needs, such as \cite[Corollary 3.5.16]{karatzas-shreve} or \cite[Theorem 7.7]{liptser-shiryaev}. But those results, strictly speaking, do not allow a general diffusion coefficient $\sigma$. The result \cite[Theorem 7.7]{liptser-shiryaev} is extended in \cite[Section 7.6]{liptser-shiryaev} but still requires Lipschitz coefficients, which is not good enough for us because of the $\gamma_t$ term in the local equation \eqref{statements:localequation-regular}, which need not be Lipschitz even when $b$ is. Our second result below, Lemma \ref{le:ap:girsanov2}, is not directly covered by the aforementioned results either, because it involves an infinite-dimensional SDE system, though we only consider a change in drift for a finite number of coordinates.
In any case, we give simple proofs of our two results using an elegant recent criterion of \cite{blanchet-ruf}.

\begin{lemma} \label{le:ap:girsanov}
Let $d \in \N$ and $\lambda_0 \in \P(\R^d)$. For $T \in (0,\infty)$, suppose $b : [0,T] \times \C \to \R^d$ and $\sigma : [0,T] \times \C \to \R^{d \times d}$ are progressively measurable. Assume $\sigma(t,x)$ is invertible for each $(t,x)$ and that $\sigma$ and $\sigma^{-1}$ are uniformly bounded.
For $i=1,2$, suppose $(\Omega^i,\F^i,\FF^i = \{\F^i_t\}_{t \geq 0},\PP^i)$ is a filtered probability space supporting a $d$-dimensional $\FF^i$-Brownian motion $W^i$ and continuous $d$-dimensional $\FF^i$-adapted process $X^i$, which satisfy
for $t \in [0,T]$,
\begin{align}
  \label{eq-SDE1}
  dX^1(t) &= b(t,X^1)dt + \sigma(t,X^1)dW^1(t), \quad X^1(0) \sim \lambda_0, \\
  \nonumber 
dX^2(t) &= \sigma(t,X^2)dW^2(t), \qquad\qquad\qquad X^2(0) \sim \lambda_0.
\end{align}
Assume the latter SDE is unique in law, and 
	that
	\begin{equation}
	\PP^i\left(\int_0^T |b(t,X^i)|^2dt < \infty\right)  = 1, \quad i=1,2. \label{ap:asmp:girsanov}
	\end{equation}
	Then 
        $\L(X^1[T])$ and $\L(X^2[T])$ are equivalent, and for $x \in \C_T$,  
	\begin{equation}
	\frac{d\L(X^1[T])}{d\L(X^2[T])}(x) = \exp\left(\int_0^T (\sigma\sigma^\top)^{-1}b(t,x) \cdot dx(t) - \frac12 \int_0^T | \sigma^{-1}b(t,x)|^2dt \right). \label{ap:def:girsanov}
	\end{equation}
\end{lemma}

\begin{remark} \label{re:ap:girsanov}
If $t \mapsto b(t,x)$ is continuous for each $x$, then $\int_0^T|b(t,x)|^2dt \le T \sup_{t \in [0,T]}|b(t,x)|^2 < \infty$ for each $x$, and the key assumption \eqref{ap:asmp:girsanov} in Lemma \ref{le:ap:girsanov} holds automatically.
\end{remark}

\begin{proof}[Proof of Lemma \ref{le:ap:girsanov}]
If $b$ is uniformly bounded, then uniqueness in law of the SDE for $X^1$ and  \eqref{ap:def:girsanov} are completely standard, following from Girsanov's theorem. Now, fix $T \in (0, \infty)$ and assume more generally that $\PP(\int_0^T |b(t,X^1)|^2dt < \infty) = \PP(\int_0^T |b(t,X^2)|^2dt < \infty) = 1$. Define $\tau_n : \C \to [0,T] \cup \{\infty\}$ and $b_n : [0,T] \times \C \to \R^d$ by
\begin{align*}
b_n(t,x) := 1_{\{t \le \tau_n(x)\}}b(t,x), \quad \tau_n(x) := \inf\Big\{t \in [0,T] : \int_0^t |b(s,x)|^2ds \ge n\Big\}.
\end{align*}
Abbreviate $P^2=\L(X^2[T])$.  
Now, define   $R : [0,T] \times \C \to \R_+$ by 
\begin{align*}
R(t,x) &:= \exp\left(\int_0^t (\sigma\sigma^\top)^{-1}b(s,x) \cdot dx(s) - \frac12 \int_0^t | \sigma^{-1}b(s,x)|^2ds \right).
\end{align*} 
Note that the uniform boundedness of $\sigma^{-1}$ and the bound \eqref{ap:asmp:girsanov}
  ensure that  
  $(R(t, \cdot))_{t \in [0,T]}$ is well defined $P^2$-a.e.
Moreover,  the uniform boundedness of $\sigma^{-1}$ and  the definition of $b_n$ guarantee that 
$\int_0^T |\sigma^{-1} b_n(t,x)|^2 dt = \int_0^{T \wedge \tau_n(x)}  |\sigma^{-1} b(t,x)|^2 ds \le n$ for all $x \in \C$,
and thus Novikov's condition 
is satisfied. Hence, $(R(t \wedge \tau_n(X^2), X^2), {\mathcal F}^2_t)_{t \in [0,T]}$
is a  $\PP^2$-martingale for each $n$ \cite[Corollary 3.5.13]{karatzas-shreve}. 
Thus, by Girsanov's theorem (see, e.g., \cite[Theorem 3.5.1]{karatzas-shreve}), the SDE 
\[   dX^{1,n}(t) = b_n(t,X^{1,n}) dt + \sigma(t,X^{1,n}) dW(t), \quad X^{1,n}(0) \sim \lambda_0, \]
is unique in law, with its law $P^{1,n}$ satisfying $P^{1,n} \ll P^2$, 
where 
  \[ \frac{dP^{1,n}}{dP^2}(x) := R(T \wedge \tau_n(x), x) = 
  \exp\left(\int_0^T (\sigma\sigma^\top)^{-1}b_n(t,x) \cdot dx(t) - \frac12 \int_0^T | \sigma^{-1}b_n(t,x)|^2dt \right), 
  \]
for $P^2$-almost every    $x \in \C.$    
Assume $X^{1,n}$ is constructed on a probability space $(\Omega^{1,n},\F^{1,n},\FF^{1,n},\PP^{1,n})$.

We will now apply the criterion of \cite[Corollary 2.1]{blanchet-ruf} to prove that under $P^2$, the process $(R(t,\cdot), {\mathcal F}_t^2)_{t \in [0,T]}$
is not only a local martingale  but is in fact a true martingale. To this end, note that the assumption $\PP^2(\int_0^T |b(t,X^2)|^2dt < \infty) = 1$ from \eqref{ap:asmp:girsanov} and the uniform boundedness of $\sigma$ and $\sigma^{-1}$ ensure that $\tau_n(X^2)\to \infty$ and $R(t \wedge \tau_n(X^2),X^2) \to R(t,X^2)$ a.s.\ as $n\to\infty$.  Now, for each $n \in \N$ and $t \in [0,T]$,  define $Q^t_n \ll P^2$ by
\[ \frac{dQ^t_n}{dP^2}(x)=R(t \wedge \tau_n(x),x), \quad  x  \in \C. \]
Then \cite[Corollary 2.1]{blanchet-ruf} states that $(R(t, \cdot))_{t \in [0,T]}$ is a $P^2$-martingale if and only if
\begin{equation}
  \label{eq-criterion}
  \lim_{n\to\infty}Q^t_n(\tau_n \le t) = 0, \quad \mbox{  for each } t \in [0,T].
  \end{equation}
But the latter follows from the assumption $\PP^1(\int_0^T |b(t,X^1)|^2dt < \infty) = 1$ imposed in \eqref{ap:asmp:girsanov}, since recalling $P^2 = \PP^2 \circ (X^2)^{-1}$ and $P^{1,n} = \PP^{1,n} \circ (X^{1,n})^{-1}$ and
letting $\E^2$ and $\E^{1,n}$ denote expectation under $\PP^2$ and
$\PP^{1,n}$, respectively,
 we have \begin{align*}
Q^t_n(\tau_n  \le t) = \E^2[R(t \wedge \tau_n(X^2),X^2)1_{\{\tau_n(X^2) \le t\}}] &= \E^2[R(T \wedge \tau_n(X^2),X^2)1_{\{\tau_n(X^2) \le t\}}] \\
&= \PP^{1,n}(\tau_n(X^{1,n}) \le t) \\
& = \PP^1(\tau_n(X^1) \le t) \\
&= \PP^1\left( \int_0^t|b(s,X^1)|^2ds \ge n\right),  
\end{align*}
 where the penultimate step used the fact that $(X^1_{t \wedge \tau_n(X^1)})_{t \in [0,T]}$
 satisfies the SDE \eqref{eq-SDE1} with $b$ replaced by $b_n$ and thus, by uniqueness in law
 of the latter SDE, the law of 
 $(X^1_{t \wedge \tau_n(X^1)})_{t \in [0,T]}$ under $\PP^1$ coincides
 with that of $(X^{1,n}_{t \wedge \tau_n(X^{1,n})})_{t \in [0,T]}$  under $\PP^{1,n}$.
 Since  the right-hand side of the last display  vanishes as $n \to \infty$ due to \eqref{ap:asmp:girsanov}, this proves \eqref{eq-criterion}. 

Hence, under $P^2$, we have shown that $R$ is a martingale on a finite time horizon, and thus a uniformly integrable martingale on that time horizon. Since 
$dP^{1,n}/dP^2=R(T\wedge \tau_n(\cdot),\cdot)$ for each $n$, we deduce easily that $dP^1/dP^2=R(T,\cdot)$. Since $R(T,\cdot) > 0$, we deduce that $P^1$ and $P^2$ are equivalent.
\end{proof}

Recalling the definition of relative entropy functional $H$ from \eqref{def-relentropy},
  we record the following well-known relative entropy identity as a corollary:

\begin{corollary} \label{co:entropyestimate}
Let $d \in \N$ and $\lambda_0 \in \P(\R^d)$. Suppose $b^1,b^2 : [0,T] \times \C \to \R^d$ and $\sigma : [0,T] \times \C \to \R^{d \times d}$ are progressively measurable and bounded. Assume $\sigma(t,x)$ is invertible for each $(t,x)$ and that $\sigma^{-1}$ is uniformly bounded.
For $i=1,2$, suppose $(\Omega^i,\F^i,\FF^i ,\PP^i)$ is a filtered probability space supporting a $d$-dimensional $\FF^i$-Brownian motion $W^i$ and continuous $d$-dimensional $\FF^i$-adapted process $X^i$ satisfying
\begin{align*}
dX^i(t) &= b^i(t,X^i)dt + \sigma(t,X^i)dW^i(t), \quad X^i(0) \sim \lambda_0.
\end{align*}
Assume the driftless SDE
\begin{align*}
dX(t)=\sigma(t,X)dW(t), \quad X(0)\sim \lambda_0
\end{align*}
is unique in law.
Then the following relative entropy identity holds:
	\begin{equation*}
		H(\L(X^1[T]) \,|\, \L(X^2[T])) = \frac12 \E^{\PP^1} \left[ \int_0^T |\sigma^{-1}b^1(t,X^1)-\sigma^{-1}b^2(t,X^1)|^2 \,dt \right].
	\end{equation*}
\end{corollary}
\begin{proof}
Abbreviate $P^i=\L(X^i[T])$ for $i=1, 2$. 
The boundedness of $b^i$ ensures that \eqref{ap:asmp:girsanov} holds trivially. We may therefore  apply Lemma \ref{le:ap:girsanov} twice to get
\begin{equation*}
\frac{dP^1}{dP^2}(x) = \exp\left(\int_0^T (\sigma\sigma^\top)^{-1}(b^1-b^2)(t,x) \cdot dx(t) + \frac12 \int_0^T (|\sigma^{-1}b^2(t,x)|^2 - |\sigma^{-1}b^1(t,x)|^2) \,dt \right).
\end{equation*}
Hence, it follows that 
\begin{align*}
H&(P^1 | P^2) \\
	&= \E^{\PP^1} \left[ \int_0^T (\sigma\sigma^\top)^{-1}(b^1-b^2)(t,X^1) \cdot dX^1(t) + \frac12 \int_0^T (|\sigma^{-1}b^2(t,X^1)|^2 - |\sigma^{-1}b^1(t,X^1)|^2) \,dt \right] \\
	& = \frac12 \E^{\PP^1} \left[ \int_0^T |\sigma^{-1}b^1(t,X^1)-\sigma^{-1}b^2(t,X^1)|^2 \,dt \right].
\end{align*}
This completes the proof. 
\end{proof}

Lastly, we prove an infinite-dimensional result similar to Lemma \ref{le:ap:girsanov}, tailor-made for its use in the proof of Lemma \ref{le:infinitegraphlimit-GW}.

\begin{lemma} \label{le:ap:girsanov2}
Let $d \in \N$, and let $V$ be a countable set. Let $\lambda_0 \in \P((\R^d)^V)$.
Suppose $b^1_v,b^2_v : [0,T] \times \C^V \to \R^d$ for $v \in V$ and $\sigma : [0,T] \times \C \to \R^{d \times d}$ are progressively measurable.
Assume $\sigma(t,x)$ is invertible for each $(t,x)$ and that $\sigma$ and $\sigma^{-1}$ are uniformly bounded.
For $i=1,2$, suppose $(\Omega^i,\F^i,\FF^i = \{\F_t^i\}_{t \geq 0},\PP^i)$ is a filtered probability space supporting independent $d$-dimensional $\FF^i$-Brownian motions $(W^i_v)_{v \in V}$ as well as continuous $d$-dimensional $\FF^i$-adapted processes $(X_v^i)_{v \in V}$ satisfying
\begin{align*}
dX_v^i(t) &= b^i_v(t,X)dt + \sigma(t,X^i_v)dW^i_v(t), \quad v \in V, \ \  X^i(0)=(X^i_v(0))_{v \in V} \sim \lambda_0,
\end{align*}
where the SDE system for $X^2$ 
is assumed to be unique in law. 
Assume that $b^1_v \equiv b^2_v$ except for at most finitely many $v \in V$, and that for $i =  1, 2$, 
\begin{equation}
\PP^i\left(\int_0^T |b^1_v(t,X^i)-b^2_v(t,X^i)|^2dt < \infty\right) =1, \quad \mbox{ for each } v \in V. 
\label{asmp:ap:girs2}
\end{equation}
Then,   if  $P^i \in \P(\C^V)$ denotes the law of $X^i=(X^i_v)_{v \in V}$ under $\PP^i$ for $i=1,2$, 
then  $P^1$ and $P^2$ are equivalent, and 
\begin{align*}
\frac{dP^1}{dP^2}(X^2) = \exp\left\{\sum_{v \in V}\left( \int_0^T \sigma^{-1}(b^1_v-b^2_v)(t,X^2) \cdot dW_v(t) - \frac12 \int_0^T | \sigma^{-1}(b^1_v-b^2_v)(t,X^2)|^2dt \right)\right\}
\end{align*}
almost surely,
where $\sigma^{-1}(b^1_v-b^2_v)$ denotes the function $[0,T] \times \C^V \ni (t,x) \mapsto \sigma^{-1}(t,x_v)(b^1_v(t,x)-b^2_v(t,x))$ for $v \in V$.
\end{lemma}
\begin{proof}
Let $V_0 := V \setminus \{v \in V : b^1_v \equiv b^2_v\}$, and note that $V_0$ is finite by assumption. If $\sum_{v \in V_0}|b^1_v-b^2_v|^2$ is uniformly bounded, then the claim is a standard application of Girsanov's theorem. For the general case, define $\tau_n : \C^V \to [0,T] \cup \{\infty\}$ and $b^{1,n}_v : [0,T] \times \C^V \to \R^d$ for $v \in V$ by
\begin{align*}
\tau_n(x) &:= \inf\left\{t \in [0,T] : \sum_{v \in V_0}\int_0^t |b^1_v(s,x)-b^2_v(s,x)|^2ds \ge n\right\}, \\
b^{1,n}_v(t,x) &:= 1_{\{t \le \tau_n(x)\}}b^1_v(t,x) + 1_{\{t > \tau_n(x)\}}b^2_v(t,x).
\end{align*}
With these definitions, the remainder of the proof follows that of Lemma \ref{le:ap:girsanov} very closely, so we give fewer details.
Define $R : [0,T] \times \C^V \to \R_+$ by 
\begin{align*}
R(t,x) := \exp\sum_{v \in V}\Bigg(&\int_0^t (\sigma\sigma^\top)^{-1}(b^1_v-b^2_v)(s,x) \cdot \big(dx_v(s) - b^2_v(s,x)ds\big) \\
	&- \frac12 \int_0^t | \sigma^{-1}(b^1_v-b^2_v)(s,x)|^2ds \Bigg),
\end{align*}
which is well-defined for $P^2$-a.e.\ $x=(x_v)_{v \in V} \in \C^V$. Note that $b^{1,n}_v \equiv b^1_v \equiv b^2_v$ for $v \in V \setminus V_0$, so that the summation in the definition of $R$ is actually over the finite set $V_0$. Since also $\sum_{v \in V_0}\int_0^{\tau_n(x)}|\sigma^{-1}(b^1_v-b^2_v)(t,x)|^2dt \le n$ for all $x \in \C^V$ by construction, Novikov's condition ensures that $(R(t \wedge \tau_n(X^2),X^2), {\mathcal F}_t^2)_{t \in [0,T]}$ is a $\PP^2$-martingale, for each $n$. Hence, by Girsanov's theorem and uniqueness in law of the $X^2$ equation, the SDE system
\begin{align*}
dX^{1,n}_v(t) &= b^{1,n}_v(t,X^{1,n})dt + \sigma(t,X^{1,n}_v)dW_v(t), \quad v \in V, \ \ X^{1,n}(0) \sim \lambda_0,
\end{align*}
is unique in law, and its law $P^{1,n}$ satisfies $P^{1,n} \ll P^2$ and, a.s.,
\begin{align*}
\frac{dP^{1,n}}{dP^2}(X^2) = \exp\left\{\sum_{v \in V} \left( \int_0^T \sigma^{-1}(b^{1,n}_v-b^2_v)(t,X^2) \cdot dW^2_v(t) - \frac12 \int_0^T | \sigma^{-1}(b^{1,n}_v-b^2_v)(t,X^2)|^2dt \right)\right\}.
\end{align*}
Assume $X^{1,n}$ is defined on a filtered probability space $(\Omega^{1,n},\F^{1,n},\FF^{1,n},\PP^{1,n})$.

To complete the proof, as in Lemma \ref{le:ap:girsanov}, it suffices to show that the local martingale $R$ is a true martingale. To this end, note that the assumption \eqref{asmp:ap:girs2} and boundedness of $\sigma$ and $\sigma^{-1}$ ensure $\tau_n(X^2) \to \infty$ and $R(t \wedge \tau_n(X^2),X^2) \to R(t,X^2)$ a.s.\ as $n\to\infty$. For each $t \in [0,T]$ and $n \in \N$, we define $Q_n^t \ll P^2$ by $dQ_n^t/dP^2(x) = R(t \wedge \tau_n(x),x)$, $x \in \C^V$. Then, by \cite[Corollary 2.1]{blanchet-ruf}, $R$ is a $P^2$-martingale if and only if $\lim_{n\to\infty}Q_n^t(\tau_n \le t) = 0$ for each $t \in [0,T]$. The latter follows from assumption \eqref{asmp:ap:girs2} by means of a calculation similar to that used in Lemma \ref{le:ap:girsanov}: Since the laws of $(X^{1,n}(t \wedge \tau_n(X^{1,n})))_{t \in [0,T]}$ under $\PP^{1,n}$ and
$(X^1(t \wedge \tau_n(X^1)))_{t \in [0,T]}$ under $\PP^1$ coincide, we have 
\begin{align*}
Q_n^t(\tau_n \le t) = \PP^{1,n}(\tau_n(X^{1,n}) \le t)  & = \PP^1(\tau_n(X^1) \le t) \\
	&= \PP^1\left(\int_0^t |b^1_v(s,X^1)-b^2_v(s,X^1)|^2ds \ge n\right),  
\end{align*}
which converges to zero as $n \rightarrow \infty$ due to \eqref{asmp:ap:girs2}. 
\end{proof}

\section{Proof of Lemma \ref{le:lingrowth}} \label{se:ap:entropy-proof}

Recall that $(X_v(0))_{v \in \V}$ are independent of $\tree$ and are i.i.d.\ and square-integrable by Assumption (\ref{assumption:A}.3), $X = (X^\tree_v)_{v \in \tree}$ satisfies the SDE system \eqref{statements:SDE}.
Using the linear growth of Assumption (\ref{assumption:A}.1) and the boundedness of $\sigma$ of Assumption (\ref{assumption:A}.2), we thus find, for all $t \in [0,T]$, 
\begin{align*}
\E[\|X_v\|^2_{*,t} \,|\, \tree] \le C\Bigg( 1 + \int_0^t\Big(\E[\|X_v\|^2_{*,s} \,|\, \tree] + \frac{1}{|N_v(\tree)|}\sum_{u \in N_v(\tree)}\E[\|X_u\|^2_{*,s} \,|\, \tree]\Big)ds \Bigg),
\end{align*}
where $C < \infty$ is a constant depending only on $T$, $\lambda_0$, and the constants of Assumptions (\ref{assumption:A}.1) and (\ref{assumption:A}.2). (As usual, the average over $N_v(\tree)$ is understood to be zero when $N_v(\tree)=\emptyset$ or $v \notin \tree$.) This implies
\begin{align*}
\sup_{v \in \V}\E[\|X_v\|^2_{*,t} \,|\, \tree] \le 2C\Bigg(1 + \int_0^t \sup_{v \in \V}\E[\|X_v\|^2_{*,s} \,|\, \tree]ds\Bigg).
\end{align*}
The proof of \eqref{def:2ndmomentbound} can be completed using Gronwall's inequality. 

To derive the entropy bounds, fix a finite set $A \subset \V$ and a time horizon $T \in (0,\infty)$. Suppose first that the tree $\tree$ is a.s.\ finite. Define a change of probability measure $\QQ^A$ by the Radon-Nikodym derivative
\begin{align*}
\frac{d\QQ^A}{d\PP} = \EE_T\left(-\sum_{v \in A}\int_0^\cdot \sigma^{-1}b(t,X_v,X_{N_v(\tree)}) \cdot dW_v(t)\right).
\end{align*}
Working conditionally on the (finite) tree, we may apply Girsanov's theorem in the form of Lemma \ref{le:ap:girsanov}, due to Assumption (\ref{assumption:A}.1) and Remark \ref{re:ap:girsanov},
to deduce that this change of measure is well defined (i.e., $d\QQ^A/d\PP$ has mean 1), and the processes
\[
W^A_v(t) := W_v(t) + \int_0^t\sigma^{-1}b(s,X_v,X_{N_v(\tree)})ds, \quad v \in \V, \ \ t\in[0,T],
\]
are independent Brownian motions under $\QQ^A$ by Girsanov's theorem. Thus, under $\QQ^A$, we find that $(X_v)_{v \in A}$ satisfy the driftless SDE
\begin{align*}
dX_v(t) = 1_{\{v \in \tree\}}\sigma(t,X_v)dW_v(t), \quad v \in A. 
\end{align*}
As this SDE is unique in law by Assumption (\ref{assumption:A}.2b), we deduce that
\[
\QQ^A \circ X_A^{-1} = \PP \circ \widehat{X}_A^{-1}, 
\]
where $\widehat{X} = (\widehat{X}^\tree_v)_{v \in \mathbb{V}},$ is the solution to the SDE
  system \eqref{pf:existence-driftless}  and we have assumed (for notational simplicity) that $X$ and $\widehat{X}$ are defined
  on the same probability space $(\Omega, \F, \PP)$.  
By the data processing inequality of relative entropy, we have
\begin{align*}
  H\big(\L(X_A[T])\,|\,\L(\widehat{X}_A[T])\big) &= H\big(\PP \circ X_A[T]^{-1}\,|\,\PP \circ \widehat{X}_A[T]^{-1}\big)
  \\
  &= H\big(\PP \circ X_A[T]^{-1}\,|\,\QQ^A \circ X_A[T]^{-1}\big) \\
	&\le H(\PP \,|\, \QQ^A) \\
	&= \frac12 \E^{\PP}\left[\sum_{v \in A}\int_0^T |\sigma^{-1}b(t,X_v,X_{N_v(\tree)})|^2 dt \right].
\end{align*}
The proof of  \eqref{def:entropybound1} can be completed by 
using the boundedness of $\sigma^{-1}$, the linear growth of $b$, and the result \eqref{def:2ndmomentbound} of the first part
(possibly changing the constant).  
Similarly, to prove \eqref{def:entropybound2}, still in the case of an a.s.\ finite tree $\tree$, we compute
\begin{align*}
  H\big(\L(\widehat{X}_A[T])\,|\,\L(X_A[T])\big) &= H\big(\PP \circ \widehat{X}_A[T]^{-1}\,|\,\PP \circ X_A[T]^{-1}\big) \\
  &= H\big(\QQ^A \circ X_A[T]^{-1}\,|\,\PP \circ X_A[T]^{-1}\big) \\
	&\le H(\QQ^A \,|\, \PP) \\
	&= \frac12 \E^{\QQ^A}\left[\sum_{v \in A}\int_0^T |\sigma^{-1}b(t,X_v,X_{N_v(\tree)})|^2 dt \right].
\end{align*}
The SDE system \eqref{statements:SDE}  under $\QQ^A$ takes the form
\begin{align*}
dX_v(t) &= 1_{\{v \in \tree\}}\Big(b(t,X_v,X_{N_v(\tree)})dt + \sigma(t,X_v)dW^A_v(t)\Big), \quad v \in \V \setminus A, \\
dX_v(t) &= 1_{\{v \in \tree\}}\sigma(t,X_v)dW^A_v(t), \quad v \in A,
\end{align*}
and it is straightforward to argue that the SDE system under $\QQ^A$ enjoys an identical second moment bound as in \eqref{def:2ndmomentbound}. This completes the proof under the additional assumption that $\tree$ is a.s.\ finite.
We prove the case of a general random tree $\tree$ by truncating the tree to the first $n$ generations, $\tree_n :=\tree\cap\V_n$, and deducing from above that the bounds \eqref{def:entropybound1} and \eqref{def:entropybound2} hold when $\tree$ is replaced with $\tree_n$. 
The particle system $(X^{\tree_n}_v)_{v \in \V}$ clearly converges to $(X^{\tree}_v)_{v \in \V}=(X_v)_{v \in \V}$ in law, and the lower semicontinuity of relative entropy 
lets us take limits as $n \rightarrow \infty$ on both sides of \eqref{def:entropybound1} and \eqref{def:entropybound2} to show that these bounds hold for $\tree$. \qed

\bibliographystyle{amsplain}

\begin{bibdiv}
\begin{biblist}

\bib{aldous-lyons}{article}{
      author={Aldous, D.},
      author={Lyons, R.},
       title={Processes on unimodular random networks},
        date={2007},
     journal={Electronic Journal of Probability},
      volume={12},
       pages={1454\ndash 1508},
        note={paper no. 54},
}

\bib{BhaBudWu18}{article}{
      author={Bhamidi, S.},
      author={Budhiraja, A.},
      author={Wu, R.},
       title={Weakly interacting particle systems on inhomogeneous random
  graphs},
        date={2019},
     journal={Stoch. Proc. Appl.},
      volume={129},
      number={6},
       pages={2174\ndash 2206},
}

\bib{blanchet-ruf}{article}{
      author={Blanchet, J.},
      author={Ruf, J.},
       title={A weak convergence criterion for constructing changes of
  measure},
        date={2016},
     journal={Stochastic Models},
      volume={32},
      number={2},
       pages={233\ndash 252},
}

\bib{Bordenave2016}{unpublished}{
      author={Bordenave, C.},
       title={Lecture notes on random graphs and probabilistic combinatorial
  optimization},
        date={2016},
         url={https://www.math.univ-toulouse.fr/~bordenave/coursRG.pdf},
}

\bib{brunick2013mimicking}{article}{
      author={Brunick, G.},
      author={Shreve, S.},
       title={Mimicking an {I}t{\^o} process by a solution of a stochastic
  differential equation},
        date={2013},
     journal={The Annals of Applied Probability},
      volume={23},
      number={4},
       pages={1584\ndash 1628},
}

\bib{CopDieGia18}{article}{
      author={Coppini, F.},
      author={Dietert, H.},
      author={Giacomin, G.},
       title={A law of large numbers and large deviations for interacting
  diffusions on {E}rd\"{o}s-{R}\'{e}nyi graphs},
        date={2020},
     journal={Stochastics and Dynamics},
      volume={20},
      number={2},
        note={DOI 10.1142/S0219493720500100},
}

\bib{CrimaldiPratelli2005}{article}{
      author={Crimaldi, I.},
      author={Pratelli, L.},
       title={Convergence results for conditional expectations},
        date={2005},
     journal={Bernoulli},
      volume={11},
      number={4},
       pages={737\ndash 745},
}

\bib{CsiKor11}{book}{
      author={Csisz\'{a}r, I.},
      author={K\"{o}rner, J.},
       title={Information theory: Coding theorems for discrete memoryless
  systems},
   publisher={Cambridge University Press},
        date={2011},
}

\bib{DelGiaLuc16}{article}{
      author={Delattre, S.},
      author={Giacomin, G.},
      author={Lu\c{c}on, E.},
       title={A note on dynamical models on random graphs and {F}okker-{P}lanck
  equations},
        date={2016},
     journal={J. Stat. Phys},
      volume={165},
       pages={785\ndash 798},
}

\bib{dembo-montanari}{article}{
      author={Dembo, A.},
      author={Montanari, A.},
       title={Gibbs measures and phase transitions on sparse random graphs},
        date={2010},
     journal={Brazilian Journal of Probability and Statistics},
      volume={24},
      number={2},
       pages={137\ndash 211},
}

\bib{DetFouIch18}{article}{
      author={Detering, N.},
      author={Fouque, J.-P.},
      author={Ichiba, T.},
       title={Directed chain stochastic differential equations},
        date={2020},
     journal={Stochastic Processes and their Applications},
      volume={130},
      number={4},
       pages={2519\ndash 2551},
}

\bib{georgii2011gibbs}{book}{
      author={Georgii, H.-O.},
       title={Gibbs measures and phase transitions},
   publisher={Walter de Gruyter},
        date={2011},
      volume={9},
}

\bib{gyongy1986mimicking}{article}{
      author={Gy{\"o}ngy, I.},
       title={Mimicking the one-dimensional marginal distributions of processes
  having an {I}t{\^o} differential},
        date={1986},
     journal={Probability theory and related fields},
      volume={71},
      number={4},
       pages={501\ndash 516},
}

\bib{harris-book}{book}{
      author={Harris, T.E.},
       title={The theory of branching processes},
   publisher={Courier Corporation},
        date={2002},
}

\bib{israel1986some}{article}{
      author={Israel, R.B.},
       title={Some examples concerning the global {M}arkov property},
        date={1986},
     journal={Communications in mathematical physics},
      volume={105},
      number={4},
       pages={669\ndash 673},
}

\bib{karatzas-shreve}{book}{
      author={Karatzas, I.},
      author={Shreve, S.E.},
       title={Brownian motion and stochastic calculus},
      series={Graduate Texts in Mathematics},
   publisher={Springer New York},
        date={1991},
}

\bib{Kes85}{article}{
      author={Kessler, C.},
       title={Examples of extremal lattice fields without the global markov
  property},
        date={1985},
     journal={Publ. RIMS, Kyoto Univ.},
      volume={21},
       pages={877\ndash 888},
}

\bib{Kol10}{book}{
      author={Kolokoltsov, N.},
       title={Nonlinear markov processes and kinetic equations},
      series={Cambridge Tracts in Mathematics},
   publisher={Cambridge University Press},
        date={2010},
      volume={vol. 182},
}

\bib{KotKur10}{article}{
      author={Kotelenez, P.M.},
      author={Kurtz, T.G.},
       title={Macroscopic limits for stochastic partial differential equations
  of mckean–vlasov type},
        date={2010},
     journal={Probability Theory and Related Fields},
       pages={146\ndash 189},
}

\bib{Kurtz-Protter}{article}{
      author={Kurtz, T.G.},
      author={Protter, P.E.},
       title={Weak limit theorems for stochastic integrals and stochastic
  differential equations},
        date={1991},
     journal={The Annals of Probability},
       pages={1035\ndash 1070},
}

\bib{kurtz1999particle}{article}{
      author={Kurtz, T.G.},
      author={Xiong, J.},
       title={Particle representations for a class of nonlinear {SPDE}s},
        date={1999},
     journal={Stochastic Processes and their Applications},
      volume={83},
      number={1},
       pages={103\ndash 126},
}

\bib{LacRamWu19a}{article}{
      author={Lacker, D.},
      author={Ramanan, K.},
      author={Wu, R.},
       title={Large sparse networks of interacting diffusions},
        date={2019},
     journal={preprint arXiv:1904.02585v1},
}

\bib{LacRamWu20a}{article}{
      author={Lacker, D.},
      author={Ramanan, K.},
      author={Wu, R.},
       title={Local weak convergence for sparse networks of interacting
  processes},
        date={2020},
     journal={preprint arXiv:1904.02585v3},
}

\bib{LacRamWu19b}{article}{
      author={Lacker, D.},
      author={Ramanan, K.},
      author={Wu, R.},
       title={Locally interacting diffusions as {M}arkov random fields on path
  space},
        date={2021},
        ISSN={0304-4149},
     journal={Stochastic Processes and their Applications},
      volume={140},
       pages={81\ndash 114},
  url={https://www.sciencedirect.com/science/article/pii/S0304414921001009},
}

\bib{LacRamWu20c}{unpublished}{
      author={Lacker, D.},
      author={Ramanan, K.},
      author={Wu, R.},
       title={Marginal dynamics of probabilistic cellular automata on trees},
        date={2021},
        note={Work in progress},
}

\bib{lauritzen1996graphical}{book}{
      author={Lauritzen, S.L.},
       title={Graphical models},
   publisher={Clarendon Press},
        date={1996},
      volume={17},
}

\bib{liptser-shiryaev}{book}{
      author={Liptser, R.S.},
      author={Liptser, R.S.},
       title={Statistics of random processes: I. general theory},
   publisher={Springer Science \& Business Media},
        date={2001},
      volume={1},
}

\bib{Luc18quenched}{article}{
      author={Lu\c{c}on, E.},
       title={Quenched asymptotics for interacting diffusions on inhomogeneous
  random graphs},
        date={2020},
     journal={Stochastic Processes and their Applications},
      volume={130},
      number={11},
       pages={6783\ndash 6842},
}

\bib{lyons1995conceptual}{article}{
      author={Lyons, R.},
      author={Pemantle, R.},
      author={Peres, Y.},
       title={Conceptual proofs of {L} log {L} criteria for mean behavior of
  branching processes},
        date={1995},
     journal={The Annals of Probability},
       pages={1125\ndash 1138},
}

\bib{Mck67}{incollection}{
      author={McKean, H.P.},
       title={Propagation of chaos for a class of non-linear parabolic
  equations},
        date={1967},
   booktitle={Stochastic differential equations},
      series={(Lecture Series in Differential Equations, Session 7, Catholic
  Univ.)},
       pages={41\ndash 57},
}

\bib{Med18}{article}{
      author={Medvedev, G.S.},
       title={The continuum limit of the {K}uramoto model on sparse directed
  graphs},
        date={2019},
     journal={Communications in Mathematical Sciences},
      volume={17},
      number={4},
       pages={883\ndash 898},
}

\bib{neveu1986arbres}{article}{
      author={Neveu, J.},
       title={Arbres et processus de {G}alton-{W}atson},
        date={1986},
     journal={Ann. Inst. H. Poincar{\'e} Probab. Statist},
      volume={22},
      number={2},
       pages={199\ndash 207},
}

\bib{OliReiSto19}{article}{
      author={Oliveira, R.I.},
      author={Reis, G.~H.},
      author={Stolerman, L.~M.},
       title={Interacting diffusions on sparse graphs: hydrodynamics from local
  weak limits},
        date={2020},
     journal={Electronic Journal of Probability},
      volume={25},
      number={110},
        note={35 pp.},
}

\bib{OliRei18}{article}{
      author={Oliveira, R.I.},
      author={Reis, G.H.},
       title={Interacting diffusions on random graphs with diverging degrees:
  hydrodynamics and large deviations},
        date={2019},
     journal={Journal of Statistical Physics},
      volume={176},
       pages={1057\ndash 1087},
}

\bib{rogers-williams}{book}{
      author={Rogers, L.C.G.},
      author={Williams, D.},
       title={Diffusions, {M}arkov processes and martingales: {V}olume 2,
  {I}t{\^o} calculus},
   publisher={Cambridge University Press},
        date={2000},
      volume={2},
}

\bib{GerMenSto20}{article}{
      author={S.~Gerchinovitz, P.~M\'{e}nard},
      author={Stoltz, G.},
       title={Fano's inequality for random variables},
        date={2020},
     journal={Statistical Science},
      volume={35},
      number={2},
       pages={178\ndash 201},
}

\bib{SudijonoThesis19}{article}{
      author={Sudijono, T.},
       title={Stationarity and ergodicity of local dynamics of interacting
  {M}arkov chains on large sparse graphs},
        date={2019},
        note={Senior Honors Thesis, Brown University; Advisor: K. Ramanan;
  Mentor: A. Ganguly},
}

\bib{sznitman1991topics}{article}{
      author={Sznitman, A.-S.},
       title={Topics in propagation of chaos},
        date={1991},
     journal={Ecole d'Et{\'e} de Probabilit{\'e}s de Saint-Flour XIX—1989},
       pages={165\ndash 251},
}

\bib{van2009random}{article}{
      author={{van der Hofstad}, R.},
       title={Random graphs and complex networks},
        date={2009},
     journal={Available on http://www.win.tue.nl/~rhofstad/NotesRGCN.pdf},
      volume={11},
}

\bib{van2009randomII}{unpublished}{
      author={{van der Hofstad}, R.},
       title={Random graphs and complex networks, volume 2},
        date={2016},
         url={https://www.win.tue.nl/~rhofstad/NotesRGCNII.pdf},
}

\bib{Wei80}{incollection}{
      author={{von Weizs\"{a}cker}, H.},
       title={A simple example concerning the global markov property of lattice
  random fields},
        date={1980},
   booktitle={8th winter school on abstract analysis},
   publisher={Czechoslovak Academy of Sciences, Praha},
       pages={194\ndash 198},
}

\bib{MitchellThesis18}{article}{
      author={Wortsman, M.},
       title={Systems of interacting particles and efficient approximations for
  large sparse graphs},
        date={2018},
        note={Senior Honors Thesis, Brown University; Advisor: K. Ramanan;
  Mentor: A. Ganguly},
}

\end{biblist}
\end{bibdiv}


\end{document}